\documentclass[a4paper,10pt]{article}
\usepackage[top=1in, bottom=1in, left=1in, right=1in]{geometry}
\usepackage{amsmath,amsfonts,amscd,amssymb,amsthm}
\usepackage{fancyhdr}
\usepackage{enumitem}
\usepackage[percent]{overpic}
\usepackage{tikz}
\usetikzlibrary{shapes.misc}
\usepackage{mathrsfs}
\usepackage{ifthen,tabularx,graphicx,multirow}
\usepackage{xargs}
\usepackage[colorinlistoftodos,prependcaption,textsize=tiny]{todonotes}
\usepackage{tcolorbox}
\usepackage{enumerate}
\usepackage{comment}
\usepackage{booktabs}
\usepackage{varwidth}
\usepackage{graphicx}
\usepackage{epstopdf}
\usepackage{mathtools}
\usepackage{setspace}
\usepackage{palatino} 
\usepackage{natbib}
\usepackage{bm}
\usepackage[bottom,flushmargin,hang,multiple]{footmisc}

\usepackage[pdfencoding=auto, psdextra]{hyperref}
\usepackage{lscape}
\usepackage{afterpage}
\usepackage{multirow}

\usepackage[capitalise,noabbrev]{cleveref}

\usepackage{algorithmic}
\usepackage{algorithm}

\newcommandx{\at}[2][1=]{\todo[linecolor=red,backgroundcolor=red!25,bordercolor=red,#1]{#2}}
\numberwithin{equation}{section}
\newtheorem{theorem}{Theorem}

\newtheorem{proposition}{Proposition}

\numberwithin{theorem}{section}
\numberwithin{lemma}{section}
\numberwithin{corollary}{section}
\numberwithin{proposition}{section}
\numberwithin{definition}{section}
\numberwithin{example}{section}
\usepackage{bm} 
\newcommand{\tol}{{\boldsymbol\tau}}
\newcommand {\inp}[2]{\mbox{$\langle #1, #2\rangle$}}
\newcommand{\Koop}{\mathcal{K}}
\newcommand{\KoopM}{\mathbb{K}}
\newcommand{\pp}{[-\pi,\pi]_{\mathrm{per}}}
\newcommand{\C}{\ensuremath{\mathbb{C}}} 
\newcommand{\R}{\ensuremath{\mathbb{R}}}

\def\roff  {\mbox{\boldmath$\varepsilon$}}

\newcommand*{\dd}{{\,\mathrm{d}}}
\newcommand{\Av}{\mathbf{A}}
\newcommand{\Bv}{\mathbf{B}}
\newcommand{\Cv}{\mathbf{C}}

\newcommand{\Gv}{\mathbf{G}}

\newcommand{\Iv}{\mathbf{I}}

\newcommand{\Kv}{\mathbf{K}}
\newcommand{\Lv}{\mathbf{L}}

\newcommand{\Pv}{\mathbf{P}}
\newcommand{\Qv}{\mathbf{Q}}
\newcommand{\Rv}{\mathbf{R}}

\newcommand{\Uv}{\mathbf{U}}
\newcommand{\Vv}{\mathbf{V}}
\newcommand{\Wv}{\mathbf{W}}
\newcommand{\Xv}{\mathbf{X}}
\newcommand{\Yv}{\mathbf{Y}}

\newcommand{\gv}{\mathbf{g}}
\newcommand{\cv}{\mathbf{c}}
\newcommand{\zv}{\mathbf{z}}
\newcommand{\sv}{\mathbf{s}}
\newcommand{\vv}{\mathbf{v}}
\newcommand{\wv}{\mathbf{w}}
\newcommand{\Psiv}{\mathbf{\Psi}}

\newcommand{\pvm}{\mathcal{E}}

\providecommand{\keywords}[1]{\textbf{\textit{Keywords:}} #1}

\title{An Introductory Guide to Koopman Learning}
\author{Matthew Colbrook \thanks{Department of Applied Mathematics \& Theoretical Physics, University of Cambridge, Cambridge, UK. (mjc249@cam.ac.uk)} \and Zlatko Drmač\thanks{Department of Mathematics, University of Zagreb, Zagreb, Croatia. (drmac@math.hr)} \and Andrew Horning\thanks{Department of Mathematical Sciences, Rensselaer Polytechnic Institute, Troy, New York, USA. (hornia3@rpi.edu)}}

\begin{document}

\maketitle

\begin{abstract}
Koopman operators provide a linear framework for data-driven analyses of nonlinear dynamical systems, but their infinite-dimensional nature presents major computational challenges. In this article, we offer an introductory guide to Koopman learning, emphasizing rigorously convergent data-driven methods for forecasting and spectral analysis. We provide a unified account of error control via residuals in both finite- and infinite-dimensional settings, an elementary proof of convergence for generalized Laplace analysis---a variant of filtered power iteration that works for operators with continuous spectra and no spectral gaps---and review state-of-the-art approaches for computing continuous spectra and spectral measures. The goal is to provide both newcomers and experts with a clear, structured overview of reliable data-driven techniques for Koopman spectral analysis.
\end{abstract}

\keywords{Data-driven dynamics, Koopman operator, Dynamic Mode Decomposition, generalized Laplace analysis, spectral measures}

\footnotesize
\renewcommand{\contentsname}{\small\bfseries Contents}
\setcounter{tocdepth}{2}
\tableofcontents 
\normalsize

\section{Introduction}

In this paper, we give an introductory guide to Koopman learning for studying data-driven spectral problems for discrete-time dynamical systems of the form
\begin{equation}
\label{eq:dyn_sys}
x_{n+1}=F(x_n), \qquad n= 0,1,2,\ldots.
\end{equation}
Here, $x\in\mathcal{X}$ denotes the state of the system and the state space $\mathcal{X}$ is a metric space. The function $F: \mathcal{X} \rightarrow \mathcal{X}$, which governs the evolution of the state, is assumed to be continuous. Frequently, though not always, $\mathcal{X}\subset\mathbb{R}^d$ and $F$ is nonlinear.

For many modern systems, the function $F$ is either unknown or too complex to analyze directly. Instead, we rely on observations in the form of \textit{snapshot data}:
\begin{equation}
\label{eq:snapshot_data}
\left\{\left(x^{(m)},y^{(m)}\right)\right\}_{m=1}^M\subset\mathcal{X}\quad \text{such that}\quad y^{(m)}=F(x^{(m)}),\quad m=1,\ldots,M.
\end{equation}
Such data may come from a single long trajectory or from many shorter experiments or simulations. Our aim is to show how this information can be used to study Koopman operators, which provide a powerful linear framework for analyzing nonlinear dynamics. Namely, one studies functions (called ``observables'') $g:\mathcal{X}\rightarrow\mathbb{C}$ and the Koopman operator acts on $g$, yielding another observable $\Koop g$ defined by
\begin{equation}
\label{koop_def}
[\Koop g](x) = g(F(x)).
\end{equation}
The key features of $\Koop$ are that it is always linear on the space of observables and its spectral decompositions encode information about the state-space dynamics. Numerical approximations of $\Koop$ from snapshot data allow nonlinear forecasting, and spectral computations uncover structure in the underlying dynamics. 

Since Koopman operators act on infinite-dimensional spaces of observables, computations require truncation or finite-dimensional approximation. Our focus is on methods whose truncations converge rigorously to the true properties of the operator. This is an essential requirement, as non-convergent approximations often yield misleading results unrelated to the underlying dynamical system.

The literature on Koopman operators is extensive, and it is easy for a newcomer to get lost. Our aim is not to catalogue every development in Koopman operator computations (which would run into several hundred pages) but to highlight key advances and unifying principles. In particular, this review focuses on:
\begin{itemize}
\item A comprehensive, unified discussion of the role of residuals for error control in both finite- and infinite-dimensional computations;
\item An elementary proof of the convergence of generalized Laplace analysis for computing Koopman modes when the Koopman operator is ``spectral'';
\item An extensive review and comparison of state-of-the-art methods for computing and handling continuous spectra (typically associated with chaotic systems) and spectral measures.
\end{itemize}
These contributions complement existing reviews. The survey by Mezić \cite{mezic2013analysis} and the more recent review by Schmid \cite{schmid2022dynamic} primarily focus on developments in fluid dynamics applications. While Koopman operator theory and DMD techniques were first applied to fluid problems, their broader applicability has now been firmly established across diverse fields. For instance, applications in control theory are discussed by Otto and Rowley \cite{otto2021koopman}. An early overview of ``Applied Koopmanism'' was presented by Budi{\v{s}}i{\'c}, Mohr and Mezi{\'c} \cite{budivsic2012applied}, while Brunton, Budi{\v{s}}i{\'c}, Kaiser and Kutz \cite{brunton2021modern} provided a wide-ranging introduction that emphasizes connections to other areas. Most recently, Colbrook \cite{colbrook2023multiverse} offered a comprehensive review of DMD methods, highlighting their interplay with the spectral properties of Koopman operators and related numerical computations.

\subsection{The setting}

We assume throughout that the metric space \((\mathcal{X}, d_{\mathcal{X}})\) is equipped with a Borel \(\sigma\)-finite measure \(\omega\), and that the Borel sets of \(\mathcal{X}\) are countably generated. This framework allows us to work in the separable Hilbert space \(L^2(\mathcal{X}, \omega)\), which is the most common setting for Koopman operators. The choice of measure \(\omega\) is typically application-dependent: it may correspond to a natural physical measure (e.g., on the Lorenz attractor) or be selected to assign different weights to regions of 
\(\mathcal{X}\).

To ensure that a pointwise definition in \cref{koop_def} is well-defined for an operator acting on equivalence classes of functions in $L^2(\mathcal{X}, \omega)$, we must assume that $F$ is non-singular with respect to $\omega$. The \textit{pushforward measure}, $F\#\omega$, is defined on Borel subsets $S\subset\mathcal{X}$ by
$$
F\#\omega(S)=\omega(F^{-1}(S)),\quad\text{where}\quad F^{-1}(S)=\{x\in\mathcal{X}:F(x)\in S\},
$$
and $F$ is \textit{non-singular} with respect to $\omega$ if $F\#\omega(S)=0$ whenever $\omega(S)=0$.

By \cite{Singh-Compact-QuasiN-1974}, if $(\mathcal{X}, \omega)$ is a non-atomic measure space, $\Koop$ is never compact. Consequently, the classical theory for computing spectra of compact operators does not apply \cite{babuvska1991eigenvalue,osborn1975spectral}. The spectral analysis of non-compact operators is considerably more delicate; see \cite{colbrook3,colbrook2019compute,lewin2009spectral} for a discussion of these challenges. We shall assume throughout this paper that $\Koop$ is bounded but not compact.\footnote{Of course, the Koopman operator depends on the space of observables. For compact Koopman operators on the Hardy space $\mathcal{H}^2(\mathbb{D})$ of analytic functions on the unit disk with square-summable power-series coefficients, see \cite{Shapiro-Comp_ops-1993}.}

\section{Matrix approximations of the operator}
\label{sec:Mat_compression_EDMD}

\subsection{Finite sections}

To approximate the infinite-dimensional operator $\Koop$ with finitely many computational resources and data, the simplest approach is to project onto a sequence of finite-dimensional subspaces. 
One first chooses a \textit{dictionary} $\mathcal{D}=\{\psi_1,\ldots,\psi_{N}\}$, i.e., a list of observables in the space $L^2(\mathcal{X}, \omega)$. The functions $\psi_n$ need not be normalized or orthogonal, but we assume they are linearly independent. Denote the dictionary span by $V_{N}=\mathrm{span}\{\psi_1,\ldots,\psi_{N}\}$ and the orthogonal projection from $L^2(\mathcal{X}, \omega)$ onto $V_{N}$ by $\mathcal{P}_{V_{N}}$. The goal is to construct increasingly accurate approximations as $N\rightarrow\infty$.

The ``finite section'' (or compression) approximation of $\Koop$ is $\Koop_N=\mathcal{P}_{V_{N}}\Koop\mathcal{P}_{V_{N}}^* : V_N\rightarrow V_N$, which can be extended to all of $L^2(\mathcal{X}, \omega)$ as $\Koop_N \mathcal{P}_{V_{N}}$. This approach has a long history outside Koopman theory. In the Koopman setting, it underlies the Extended Dynamic Mode Decomposition (EDMD) \cite{Williams2015}.

We seek a matrix representation $\KoopM_N$ of $\mathcal{P}_{V_{N}}\Koop\mathcal{P}_{V_{N}}^*$ so that
\begin{equation}\label{eq:Ufi(x)}
	[\Koop \psi_i](x)=\psi_i(F(x))=\sum_{j=1}^N [\KoopM_N]_{ji} \psi_j(x) + \rho_i(x),\;\;i=1,\ldots, N, \;\;x\in\mathcal{X},
\end{equation}
where $\rho_i$ is the residual.
The projection $\mathcal{P}_{V_{N}}\Koop\psi_i$ is obtained by choosing the $i$th column $\KoopM_N(:,i)$ so that $\int_{\mathcal{X}} |\rho_i(x)|^2\dd \omega(x)$ is minimized. Equivalently, the residual is orthogonal to $V_N$, which can be expressed as
\begin{equation}\label{eq:ort-resid}
\Koop_N\psi_i - \sum_{j=1}^N [\KoopM_N]_{ji}\psi_j \perp \psi_\ell \Longleftrightarrow 
\sum_{j=1}^N \inp{\psi_j}{\psi_\ell} [\KoopM_N]_{ji} = \inp{\Koop_N\psi_i}{\psi_\ell},\;\;\ell=1,\ldots, N.
\end{equation}
This defines a linear system with Gram matrix $[G_N]_{\ell j}=\inp{\psi_j}{\psi_\ell}$.
This matrix is Hermitian and positive definite, ensuring that $\KoopM_N$ is uniquely determined.

The quality of this approximation is assessed through the convergence of $\Koop_N \mathcal{P}_{V_{N}}$ to $\Koop$. Since $\Koop$ is not compact,
there is no hope of uniform (in norm) convergence $\Koop_N \mathcal{P}_{V_{N}}\longrightarrow\Koop$ as $N\rightarrow\infty$ (recall that a norm-limit of finite-rank operators is compact). Instead, one only obtains strong convergence: for every fixed $g\in L^2(\mathcal{X},\omega)$,
$$
\lim_{N\rightarrow\infty}\Koop_N \mathcal{P}_{V_{N}}g=\Koop g.
$$
Unfortunately, as we discuss below, this does not imply convergence of the eigenvalues of $\KoopM_N$ to the spectrum of $\Koop$. However, we will still be able to obtain useful approximations of spectral properties of $\Koop$ from data-driven approximations of $\KoopM_N$.

\subsection{Data-driven approximations (a.k.a. EDMD)}

In our data-driven setting, direct access to the residual $\rho_i$ or to the inner products in \cref{eq:ort-resid} is infeasible. Instead, we minimize $\rho_i$ over the available snapshot data in \cref{eq:snapshot_data}. For $i=1,\ldots, N$, we minimize the weighted least-squares residual
\begin{equation}\label{eq:LS:Upi}
	\sum_{m=1}^{M}w_m |\rho_i(x^{(m)})|^2= 
	\sum_{m=1}^{M}w_m \bigg| \underbrace{\sum_{j=1}^N [\Kv]_{ji} \psi_j(x^{(m)}) - \psi_i(F(x^{(m)}))}_{\rho_i(x^{(m)})}\bigg|^2,
\end{equation}
where $w_m$ are quadrature weights such that
$$
\sum_{m=1}^{M}w_m |\rho_i(x^{(m)})|^2\approx \int_{\mathcal{X}} |\rho_i(x)|^2\dd \omega(x).
$$
We use the matrix $\Kv$ to distinguish from the matrix $\KoopM_N$ that corresponds to the limit $M\rightarrow\infty$, assuming the quadrature rule converges.

Let $\Wv=\mathrm{diag}(w_1,\ldots,w_{M})$, $\Psiv(x^{(m)})=(\psi_1(x^{(m)})\; \cdots \;\psi_{N}(x^{(m)}))$, $y^{(m)}=F(x^{(m)})$, and define the data matrices
\begin{equation}\label{eq:psidef}
		\Psiv_X=\begin{pmatrix}
			\Psiv(x^{(1)}) \\
			\vdots              \\
			\Psiv(x^{(M)})
		\end{pmatrix}\in\mathbb{C}^{M\times N},\quad
		\Psiv_Y=\begin{pmatrix}
			\Psiv(y^{(1)}) \\
			\vdots              \\
			\Psiv(y^{(M)})
		\end{pmatrix}\in\mathbb{C}^{M\times N}.
\end{equation}
Minimizing the residuals in \cref{eq:LS:Upi} for all $i=1,\ldots,N$ yields the weighted least-squares problem
\begin{equation}\label{eq:EDMD_opt_prob2}
	\min_{\Kv\in\mathbb{C}^{N\times N}} \sum_{m=1}^{M} w_m\left\|\Psiv(y^{(m)})-\Psiv(x^{(m)})\Kv\right\|^2_{\ell^2} = \min_{\Kv\in\mathbb{C}^{N\times N}} \left\|\Wv^{1/2}\Psiv_Y-\Wv^{1/2}\Psiv_X \Kv\right\|_{\mathrm{F}}^2,
\end{equation}
where $\|\cdot\|_{\mathrm{F}}$ denotes the Frobenius norm. Its solution is
\begin{equation}\label{eq:KN-LS-solution}
\Kv = (\Wv^{1/2} \Psiv_X)^\dagger (\Wv^{1/2} \Psiv_Y)= (\Psiv_X^* \Wv \Psiv_X)^{-1}\Psiv_X^* \Wv \Psiv_Y ,
\end{equation}
where in the first equality $(\Wv^{1/2} \Psiv_X)^\dagger$ is the Moore--Penrose pseudoinverse. For any $g=\sum_{i=1}^N \mathbf{g}_i \psi_i\in V_{N}$, we have 
\begin{gather*}
f(x) = [\Koop g](x) = \sum_{i=1}^N \mathbf{g}_i \left[ \sum_{j=1}^N [\Kv]_{ji} \psi_j(x) + \rho_i(x) \right]  = \sum_{j=1}^N \psi_j(x) \mathbf{f}_j + \sum_{i=1}^N \mathbf{g}_i\rho_i(x),\\
\mbox{where}\;\; \mathbf{f}_j=\sum_{i=1}^N [\Kv]_{ji}\mathbf{g}_i,\;\;\mbox{i.e.,}\;\;
\mathbf{f}=\begin{pmatrix}\mathbf{f}_1 \cr \vdots\cr \mathbf{f}_N \end{pmatrix} = \Kv \begin{pmatrix}\mathbf{g}_1 \cr \vdots\cr \mathbf{g}_N\end{pmatrix}=\Kv\mathbf{g}.
\end{gather*}
The matrix $\Kv$ is a data-driven approximation of the matrix representation of the finite section $\mathcal{P}_{V_{N}}\Koop\mathcal{P}_{V_{N}}^*$, with respect to the dictionary $\mathcal{D}=\{\psi_1,\ldots,\psi_{N}\}$.

At this stage, rigorous analysis requires an additional limiting process. In the data limit 
$M\rightarrow\infty$ with suitable quadrature weights, the algebraic least-squares projection converges to the Hilbert space projection. Only then can the analysis proceed in the Hilbert space setting as 
$N\rightarrow\infty$, as described above.

When $N > M$, the solution of \cref{eq:EDMD_opt_prob2} is not unique. For any matrix $\Bv$ such that $\Psiv_X \Bv=0$, the matrix $\Kv+\Bv$ is also a minimizer. The particular choice $\Kv$ via the pseudoinverse selects the one of smallest norm among all minimizers (this additional property is built into the definition of the generalized inverse).
Even when $\Psiv_X$ has full column rank, it may be ill-conditioned. In practice, least-squares solvers employ rank-revealing decompositions (such as pivoted QR or SVD). In the presence of numerical rank deficiency, the returned solution—being non-unique—may depend on the chosen algorithm and its implementation; see \cite{Drmac-Mezic-Mohr-InfGen-2021}.

\subsection{Further compressions}

When $\Psiv_X$ has rank $r<N$, we can reduce the eigenvalue problem of $\Kv$ to an $r$--dimensional one by using the Rayleigh quotient with respect to the range of $\Psiv_X^\dagger$. If $\Psiv_X$ has full row rank (so $r=M<N$), we obtain $(\Psiv_X^\dagger)^\dagger \Kv\Psiv_X^\dagger=\Psiv_Y \Psiv_X^\dagger$.  Because of possible ill-conditioning, this reduction is delicate: the numerical rank and the range of $\Psiv_X$ are best determined using the SVD, which provides an optimal low-rank approximation within a given tolerance. Moreover, the Rayleigh quotient should be computed with respect to an orthonormal basis.

To this end, let the economy size SVD of $\Psiv_X$ be 
\begin{equation}\label{eq:SVD-psix}
\Psiv_X = U \Sigma V^*,\;\;\Sigma=\mathrm{diag}(\sigma_1,\ldots,\sigma_r),\, U\in\C^{M\times r},\;
V\in\C^{N\times r},\;U^*U=V^*V=I_r.
\end{equation}
Using $\Psiv_X^\dagger = V\Sigma^{-1}U^*$, the Rayleigh quotient $\Kv_r = V^* \Kv V$ is the $r\times r$ matrix
\begin{equation}\label{eq:RQ0}
\Kv_r = V^* \Kv V = V^*(V\Sigma^{-1}U^*)\Psiv_Y V = \Sigma^{-1}U^*\Psiv_Y V.
\end{equation}
It can be checked that $\Kv V=V\Kv_r$.
Assuming $\Kv_r$ is diagonalizable, its $r$ eigenpairs, $\Kv_r \sv_i = \lambda_i \sv_i$,  are lifted using $V$ into eigenpairs of $\Kv$:
$$
\Kv (V\sv_i) = \Psiv_X^\dagger \Psiv_Y V \sv_i = V \Sigma^{-1}U^*\Psiv_Y V \sv_i = V \Kv_r \sv_i=\lambda_i (V \sv_i).
$$ 
The approximate eigenfunctions of $\Koop$ are
$$
\phi_i(x)   = \sum_{j=1}^N \psi_j(x) (V\sv_i)_j,\quad i=1,\ldots, r.
$$
Below, when discussing ResDMD, we shall show how to compute the residual $\|\Koop \phi_i-\lambda_i\phi_i\|$ associated with such approximations.

The rank revealing SVD in \cref{eq:SVD-psix} assumes exact computation. In finite precision computation, determining the numerical rank is a delicate issue.
Instead, a numerical rank is determined 
as
\begin{equation}\label{eq:truncation}
	k = \max\{ i\; :\; \sigma_i > \sigma_1 \tol\},
\end{equation} 
where the tolerance level is usually a multiple of the round-off unit $\roff$, e.g. $\tol=N\roff$.
By the Eckart--Young--Mirsky theorem, the closest matrix of rank of at most $k<r$ to $\Psiv_X$ is at the distance $\sigma_{k+1}$ (in the spectral norm). The above construction can be repeated with $k$ instead of $r$, with  $\Psiv_X \approx U_k\Sigma_k V_k^*$, $U_k=U(:,1:k)$, $V_k=V(:,1:k)$, $\Sigma_k=\Sigma(1:k,1:k)$.

Finally, the finite-dimensional compression $\KoopM_N$ and its approximation $\Kv$ may fail to be diagonalizable (even if $\Koop$ is unitary) or may possess highly ill-conditioned eigenvectors. An alternative approach---constructing a (triangular) Schur form of the Koopman operator and using a flag of nearly invariant subspaces instead of eigenfunctions---is proposed in \cite{Drmac-Mezic-Koopman-Schur-2024}.

\subsection{Transposes and the DMD connection}
There is a connection between EDMD (the finite section method) and DMD (the most widely used algorithm associated with Koopman operators). We have
\begin{equation}\label{eq:KNT}
\begin{pmatrix} \mathcal{P}_{V_{N}}\Koop\psi_1(x) \cr \vdots \cr \mathcal{P}_{V_{N}}\Koop\psi_N(x)\end{pmatrix} \simeq
\Kv^\top \begin{pmatrix} \psi_1(x) \cr \vdots \cr \psi_N(x)\end{pmatrix},\;\;\mbox{i.e.,}\;\;
\begin{pmatrix} \psi_1(y^{(m)}) \cr \vdots \cr \psi_N(y^{(m)})\end{pmatrix} \simeq
\Kv^\top \begin{pmatrix} \psi_1(x^{(m)}) \cr \vdots \cr \psi_N(x^{(m)})\end{pmatrix},\;m=1,\ldots,M.
\end{equation}
Interpret $\Psiv(x^{(m)})^\top=(\psi_1(x^{(m)}), \ldots, \psi_{N}(x^{(m)}))^\top$ as data snapshots with the dictionary functions $\psi_i$'s as observables, and store them column--wise to form
$$
		\Psiv_X^\top=\begin{pmatrix}
			\Psiv(x^{(1)})^\top &\cdots&
			\Psiv(x^{(M)})^\top
		\end{pmatrix}\in\mathbb{C}^{N\times M},\quad
		\Psiv_Y^\top=\begin{pmatrix}
			\Psiv(y^{(1)})^\top &
			\cdots              &
			\Psiv(y^{(M)})^\top
		\end{pmatrix}\in\mathbb{C}^{N\times M}.
$$
In \cref{eq:KNT}, the columns $\Xv=\Psiv_X^\top$ are ``pushed forward'' into $\Yv=\Psiv_Y^\top$ by a linear mapping. If $\mathcal{X}=\mathbb{R}^d$, $N=d$, and $\psi_i(x)=e_i^\top x$, then $\Psiv(x^{(m)})^\top=x^{(m)}$ becomes the full state observable. Several trajectories can be arranged in the data matrices $\Xv$ and $\Yv$.

Given a collection of snapshots generated by nonlinear dynamics, one can postulate a linear relation $\Yv \approx \mathbb{A} \Xv$ without explicitly invoking the Koopman operator. This idea is based on local tangential approximation and sampling over small intervals, a common technique in the numerical solution of differential equations. Schmid and Sesterhenn \cite{Schmid-Sesterhenn-DMD-2008,Schmid-DMD-2010} developed this approach into the Dynamic Mode Decomposition (DMD)—a powerful computational tool in fluid dynamics, well suited both for data-driven applications and for analyses based on numerical simulations.

Hence, we can think of each column of $\Yv$ as a result of a linear operator action on the corresponding column of $\Xv$, and we can try to find a matrix $\mathbb{A}$ such that $\Yv-\mathbb{A} \Xv$ is small. Since $\Xv=\Psiv_X^\top$, $\Yv=\Psiv_Y^\top$,  the optimal $\mathbb{A}$ minimizes $\|\Xv^\top\mathbb{A}^\top-\Yv^\top\|_{\mathrm{F}}=\|\Psiv_X \mathbb{A}^\top - \Psiv_Y\|_{\mathrm{F}}$, which means that $\mathbb{A} = \Kv^\top= (\Psiv_X^\dagger\Psiv_Y)^\top= ((\Xv^\top)^\dagger \Yv^\top)^\top=\Yv\Xv^\dagger$, where $\Kv$ is defined in \cref{eq:EDMD_opt_prob2,eq:KN-LS-solution}, with $W=\tfrac{1}{M}I_N$. In typical applications, the matrices $\Xv$ and $\Yv$ are tall and skinny with $N\gg M$ and, as already discussed, the solution is a linear manifold and the particular choice $\mathbb{A}=\Yv\Xv^\dagger=(\Psiv_Y)^\top(\Psiv_X)^{\top\dagger}$ is selected to have smallest norm.

\subsubsection{Rayleigh--Ritz extraction}

The goal of DMD is to decompose the data snapshots in terms of eigenvectors of $\mathbb{A}$. Since in this setting $\mathbb{A}$ is implicitly defined only on the range of $\Xv$, it is natural to use a Rayleigh--Ritz extraction. An orthonormal basis for the range of $\Xv$ is available from \cref{eq:SVD-psix} as $V^{*\top}$ and the corresponding Rayleigh quotient is 
\begin{equation}\label{Ar=KrT}
\mathbb{A}_r = V^\top\mathbb{A} V^{*\top} = (V^*\mathbb{A}^\top V)^\top = \Kv_r^\top.
\end{equation}
The matrix $\mathbb{A}_r$ has the same eigenvalues as $\Kv_r$. The rank revealing SVD of $\Xv$ is $\Xv=U_x\Sigma V_x^*$, where (using \cref{eq:SVD-psix}) $U_x=V^{*\top}$, $V_x^*=U^\top$. From the definition of $\mathbb{A}$, $\mathbb{A}\Xv = \Yv \Pv_{\Xv^\top}$, i.e., $\mathbb{A} U_x\Sigma V_x^* = \Yv V_x V_x^*$ and then $\mathbb{A} U_x\Sigma = \Yv V_x$. For all $k=1,\ldots,r$, with $U_{xk}=U_x(:,1:k)$, $V_{xk}=V_x(:,1:k)$, $\Sigma_k=\Sigma(1:k,1:k)$, the Rayleigh quotient $A_k=U_{xk}^*\mathbb{A} U_{xk}$ is computed from
\begin{equation}\label{eq:RQ-Ak}
\mathbb{A} U_{xk} = \Yv V_{xk}\Sigma_k^{-1} \; \mbox{as}\; A_k = U_{xk}^*\mathbb{A} U_{xk}=U_{xk}^*\Yv V_{xk}\Sigma_k^{-1}.
\end{equation}
If $A_k b_i = \lambda_i b_i$, $\|b_i\|_{\ell^2}=1$, $i=1,\ldots, k$, then the approximate eigenvectors of $\mathbb{A}$ 
are $\zv_i = U_{xk}b_i = V^{*\top}(:,1:k)b_i$.

\subsubsection{Finite-dimensional residuals for the compression of $\mathbb{A}$}

It is important to understand that $U_{xk}$ (or $\Xv$) does not, in general, span an $\mathbb{A}$-invariant subspace. Using all computed pairs $(\lambda_i,U_{xk}b_i)$ in the modal analysis of the data is not justified---a common mistake in the published DMD literature and its applications. In fact,
$\mathbb{A} (U_{xk}b_i) = \lambda_i (U_{xk}b_i) + (I_N - U_{xk}U_{xk}^*)\mathbb{A} U_{xk}b_i,\;\;i=1,\ldots, k,$ so that each eigenpair $(\lambda_i,\zv_i)$ has the (computable) residual
\begin{equation}\label{eq:dd-residual}
r_k(i) = \| \mathbb{A}\zv_i - \lambda_i\zv_i \|_{\ell^2} = 
\| \Yv V_{xk}\Sigma_k^{-1} \zv_i - \lambda_i\zv_i\|_{\ell^2}, \;\;\; i=1,\ldots, k.
\end{equation}
Thus, every approximate eigenpair $(\lambda_i,\zv_i)$ is accompanied by a residual (associated with the larger, but finite matrix $\mathbb{A}$---we discuss residuals associated with $\Koop$ in the next section), which can be used to identify and retain only the reliable approximations. This procedure is summarized in \cref{zd:ALG:DMD}.

Extracting spectral information from the range of $\Xv$ is most effective when the largest possible subspace is used. Why, then, do we truncate the singular values and restrict to a smaller subspace? A key reason is that the SVD is computed in finite precision, and the smallest singular values are typically contaminated by large numerical errors. As a result, the data-driven formula~\eqref{eq:RQ-Ak} becomes unreliable.

\begin{algorithm}[t]
\caption{$(Z_k, \Lambda_k, r_k, [C_k], [Z_k^{(ex)}])=\mathrm{DMD}(\Xv,\Yv; \tol)$\\
A useful preprocessing step \cite{Drmac-Mezic-Mohr-EnhancedDMD-2018,Drmac-2020-koopman-book-chapter,Drmac-DMD-TOMS-2024,Drmac-Herm-DMD-TOMS-2024} is to scale the data: multiplying both $\Xv$ and $\Yv$ on the right by a diagonal matrix that normalizes the columns of $\Xv$ is permissible operation. This improves the condition number of $\Xv$, yielding a more accurate SVD and thereby allowing more singular values to satisfy \cref{eq:truncation}.}
\label{zd:ALG:DMD}
\begin{algorithmic}[1]
	\REQUIRE \  
	$\Xv=(x^{(1)},\ldots,x^{(M)}), \Yv=(y^{(1)},\ldots,y^{(M)})\in {\R}^{N\times M}$ that define a sequence of snapshots
	pairs $(x^{(i)},y^{(i)})$ (with $M \ll N$), tolerance $\tol$ for the truncation (\ref{eq:truncation}).
	\STATE $D_{\Xv}=\mathrm{diag}(\|\Xv(:,1)\|_{\ell^2},\|\Xv(:,2)\|_{\ell^2},\ldots,\|\Xv(:,M)\|_{\ell^2})$; $\Xv_c= \Xv D_{\Xv}^{\dagger}$; $\Yv_c =\Yv D_{\Xv}^{\dagger}$.
	\STATE $[U,\Sigma, V]=\texttt{svd}(\Xv_c)$ ; \hfill\COMMENT{\emph{Thin SVD: $\Xv_c = U \Sigma V^*$, $U\in {\C}^{N\times M}$, $\Sigma=\mathrm{diag}(\sigma_1,\ldots,\sigma_M)$}}
	\STATE Determine numerical rank $k$, using (\ref{eq:truncation}) with the threshold $\tol$.
	\STATE Set $U_k=U(:,1:k)$, $V_k=V(:,1:k)$, $\Sigma_k=\Sigma(1:k,1:k)$ 	
	\STATE ${C}_k = \Yv_c (V_k\Sigma_k^{-1})$; \hfill\COMMENT{\emph{Data driven formula (\ref{eq:RQ-Ak}) for $\mathbb{A} U_k$. [optional output]}}
	\STATE $A_k = U_k^* C_k$ \COMMENT{\emph{$A_k = U_k^* \mathbb{A} U_k$ is the Rayleigh quotient (\ref{eq:RQ-Ak}).}}
	\STATE $[B_k, \Lambda_k] = \texttt{eig}(A_k)$ \hfill\COMMENT{$\Lambda_k=\mathrm{diag}(\lambda_i)_{i=1}^k$; $A_k B_k(:,i)=\lambda_i B_k(:,i)$; $\|B_k(:,i)\|_{\ell^2}=1$}
	\STATE $Z_k = U_k B_k$ \hfill\COMMENT{\emph{The Ritz vectors}}
	\STATE $Z_k^{(ex)} = C_k B_k$ \hfill\COMMENT{\emph{Optional: Unscaled Exact DMD vectors \cite[\S 2.2--2.3]{tu-rowley-dmd-theory-appl-2014}, \cite[\S 3.1]{Drmac-DMD-TOMS-2024}.}}
	\STATE $r_k(i) = \|C_k B_k(:,i) - \lambda_i Z_k(:,i)\|_{\ell^2}$, $i=1,\ldots, k$. \hfill\COMMENT{\emph{The residuals (\ref{eq:dd-residual}).}}
	\ENSURE $Z_k$, $\Lambda_k$, $r_k$, $[C_k]$, $[Z_k^{(ex)}]$.
\end{algorithmic}
\end{algorithm}

\subsubsection{Spatio--temporal representation}

A common task of the DMD analysis is to compute a spectral spatio--temporal representation of the snapshots:  $x^{(m)}\approx  \sum_{j=1}^\ell \zv_{\varsigma_j} \alpha_j \lambda_{\varsigma_j}^{m-1}$, $m=1,\ldots , M$, i.e.
	\begin{equation}\label{eq:f_i-reconstruct-ell}
	      \Xv \approx \begin{pmatrix} \zv_{\varsigma_1} & \zv_{\varsigma_2} & \ldots & \zv_{\varsigma_\ell} \end{pmatrix} \begin{pmatrix} 
		{\alpha}_1 &  &  & \cr
		& {\alpha}_2 &  &   \cr 
		&  & \ddots &        \cr 
		&        &  & {\alpha}_\ell\end{pmatrix}
	\begin{pmatrix} 
		1 & \lambda_{\varsigma_1} & \ldots & \lambda_{\varsigma_1}^{M-1} \cr
		1 & \lambda_{\varsigma_2} & \ldots & \lambda_{\varsigma_2}^{M-1} \cr
		\vdots & \vdots & \cdots & \vdots \cr
		1 & \lambda_{\varsigma_\ell} & \ldots & \lambda_{\varsigma_\ell}^{M-1} \cr
	\end{pmatrix} .
	\end{equation}
For some suitable selection of the modes $\zv_{\varsigma_j}$ and weights 
$\omega_m\geq 0$, the reconstruction coefficients are selected by solving the least squares problem
\begin{equation}\label{eq:rec-error-min}
		\min_{\alpha_j}\sum_{m=1}^{M} \omega_i^2 \left\| x^{(m)} - \sum_{j=1}^\ell \zv_{\varsigma_j} \alpha_j \lambda_{\varsigma_j}^{m-1}\right\|_{\ell^2}^2. 
	\end{equation}	
This is a structured least-squares problem; see \cite{Jovanovic-Schmid-SPDMD:2014,SPDMD-Software,Drmac-Mezic-Mohr-LS-Khatri-Rao-2020} for further details. If the goal is, for example, forecasting, the weights $\omega_m$ can be chosen to emphasize the most recent snapshots.

\subsection{Example}
We illustrate the DMD algorithm on a standard benchmark---the two-dimensional linearized Navier--Stokes equation for plane Poiseuille flow.  
The data are generated as in \cite[\S IV.A.]{Jovanovic-Schmid-SPDMD:2014} and the Orr--Sommerfeld equation for the wall-normal velocity fluctuations is solved using the accompanying Matlab software \cite{SPDMD-Software}. 
The discretized Orr--Sommerfeld operator is denoted by $\mathbf{\Omega}$, and the 
simulation data are obtained by $x_{n+1}=e^{\Delta t\mathbf{\Omega}}x_n$, corresponding to a linear dynamical system.

The key information from the DMD computation is shown in \cref{DMD_eigs_residuals}. 
The truncation of the SVD of $\Xv$ is at $k=26$ and the $26$ computed Ritz values of $e^{\Delta t\mathbf{\Omega}}$
(the DMD eigenvalues) are shown in the middle panel.
The spectral information extracted from the data is $(Z_k, \Lambda_k)$. The right panel shows that not all computed Ritz pairs $(\lambda_i,\zv_i)$ have small residuals, and some are considerably better than the others.

\begin{figure}[t]
\centering
\includegraphics[width=0.32\linewidth]{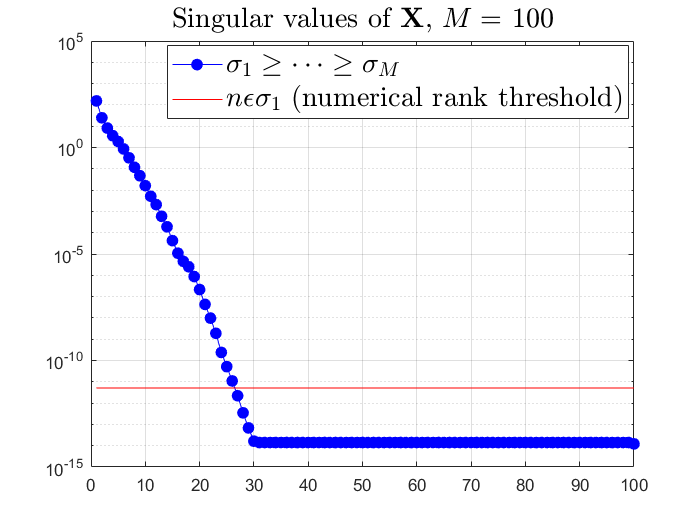}\hfill
\includegraphics[width=0.32\linewidth]{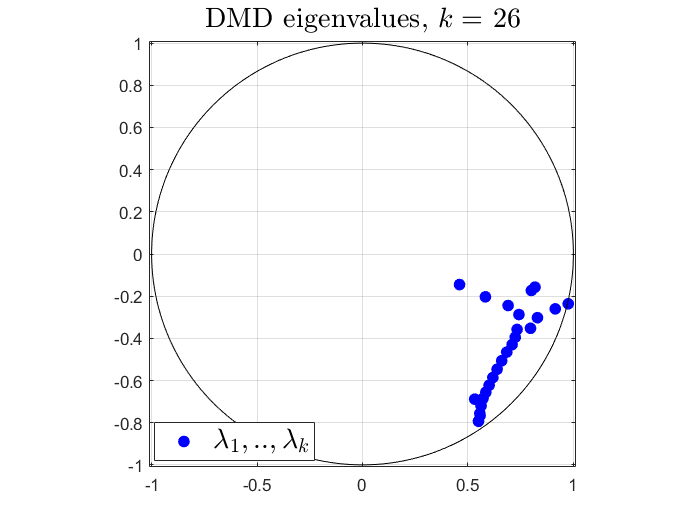}\hfill
\includegraphics[width=0.32\linewidth]{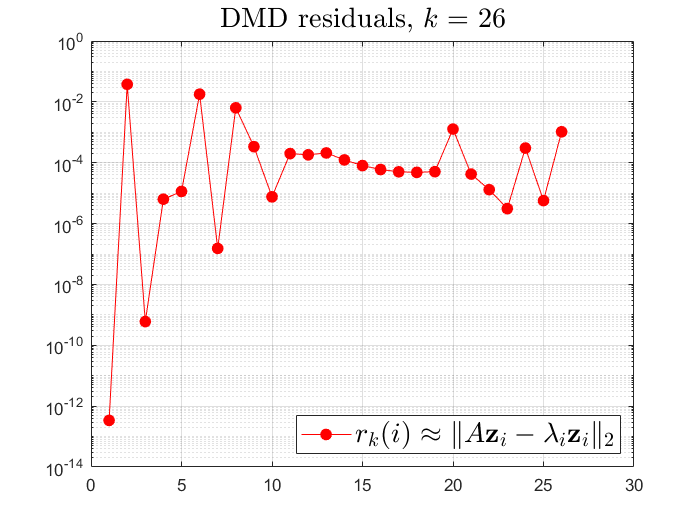}
\caption{\textit{Left:} The singular values of $\Xv$. After truncation, $k=26$ singular values are used and $\Xv \approx U_k\Sigma_k V_k^*$. \textit{Middle:} The eigenvalues computed by the DMD algorithm. \textit{Right:} The DMD residuals.\label{DMD_eigs_residuals}} 
\end{figure}

Since the matrix $\mathbb{A}$ that generated the data is known in this example, we can assess how well the residual reflects the quality of the output. Let the eigenvalues of $\mathbf{\Omega}$ be $\omega_1, \ldots, \omega_n$; then the
eigenvalues of $\mathbb{A}=e^{\Delta t\mathbf{\Omega}}$ are $e^{\Delta t\omega_j}$, $j=1,\ldots, n$. \cref{fig:CHF_eigs_labeled} compares the DMD eigenvalues and the eigenvalues of $e^{\Delta t\mathbf{\Omega}}$.

\begin{figure}[t]
\centering
\includegraphics[width=0.45\linewidth]{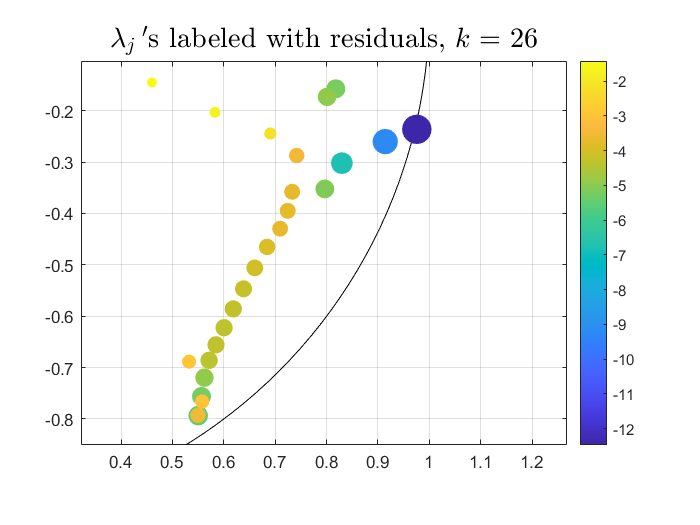}\hfill
\includegraphics[width=0.45\linewidth]{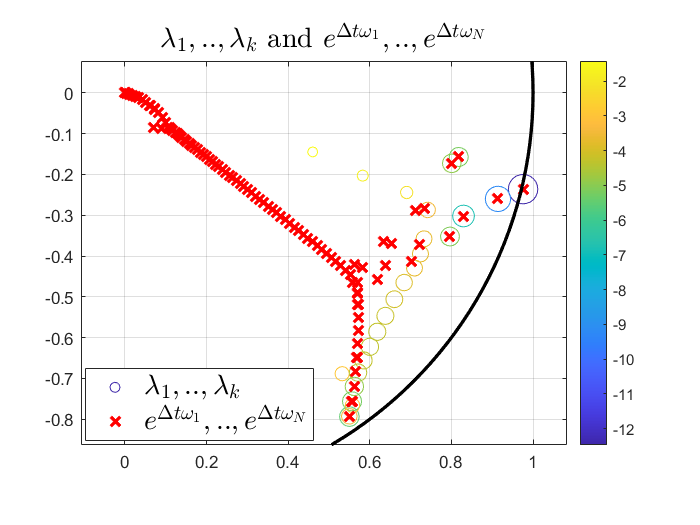}
\caption{\textit{Left:} DMD eigenvalues with corresponding residuals (see \cref{DMD_eigs_residuals}). Larger markers indicate higher accuracy, i.e., smaller residual (colorbar indicates $\log_{10}$ of residuals). \textit{Right:} Comparison of the computed Ritz values (DMD eigenvalues) with explicitly computed eigenvalues of $e^{\Delta t\mathbf{\Omega}}$. \label{fig:CHF_eigs_labeled}} 
\end{figure}

In the next stage of the DMD analysis, the objectives are to reveal the latent structure of the data and to obtain forecasting capability. The coefficients $\alpha_j$ in \cref{eq:f_i-reconstruct-ell} are often computed as $(\alpha_j)_{j=1}^k = Z_k^\dagger x^{(1)}$, after which the remaining snapshots are approximated as $x^{(m)} \approx \sum_{j=1}^k \zv_{j} \alpha_j \lambda_{j}^{m-1}$, tacitly assuming that $\mathbb{A} \zv_j \approx \lambda_j \zv_j$. This approach often performs well, especially for data arising from linearizations. However, in more challenging cases—such as strongly nonlinear dynamics, noisy data, or large residuals—it can produce large reconstruction errors for some snapshots \cite{Drmac-Mezic-Koopman-Schur-2024}. In general, it is safer to solve the least-squares problem in \cref{eq:rec-error-min} explicitly with $\ell = r$. Nevertheless, reconstructing with all (i.e., too many) computed modes is usually not optimal for uncovering latent structure.

\Cref{fig:CHF_reconstr} shows the reconstruction coefficients, illustrating that not all DMD modes contribute equally to representing the data. The goal is to achieve accurate reconstruction with $\ell$ as small as possible. Which modes should be used? This question was addressed in \cite{Jovanovic-Schmid-SPDMD:2014}, leading to the sparsity-promoting DMD (DMDSP) algorithm. Although we do not go into detail here, the basic idea is to add a regularization term to the least-squares objective corresponding to the $\ell^1$-norm of the coefficient vector $(\alpha_j)$. \Cref{fig:CHF_DMDSPreconstr} shows the results of such an optimization process. The main advantage of this approach (automated mode selection via black-box optimization) is also, in a sense, its main drawback: the selected modes lack an intuitive connection to the underlying physics and the Koopman operator framework.

\begin{figure}[t]
\centering
\includegraphics[width=0.45\linewidth]{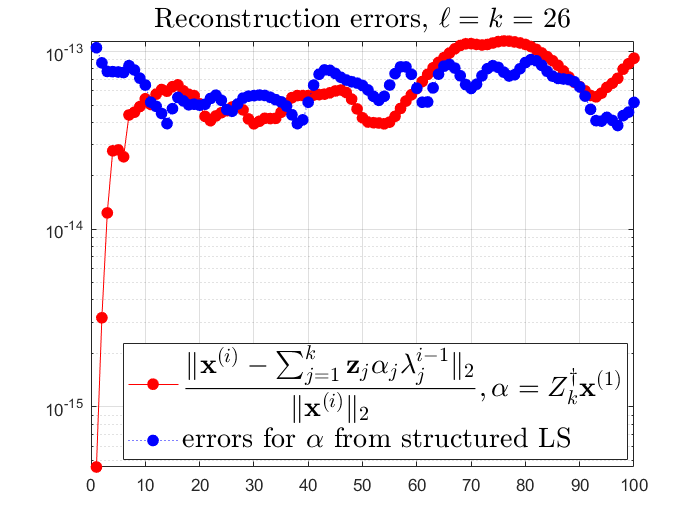}\hfill
\includegraphics[width=0.45\linewidth]{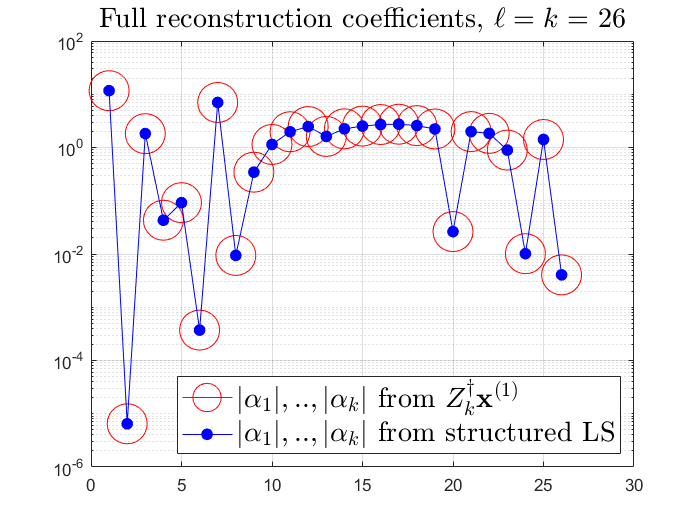}
\caption{\textit{Left:} Reconstruction errors when all computed modes ($\ell=k$) are used in \cref{eq:f_i-reconstruct-ell}. Both methods for computing $\alpha_j$ or solving the structured least squares problem in \cref{eq:rec-error-min} perform well. \textit{Right:} The moduli of the coefficients $\alpha_1,\ldots,\alpha_k$, computed by the two methods. The maximal relative difference between the two sets of values is $\mathcal{O}(10^{-8})$.\label{fig:CHF_reconstr}} 
\end{figure}

\begin{figure}[t]
\centering
\includegraphics[width=0.45\linewidth]{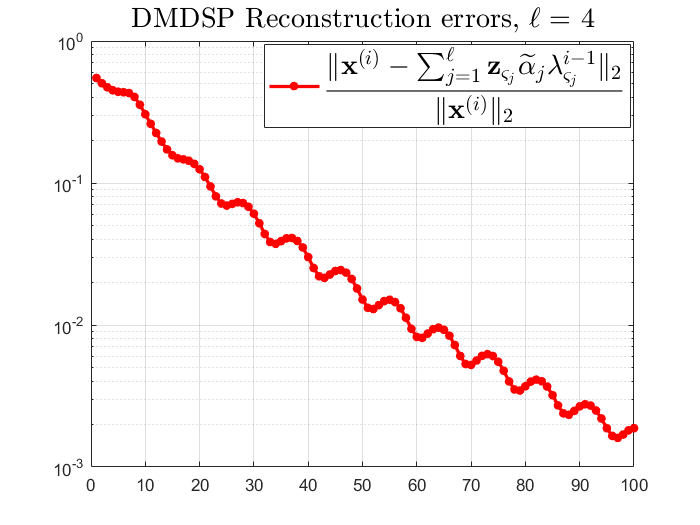}\hfill
\includegraphics[width=0.45\linewidth]{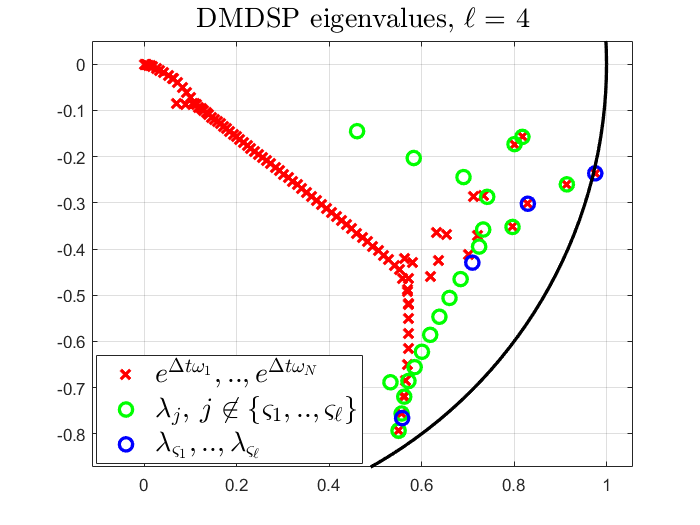}
\caption{\textit{Left:} DMDSP reconstruction errors with $\ell=4$ in \cref{eq:f_i-reconstruct-ell}. \textit{Right:} The eigenvalues $\lambda_{\varsigma_1},\ldots,\lambda_{\varsigma_{\ell}}$ selected by the sparsity-constrained optimizer in DMDSP. \label{fig:CHF_DMDSPreconstr}} 
\end{figure}

An attentive reader will have noticed the relationship between the absolute values of the coefficients $\alpha_j$ (right panel of \cref{fig:CHF_reconstr}) and the residuals (right panel of \cref{DMD_eigs_residuals}). To make this clearer, the left panel of \cref{fig:CHF_alphas_resids} compares $|\alpha_j|$ with $1/r_k(j)$. Since modes (and approximate Koopman eigenvalues) with small residuals are typically more reliable---and the residuals are readily available from \cref{zd:ALG:DMD}---it is natural to pursue sparse representations using modes with the smallest residuals.
For instance, with a residual threshold of $10^{-6}$, the selected eigenvalues are shown in \cref{fig:CHF_alphas_resids}. The corresponding pairs $(\lambda_{\varsigma_j}, \zv_{\varsigma_j})$, $j=1,\ldots,\ell$, are then used in \cref{eq:rec-error-min} to refine the coefficients $\widetilde\alpha_1,\ldots,\widetilde\alpha_{\ell}$, where suitable weighting can prioritize more relevant (e.g., most recent, for forecasting) snapshots.

Pruning modes with large residuals can also be combined with DMDSP optimization. The use of infinite-dimensional residuals for improved compression and mode selection (including nonlinear systems) is considered in \cite{colbrook2023residualJFM,colbrook2024another}.

\begin{figure}[t]
\centering
\includegraphics[width=0.45\linewidth]{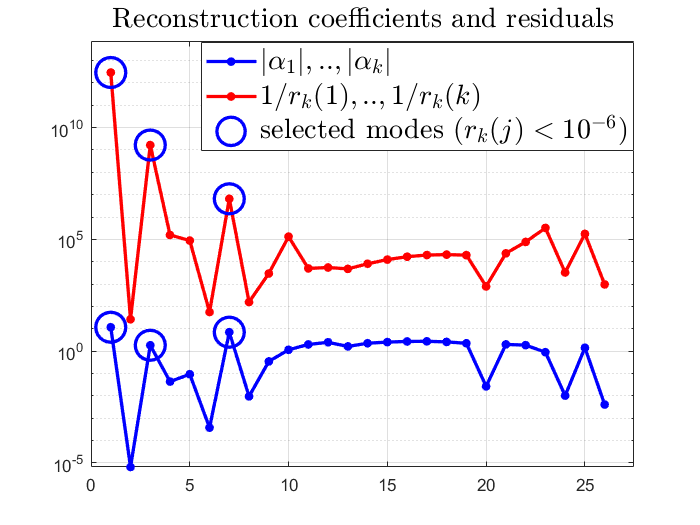}\hfill
\includegraphics[width=0.45\linewidth]{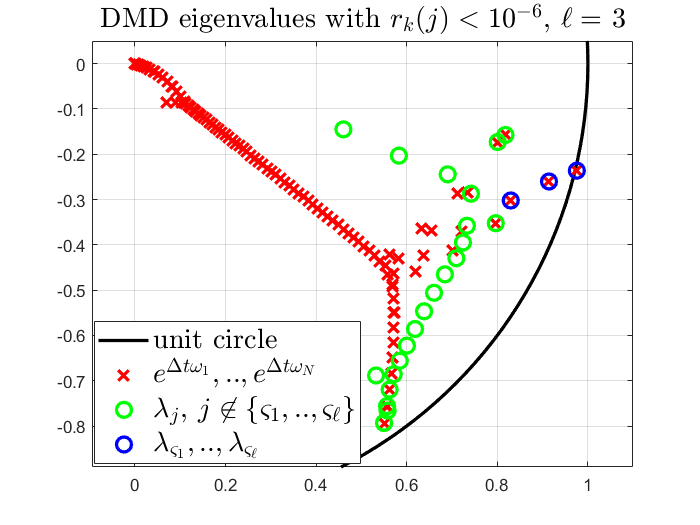}
\caption{\textit{Left:} Large $|\alpha_j|$ corresponds to small residual $r_k(j)$ (large $1/r_k(j)$).
\textit{Right:} The DMD eigenvalues $\lambda_{\varsigma_1}, \lambda_{\varsigma_2}, \lambda_{\varsigma_3}$ with residuals below $10^{-6}$. \label{fig:CHF_alphas_resids}} 
\end{figure}

\begin{figure}[t]
\centering
\includegraphics[width=0.31\linewidth]{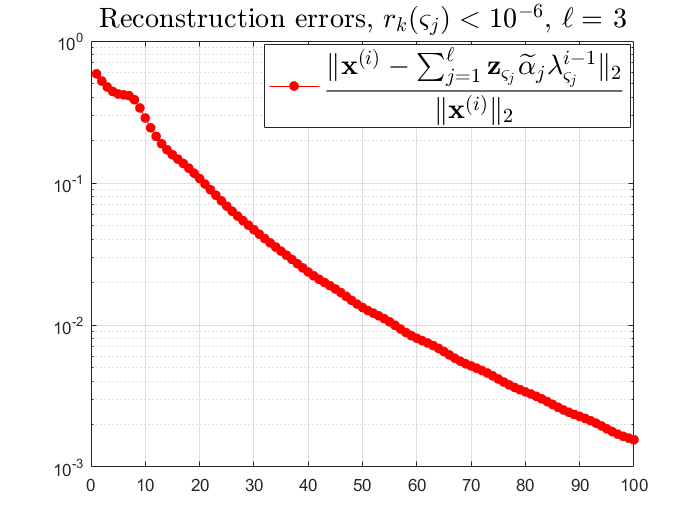}
\includegraphics[width=0.31\linewidth]{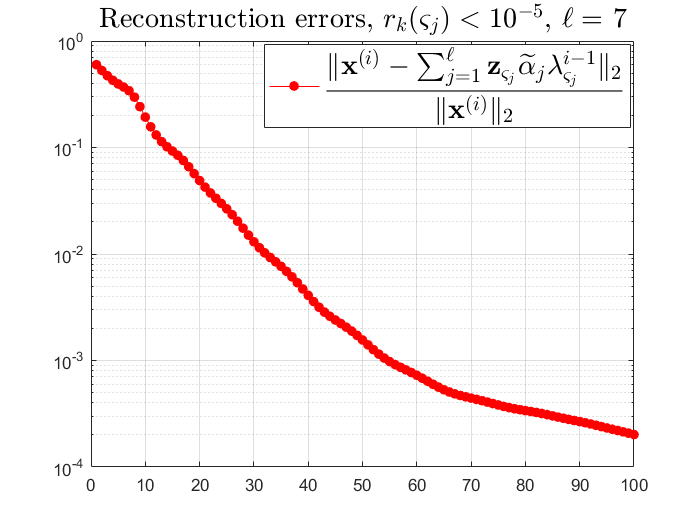}
\includegraphics[width=0.31\linewidth]{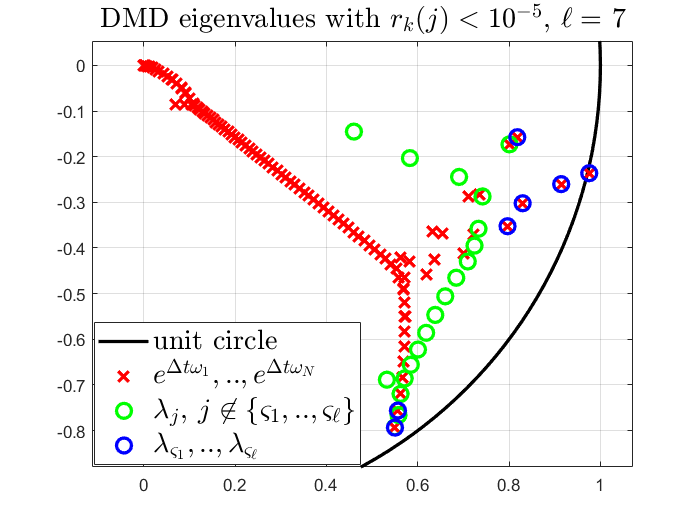}
\caption{\textit{Left:} Reconstruction error using only $(\lambda_j,\zv_j)$ with residuals below $10^{-6}$, $\ell=3$. \textit{Middle and Right:} The residual threshold is set to $10^{-5}$, $\ell=7$. (The better reconstruction accuracy for the snapshots with larger indices is rooted in the properties of Krylov sequences and the convergence mechanism of the power method.) \label{fig:CHF_L3L7}} 
\end{figure}

\section{Controlling projection error from infinite dimensions}

Given a dictionary $\{\psi_1,\ldots,\psi_N\}$ that spans $V_N=\mathrm{span}\{\psi_1,\ldots,\psi_N\}$, the (uncompressed) EDMD matrix $\Kv$ provides a data-driven approximation of the projected Koopman operator $\mathcal{P}_{V_{N}}\Koop\mathcal{P}_{V_{N}}^*$. However, truncation to $V_N$ induces approximation error that depends on the computational objective \cite{brunton2016koopman,kaiser2017data}. For example, in spectral computations, discretization may introduce spurious eigenvalues that accumulate outside the true spectrum as $N\to\infty$---the well-known phenomenon of spectral pollution \cite{lewin2009spectral,davies2004spectral,Pokrzywa_79,colbrook2019compute}. Algorithms designed to rigorously control such projection errors fall under the framework of \textit{residual DMD }(ResDMD), introduced in \cite{colbrook2021rigorousKoop}.

\subsection{Infinite-dimensional residuals}

Let $g=\Psiv\gv\in V_N$ be a candidate Koopman eigenfunction with approximate eigenvalue $\lambda$. The pair $(\lambda,g)$ might come, for example, from an EDMD eigenpair or from minimizing a residual. We assess its quality via the relative residual:
\begin{equation}
\label{residual1}
\begin{split}
\frac{\|(\Koop-\lambda I)g\|}{\|g\|}&=\sqrt{\frac{\int_{\mathcal{X}}|[\Koop g](x)-\lambda g(x)|^2\ \mathrm{d} \omega(x)}{\int_{\mathcal{X}}|g(x)|^2\ \mathrm{d} \omega(x)}}\\
&=\sqrt{\frac{\langle \Koop g,\Koop g \rangle-\lambda\langle g,\Koop g \rangle-\overline{\lambda}\langle \Koop g,g \rangle+|\lambda|^2\langle g,g \rangle}{\langle g,g \rangle}}.
\end{split}
\end{equation}
Here and below, $\|\cdot\|$ and $\langle\cdot,\cdot\rangle$ denote the $L^2(\mathcal{X},\omega)$ norm and inner product.
If $\Koop$ is a normal operator (one that commutes with its adjoint), then
$$
\mathrm{dist}(\lambda,\mathrm{Sp}(\Koop))=\inf_{f}\frac{\|(\Koop-\lambda I)f\|}{\|f\|}\leq\frac{\|(\Koop-\lambda I)g\|}{\|g\|}.
$$
In the case where $\Koop$ is non-normal, the residual in \eqref{residual1} is closely connected to the concept of pseudospectra discussed below.

Adopting the quadrature interpretation of EDMD, we define a finite-data approximation of the relative residual as follows:
$$
\mathrm{res}(\lambda,g)=\|(\Wv^{1/2}\Psiv_Y-\lambda\Wv^{1/2}\Psiv_X)\gv\|_{\ell^2}/\|\Wv^{1/2}\Psiv_X\gv\|_{\ell^2}.
$$
We then have
\begin{align}
[\mathrm{res}(\lambda,g)]^2&=\frac{\gv^*\left[\Psiv_Y^*\Wv\Psiv_Y-\lambda \Psiv_Y^*\Wv\Psiv_X -\overline{\lambda} \Psiv_X^*\Wv\Psiv_Y + |\lambda|^2\Psiv_X^*\Wv\Psiv_X \right]\gv}{\gv^*\Psiv_X^*\Wv\Psiv_X\gv}\notag\\
&=\frac{\gv\left[\Psiv_Y^*\Wv\Psiv_Y-\lambda\Av^*-\overline{\lambda}\Av+|\lambda|^2\Gv\right]\gv}{\gv^*\Gv\gv},\label{residual2}
\end{align}
where $\Gv=\Psiv_X^*\Wv\Psiv_X$ and $\Av=\Psiv_X^*\Wv\Psiv_Y$ are the matrices used in EDMD (see \cref{eq:KN-LS-solution}). The right-hand side of \eqref{residual2} has an additional matrix $\Lv:=\Psiv_Y^*\Wv\Psiv_Y$. Under the assumption that the quadrature rule converges,
$$
\lim_{M\rightarrow\infty}\mathrm{res}(\lambda,g)=\|(\Koop-\lambda I)g\|/\|g\|.
$$
The right-hand side of this equation involves neither approximation nor projection. Hence, it becomes possible to compute an \textit{infinite-dimensional residual} directly from finite matrices, achieving exactness in the limit of large data sets.

\begin{algorithm}[t]
\textbf{Input:} Snapshot data $\{(x^{(m)},y^{(m)})\}_{m=1}^M$, quadrature weights $\{w_m\}_{m=1}^{M}$, dictionary $\{\psi_j\}_{j=1}^{N}$.\\
\vspace{-4mm}
\begin{algorithmic}[1]
\STATE Compute the matrices $\Psiv_X$ and $\Psiv_Y$ defined in \cref{eq:psidef} and $\Wv=\mathrm{diag}(w_1,\ldots,w_{M})$.
\STATE Compute the EDMD matrix $\Kv= (\Wv^{1/2}\Psiv_X)^\dagger \Wv^{1/2}\Psiv_Y\in\mathbb{C}^{N\times N}$.
\STATE Compute the eigendecomposition $
\Kv\Vv=\Vv\mathbf{\Lambda}$. The columns of $\Vv=[\vv_1\cdots\vv_n]$ are eigenvector coefficients and $\mathbf{\Lambda}$ is a diagonal matrix of eigenvalues $\lambda_1,\ldots,\lambda_n$.
\STATE For eigenpairs $(\lambda_j,\vv_j)$ compute $
\mathrm{res}(\lambda_j,\Psiv\vv_j){=} \|(\Wv^{1/2}\Psiv_Y{-}\lambda_j\Wv^{1/2}\Psiv_X)\vv_j\|_{\ell^2}/\|\Wv^{1/2}\Psiv_X\vv_j\|_{\ell^2}.$
\end{algorithmic} \textbf{Output:} The eigenvalues $\mathbf{\Lambda}$, eigenvector coefficients $\Vv\in\mathbb{C}^{N\times N}$ and residuals $\{\mathrm{res}(\lambda_j,\Psiv\vv_j)\}$.
\caption{ResDMD for computing residuals.}
\label{alg:ResDMD1}
\end{algorithm}

One may, for instance, use EDMD to generate candidate eigenpairs $(\lambda,g)$ and then evaluate their residuals as in \cref{alg:ResDMD1}. This step is no more expensive than the EDMD computation itself. The approach is not restricted to EDMD and other sources of candidate pairs can be used. Spectral pollution can be avoided by discarding pairs with residuals above a chosen threshold. However, if \cref{alg:ResDMD1} relies solely on EDMD-generated pairs, parts of the spectrum may remain undetected (spectral invisibility), which typically requires pseudospectral methods to overcome.

\subsection{Pseudospectra}

Given $\epsilon>0$, we define the $\epsilon$-pseudospectrum of $\Koop$ as
\begin{equation}
\label{eq_def_pseudospectra}
\mathrm{Sp}_\epsilon(\Koop)=\mathrm{Cl}\left(\{\lambda\in\mathbb{C}:\|(\Koop-\lambda I)^{-1}\| > 1/\epsilon\}\right)=\mathrm{Cl}\left(\bigcup_{\|\mathcal{B}\|< \epsilon}\mathrm{Sp}(\Koop+\mathcal{B})\right),
\end{equation}
where $\mathrm{Cl}$ denotes the closure of a set. The set $\mathrm{Sp}_\epsilon(\Koop)$ tells how far spectra can move under perturbations of $\Koop$ of magnitude at most $\epsilon$.

Pseudospectra are important for several reasons, including:
\begin{itemize}
\item \textbf{Coherency:} An observable $g$ with $\|g\|=1$ that satisfies $\|(\Koop-\lambda I)g\|\leq\epsilon$ for some $\lambda\in\mathbb{C}$ is known as an $\epsilon$-pseudoeigenfunction. The existence of such an observable implies that $\|(\Koop-\lambda I)^{-1}\| > 1/\epsilon$ so that $\lambda\in\mathrm{Sp}_\epsilon(\Koop)$. Such observables are significant since they satisfy
$$
\|\Koop^ng-\lambda^n g\|= \mathcal{O}(n\epsilon)\quad \forall n\in\mathbb{N}.
$$
Thus, $\lambda$ characterizes an approximately coherent oscillation and decay (or growth) behavior of the observable $g$ over time. The corresponding timescale of coherency is $1/\epsilon$: smaller values of $\epsilon$ correspond to longer coherent timescales.
\item \textbf{Transient effects:} If the Koopman operator is non-normal, the system's transient behavior can differ significantly from the asymptotic behavior described by $\mathrm{Sp}(\Koop)$. In such cases, pseudospectra can be employed to detect and quantify transient effects that are not captured by the spectrum \cite{trefethen1993hydrodynamic,trefethen2005spectra}. One can also use the partial ordered Schur form of $\Koop$ proposed in \cite{Drmac-Mezic-Koopman-Schur-2024}.
	\item \textbf{Reliable computations:}	Pseudospectra allow us to identify regions of computed spectra that are accurate and reliable. Not only do they provide insights into numerical stability, but they also help detect spectral pollution.
	\item \textbf{Computing spectra:} Pseudospectra offer a method for computing spectra since
$
\lim_{\epsilon\downarrow 0}\mathrm{Sp}_\epsilon(\Koop)=\mathrm{Sp}(\Koop).
$
This convergence occurs in the Attouch–Wets metric space \cite{beer1993topologies}, corresponding to uniform convergence on compact subsets of $\mathbb{C}$. This observation extends beyond Koopman operators and has inspired recent breakthroughs in spectral computations in infinite dimensions \cite{ben2015can,colbrook2020PhD,colbrook4,colbrook3,colbrook2019compute}.
\end{itemize}

When computing pseudospectra, it is beneficial to work in the standard $\ell^2$-norm rather than the norm induced by the matrix $\Gv$. To achieve this, we first compute an economy QR decomposition of the data matrix:
$$
\Wv^{1/2}\mathbf{\Psi}_X=\Qv\Rv,\quad  \Qv\in\mathbb{C}^{M\times N},\Rv\in\mathbb{C}^{N\times N},
$$
where $\Qv$ has orthonormal columns, and $\Rv$ is upper triangular with positive diagonals. Setting $\wv=\Rv\gv$, we have
$$
\|\Wv^{1/2}\Psiv_X\gv\|_{\ell^2}^2=\gv^*\Rv^*\Qv^*\Qv\Rv\gv=\gv^*\Rv^*\Rv\gv=\wv^*\wv=\|\wv\|_{\ell^2}^2.
$$
Consequently, the residual can be expressed as:
\begin{equation}
\label{final_residual}
\mathrm{res}(z,g)=\|(\Wv^{1/2}\Psiv_Y\Rv^{-1}-z\Qv)\wv\|_{\ell^2}/\|\wv\|_{\ell^2}.
\end{equation}
For a given $z\in\mathbb{C}$, minimizing this residual corresponds to finding the smallest singular value of the matrix $(\Wv^{1/2}\Psiv_Y\Rv^{-1}-z\Qv)\in\mathbb{C}^{M\times N}$. Denoting the smallest singular value by $\sigma_{\mathrm{\inf}}$, we perform this minimization for various values of $z$. If $M>N$, a computational advantage is gained by considering the smaller $N\times N$ matrix $(\Wv^{1/2}\Psiv_Y\Rv^{-1}-z\Qv)^*(\Wv^{1/2}\Psiv_Y\Rv^{-1}-z\Qv)$ and computing
$$
\sqrt{\sigma_{\mathrm{\inf}}((\Wv^{1/2}\Psiv_Y\Rv^{-1}-z\Qv)^*(\Wv^{1/2}\Psiv_Y\Rv^{-1}-z\Qv))}=\sigma_{\mathrm{\inf}}(\Wv^{1/2}\Psiv_Y\Rv^{-1}-z\Qv).
$$
Although computing singular values in this way is not ideal—since taking square roots can introduce precision loss—the resulting error is usually negligible compared with those already present in the data matrices or quadrature. If higher precision is required, however, $\sigma_{\mathrm{\inf}}(\Wv^{1/2}\Psiv_Y\Rv^{-1}-\lambda\Qv)$ can instead be computed directly in later algorithms. Using $\Qv^*\Qv=\Iv$, we obtain
\begin{align*}
&(\Wv^{1/2}\Psiv_Y\Rv^{-1}-z\Qv)^*(\Wv^{1/2}\Psiv_Y\Rv^{-1}-z\Qv)\\
&\quad\quad\quad\quad=(\Rv^*)^{-1}\Psiv_Y^*\Wv\Psiv_Y\Rv^{-1}-z (\Rv^*)^{-1}\Psiv_Y^*\Wv^{1/2}\Qv-\overline{z}\Qv^*\Wv^{1/2}\Psiv_Y\Rv^{-1}+|z|^2\Iv.
\end{align*}
The minimum singular value of this matrix is then computed across a grid of $z$ values. This procedure is detailed in \cref{alg:ResDMD2}. If needed, the algorithm can also be extended to compute the associated $\epsilon$-pseudoeigenfunctions.

\begin{algorithm}[t]
\textbf{Input:} Snapshot data $\{(x^{(m)},y^{(m)})\}_{m=1}^M$, quadrature weights $\{w_m\}_{m=1}^{M}$, dictionary $\{\psi_j\}_{j=1}^{N}$, accuracy goal $\epsilon>0$, and grid of points $\{z_\ell\}_{\ell=1}^k\subset\mathbb{C}$.\\
\vspace{-4mm}
\begin{algorithmic}[1]
\STATE Compute the matrices $\Psiv_X$ and $\Psiv_Y$ defined in \cref{eq:psidef} and $\Wv=\mathrm{diag}(w_1,\ldots,w_{M})$.
\STATE Compute an economy QR decomposition $\Wv^{1/2}\mathbf{\Psi}_X=\Qv\Rv$, where $\Qv\in\mathbb{C}^{M\times N},\Rv\in\mathbb{C}^{N\times N}$.
\STATE Compute $\Cv_2=(\Rv^*)^{-1}\Psiv_Y^*\Wv\Psiv_Y\Rv^{-1}$ and $\Cv_1=\Qv^*\Wv^{1/2}\Psiv_Y\Rv^{-1}$.
\STATE Compute $\tau_\ell=\sigma_{\mathrm{\inf}}(\Cv_2-z_\ell \Cv_1^*-\overline{z_\ell} \Cv_1+|z_\ell|^2\Iv)$ for $\ell=1,\ldots, k$ ($\sigma_{\mathrm{\inf}}$ is smallest singular value).

\noindent(If wanted, compute the corresponding right-singular vectors $\wv_\ell$ and set $\vv_j=\Rv^{-1}\wv_j$.)
\end{algorithmic} \textbf{Output:} Estimate of the pseudospectrum $\{z_\ell:\tau_\ell<\epsilon\}$ (if wanted, corresponding pseudoeigenfunctions $\{\mathbf{\Psi}\vv_\ell:\tau_\ell<\epsilon\}$).
\caption{ResDMD for computing pseudospectra. One can also compute the singular values directly (of a $\mathbb{C}^{M\times N}$ matrix) without the square root.}
\label{alg:ResDMD2}
\end{algorithm}

\subsection{Forecast errors and coherency}

We can also use the matrix $\Lv$ to develop forecast error bounds. Throughout this subsection, we assume that all matrices have been evaluated after taking the large-data limit $M\rightarrow\infty$. Suppose, for example, we have an approximation $g\approx\Psiv\gv$, and we aim to bound the one-step forecast error:
$$
\|\Koop g-\Psiv\KoopM_N\gv\|.
$$
If the Koopman operator $\Koop$ is bounded, we can bound this forecast error by
$$
\|\Koop(g-\Psiv\gv)\| +\|(\Koop\Psiv-\Psiv\KoopM_N)\gv\|\leq \|\Koop\|\|g-\Psiv\gv\|+\|(\Koop\Psiv-\Psiv\KoopM_N)\gv\|.
$$
The first term on the right-hand side arises from the approximation $g\approx\Psiv\gv$. The square of the second term is
\begin{align*}
&\langle \Koop\Psiv\gv,\Koop\Psiv\gv \rangle -
\langle \Koop\Psiv\gv,\Psiv\KoopM_N\gv \rangle -
\langle \Psiv\KoopM_N\gv, \Koop\Psiv\gv \rangle +
\langle \Psiv\KoopM_N\gv, \Psiv\KoopM_N\gv \rangle\\
&\quad\quad\quad=\langle \Koop\Psiv\gv,\Koop\Psiv\gv \rangle-\langle \Psiv\KoopM_N\gv, \Psiv\KoopM_N\gv \rangle
=\gv^*(\Lv-\KoopM_N^*\Gv\KoopM_N)\gv,
\end{align*}
where the first equality follows since $\KoopM_N$ corresponds to the compression $\mathcal{P}_{V_{N}}\Koop\mathcal{P}_{V_{N}}^*$. Thus, we have the bound
$$
\|\Koop g-\Psiv\KoopM_N\gv\|\leq
\|\Koop\|\|g-\Psiv\gv\| + \sqrt{\gv^*(\Lv-\KoopM_N^*\Gv\KoopM_N)\gv}.
$$
Often $\|g-\Psiv\gv\|$ can be approximated from the snapshot data:
$$
\|g-\Psiv\gv\|^2\approx
\sum_{m=1}^{M}\omega_k |g(x^{(m)})-\Psiv(x^{(m)})\gv|^2.
$$
The above forecast bound can be iterated to obtain error estimates for successive timesteps. For instance, for two timesteps we obtain:
$$
\|\Koop^2g-\Psiv\KoopM_N^2\gv\|\leq
\|\Koop\|\|\Koop g-\Psiv\KoopM_N\gv\| + \sqrt{(\KoopM_N\gv)^*(\Lv-\KoopM_N^*\Gv\KoopM_N)\KoopM_N\gv}.
$$
Repeated iteration yields general bounds for $\|\Koop^ng-\Psiv\KoopM_N^n\gv\|$ for any $n\in\mathbb{N}$. Under the assumption that $\mathcal{P}_{V_{N}}\Koop\mathcal{P}_{V_{N}}^*$ converges in the strong operator topology to $\Koop$, we conclude that $\Psiv\KoopM_N^n\gv$ converges to $\Koop^ng$, provided that $\lim_{N\rightarrow\infty}\|g-\Psiv\gv\|=0$.

\subsection{Convergence theory}

Several convergence results are known for ResDMD. If the quadrature rule underlying EDMD converges, then
$$
\lim_{M\rightarrow\infty}\mathrm{res}(\lambda,g)=\|(\Koop-\lambda I)g\|/\|g\|.
$$
Hence, spectral pollution can be avoided in the large-data limit by retaining only eigenpairs with small residuals (as in \cref{alg:ResDMD1}). In the same way, one obtains convergence to the forecast error bounds discussed in the previous section.

Let $\Gamma^{\epsilon}_{N,M}$ denote the output $\{z_\ell:\tau_\ell<\epsilon\}$ from \cref{alg:ResDMD2}. With a minor modification to account for the boundary case where $\tau=\epsilon$, we obtain
$$
\lim_{M\rightarrow\infty}\Gamma^{\epsilon}_{N,M}=:\Gamma^{\epsilon}_{N}\subset \mathrm{Sp}_{\epsilon}(\Koop).
$$
Hence, ResDMD provides verified approximations of pseudospectra. Moreover, under mild conditions on the dictionary and an $N$-dependent grid $\{z_\ell\}_{\ell=1}^k$, we have
$$
\lim_{N\rightarrow\infty}\Gamma^{\epsilon}_{N}=\mathrm{Cl}\left(\left\{\lambda\in\mathbb{C}:\exists g\in L^2(\mathcal{X},\omega)\text{ such that }\|g\|=1,\|(\Koop-\lambda I)g\|<\epsilon\right\}\right).
$$
As $\epsilon\downarrow 0$, the set on the right-hand side converges to the approximate point spectrum:
$$
\mathrm{Sp}_{\mathrm{ap}}(\Koop)=\left\{\lambda\in\mathbb{C}:\exists\{g_n\}_{n\in\mathbb{N}}\subset L^2(\mathcal{X},\omega)\text{ such that }\|g_n\|=1,\lim_{n\rightarrow\infty}\|(\Koop-\lambda I)g_n\|=0\right\}.
$$
Thus, ResDMD allows the computation of $\mathrm{Sp}_{\mathrm{ap}}(\Koop)$ through a convergent algorithm. Further modifications permit computing the full pseudospectrum $\mathrm{Sp}_{\epsilon}(\Koop)$, and consequently the entire spectrum $\mathrm{Sp}(\Koop)$.

\subsection{Example}

As a simple example, consider the Hamiltonian system
$$
\dot{x}=y,\quad\dot{y}=x-x^3,
$$
the undamped nonlinear Duffing oscillator, with state $(x,y)\in\mathcal{X}=\mathbb{R}^2$ and Hamiltonian $H=y^2-x^2/2+x^4/2$. We study the corresponding discrete-time system obtained by sampling with step size $\Delta t=0.3$. Specifically, we generate $10^4$ initial points sampled uniformly at random from $[-2,2]^2$, compute their trajectories over $2$ times steps, resulting in $M=2\times 10^4$ snapshot pairs. These are partitioned into $N=50$ clusters using k-means, with centroids $\cv_j$ serving as centers for radial basis functions
$
\psi_j((x,y)^\top)=\exp(-\gamma\|(x,y)^\top-\cv_j\|_{\ell^2}),
$
where $\gamma$ is the squared reciprocal of the average $\ell^2$-norm of the snapshot data after it is shifted to mean zero. 

\cref{fig:duffing} (left) shows EDMD eigenvalues with their residuals. Most are spurious, illustrating significant spectral pollution that persists even as the dictionary size $N$ increases, a consequence of finite-dimensional approximation of the infinite-dimensional Koopman operator. \cref{fig:duffing} (right) displays pseudospectra as contour plots of several $\epsilon$-levels on a logarithmic scale. For this system, the pseudospectra form annular regions $\mathrm{Sp}_\epsilon(\Koop)=\{\lambda\in\mathbb{C}:||\lambda|-1|\leq \epsilon\}$. These pseudospectra are computed using the same snapshot data and dictionary. The resulting plot converges in the double limit $\lim_{N\rightarrow\infty}\lim_{M\rightarrow\infty}$.

\begin{figure}[t]
\centering
\raisebox{-0.5\height}{\includegraphics[width=0.45\textwidth,trim={0mm 0mm 0mm 0mm},clip]{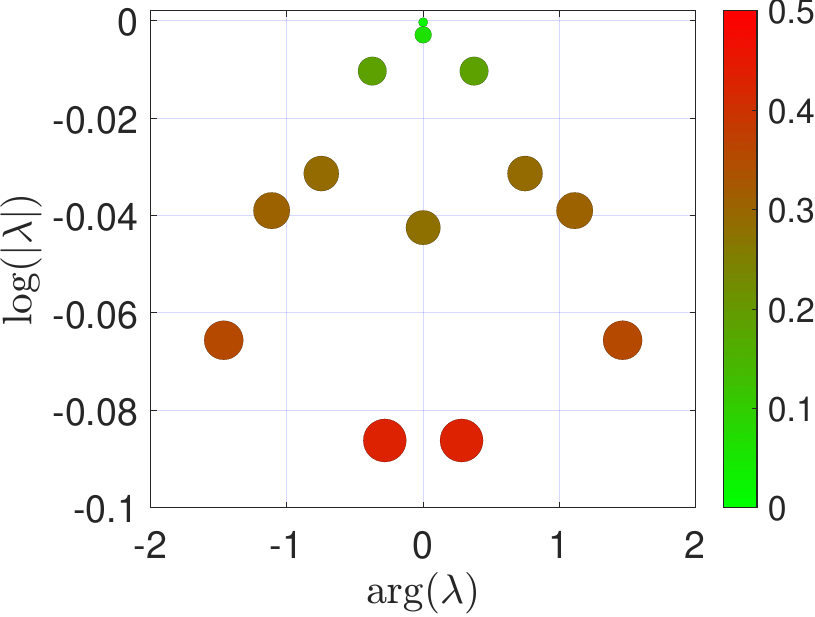}}
\raisebox{-0.5\height}{\includegraphics[width=0.45\textwidth,trim={0mm 0mm 0mm 0mm},clip]{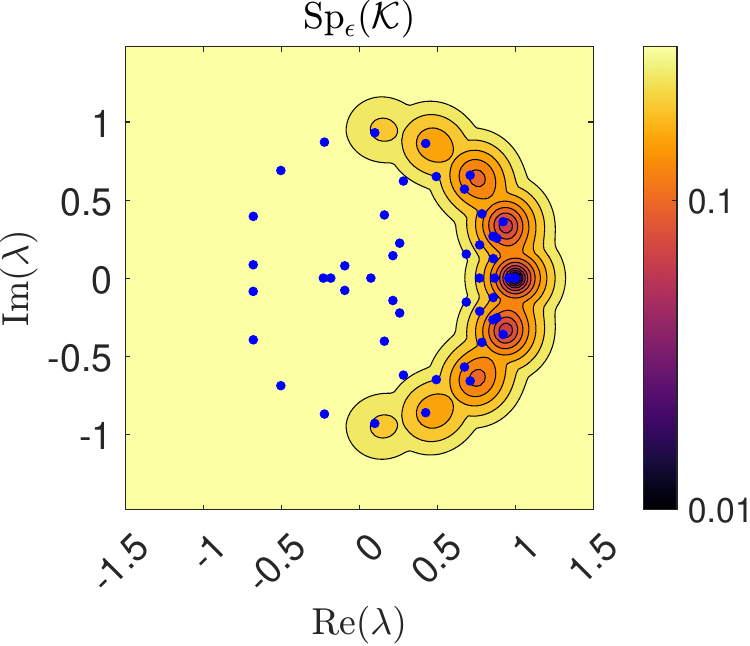}}
\caption{\textit{Left:} EDMD eigenvalues and their residuals computed using \cref{alg:ResDMD1}. The color and size of each eigenvalue represent its residual magnitude. The logarithms of eigenvalues are plotted to align with the continuous-time interpretation of the system. Since the spectrum of the associated Koopman operator is the unit circle—a horizontal line in this plot—most of the displayed EDMD eigenvalues are spurious.
\textit{Right:} Pseudospectra computed via \cref{alg:ResDMD2} using the same snapshot data and dictionary. The EDMD eigenvalues are shown as blue dots.}
\label{fig:duffing}
\end{figure}

\section{Delay embedding and Krylov subspaces}

There is no universal rule for selecting the dictionary, and choosing the functions $\psi_j$ is often more art than science. However, a particularly effective choice is a Krylov subspace based on time-delay embedding, which is well suited to high-dimensional systems, dynamics on unknown or fractal attractors, and settings with partial observations. It provides an intrinsic coordinate system that approximates an invariant subspace without requiring explicit construction. Using the same time step for both the delay interval and measurement frequency yields a data matrix with Hankel structure. In this section, we focus on \textit{Hankel-DMD}, introduced by \cite{arbabi2017ergodic}, which is simply EDMD with a dictionary generated by time-delay embedding.

From an initial observable $g$, Hankel-DMD forms the \textit{Krylov subspace}
$$
V_N=\mathrm{span}\left\{g,\Koop g,\Koop^2g,\ldots,\Koop^{N-1}g\right\}.
$$
Given a single trajectory of the observable, $\{g(x_0),g(x_1),\ldots,g(x_{M+N-1})\}$, the matrices $\Psiv_X$ and $\Psiv_Y$ in \cref{eq:psidef} are given explicitly by the Hankel matrices
$$
\Psiv_X^{(g)}=\begin{pmatrix} 
g(x_0) & g(x_1) &  \cdots & g(x_{N-1})\\
g(x_1) & g(x_2) &  \cdots & g(x_{N})\\
\vdots & \vdots & \vdots & \vdots  \\
g(x_{M-1}) & g(x_M) &  \cdots & g(x_{M+N-2})
\end{pmatrix},\quad
\Psiv_Y^{(g)}=\begin{pmatrix} 
g(x_1) & g(x_2) &  \cdots & g(x_{N})\\
g(x_2) & g(x_3) &  \cdots & g(x_{N+1})\\
\vdots & \vdots & \vdots & \vdots  \\
g(x_{M}) & g(x_{M+1}) &  \cdots & g(x_{M+N-1})
\end{pmatrix},
$$
where the superscript denotes the dependence on the observable $g$.

Suppose that the map $F$ in \cref{eq:dyn_sys} is ergodic, $\mathcal{X}$ is an attractor with a basin of attraction $\mathcal{B}$, and $\omega$ is a physical measure. If $g$ is continuous on $\mathcal{B}$, then \cite{eckmann1985ergodic}
$$
\lim_{M\rightarrow\infty}\frac{1}{M}\sum_{m=0}^{M-1}g(x_m)=\int_{\mathcal{X}} g(x)\dd \omega(x),\quad\text{for Lebesgue-almost every }x_0\in\mathcal{B}.
$$
Hence, we obtain the convergence of quadrature (large data limit $M\rightarrow\infty$). The convergence properties of Hankel-DMD as $N\rightarrow\infty$ depend on whether $g$
 generates a finite-dimensional invariant subspace.
 
\subsection{Invariant versus non-invariant subspaces}

A common assumption in Hankel-DMD is that $g$ generates a finite-dimensional $\Koop$-invariant subspace $V$ of $L^2(\mathcal{X},\omega)$. In other words, $\Koop V\subset V$, allowing us to study certain spectral properties of $\Koop$ by restricting our analysis to $V$. If such a subspace exists with dimension $k$, then it coincides with $V_k$. We can identify this invariant subspace in the limit $M\rightarrow\infty$ by selecting $N=k$ and using the aforementioned dictionary \cite{arbabi2017ergodic}. This result follows from the ergodic theorem combined with the quadrature interpretation of EDMD. These findings also hold when constructing a Krylov subspace from multiple initial observables $g_1,\ldots,g_p$.

However, the existence of such a subspace is not guaranteed, and even if it exists, the dimension $k$ is typically unknown. (This is a common misconception in papers citing \cite{arbabi2017ergodic}, which explicitly makes this point.) In practice, one postulates an \textit{approximate invariant subspace} and truncates the basis to $r \leq N$ modes using a singular value decomposition (SVD). \cref{alg:HDMD} outlines the corresponding procedure.

\begin{algorithm}[t]
\textbf{Input:} $M,N\in\mathbb{N}$, data $\{x_j\}_{j=0}^{M+N-1}$ (single trajectory),
observables $\{g_1,\ldots,g_p\}$, threshold $\epsilon_{\mathrm{tol}}>0$.\\
\vspace{-4mm}
\begin{algorithmic}[1]
\STATE Form the Hankel matrices $\Psiv_X^{(g_k)}$ and $\Psiv_Y^{(g_k)}$ for $k=1,\ldots,p$.
\STATE Compute 
					$\alpha_k= \|(g_k(x_0),\ldots,g_k(x_{M+N-1})\|/\|(g_1(x_0),\ldots,g_1(x_{M+N-1})\|$ for $k=1,\ldots,p$.
\STATE Form the matrices
$
\Xv= \begin{pmatrix}
\alpha_1 \Psiv_X^{(g_1)} & \alpha_2 \Psiv_X^{(g_2)} &\cdots & \alpha_p \Psiv_X^{(g_p)}
\end{pmatrix}
$, 
$
\Yv= \begin{pmatrix}
\alpha_1 \Psiv_Y^{(g_1)} & \alpha_2 \Psiv_Y^{(g_2)} &\cdots & \alpha_p \Psiv_Y^{(g_p)}
\end{pmatrix}.
$
\STATE Compute a truncated SVD $
\Xv \approx \Uv_1 \mathbf{\Sigma}\Uv_2^*$, $\Uv_1\in\mathbb{C}^{M\times r}$, $\mathbf{\Sigma}\in\mathbb{R}^{r\times r}$, $\Uv_2\in\mathbb{C}^{pN\times r}.$ The columns of $\Uv_1$ and $\Uv_2$ are orthonormal, $\mathbf{\Sigma}$ is diagonal, and $r$ selected to keep singular values $\geq \epsilon_{\mathrm{tol}}$.
\STATE Compute the compression $\widehat{\Kv} = \mathbf{\Sigma}^{-1}\Uv_1^*\Yv \Uv_2$ and its eigendecomposition $
\widehat{\Kv}\Vv=\Vv\mathbf{\Lambda}$.
\STATE Compute eigenfunction coordinates $\Uv_2\Vv$.
\end{algorithmic} \textbf{Output:} The eigenvalues $\mathbf{\Lambda}$ (diagonal matrix) and eigenvector coefficients $\Uv_2\Vv\in\mathbb{C}^{pN\times r}$.
\caption{Hankel DMD (EDMD with Krylov subspaces, i.e., delay embedding).}
\label{alg:HDMD}
\end{algorithm}

Suppose that we apply \cref{alg:HDMD} with $p=1$ and without truncating the singular value decomposition. If the subspace spanned by $\{g_1,\Koop g_1,\Koop^2g_1,\ldots\}$ is infinite-dimensional, then, for fixed $N$, all computed eigenvalues (in the limit $M\rightarrow\infty$) lie strictly inside the unit circle \cite{korda2020data}. Moreover, as $N$ increases, provided the spectral measure associated with $g_1$ is regular (in the sense of \cite[page 121]{simon2010szegHo}), these eigenvalues approach and become equidistributed on the unit circle \cite{korda2020data}. This is to be contrasted with mpEDMD, which we discuss below. Applying mpEDMD with a dictionary of time-delayed observables leads to an explicit convergence rate of approximations of spectral measures.

\subsection{Example}

The Lorenz (63) system \cite{lorenz1963deterministic} is the following system of differential equations:
$$
\dot{x}=10\left(y-x\right),\quad\dot{y}=x\left(28-z\right)-y,\quad \dot{z}=xy-8z/3.
$$
We consider the dynamics of $(x,y,z)$ on the Lorenz attractor. This system is chaotic and strongly mixing \cite{luzzatto2005lorenz}, so $\lambda=1$ is the only Koopman eigenvalue, corresponding to the constant eigenfunction, and it is simple. Equivalently, no nontrivial finite-dimensional invariant subspace exists. We consider the discrete-time system obtained by sampling with step size $\Delta t=0.01$.

\begin{figure}[t]
\centering
\raisebox{-0.5\height}{\includegraphics[width=0.3\textwidth,trim={0mm 0mm 0mm 0mm},clip]{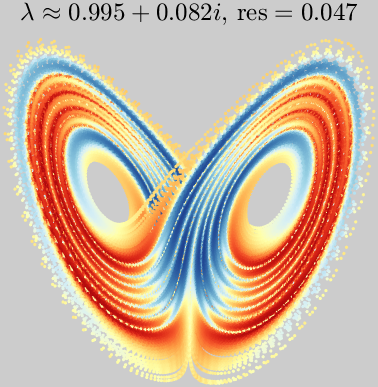}}\hfill
\raisebox{-0.5\height}{\includegraphics[width=0.3\textwidth,trim={0mm 0mm 0mm 0mm},clip]{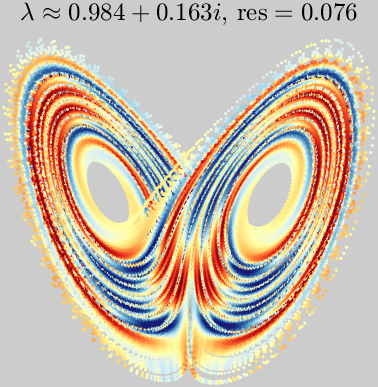}}\hfill
\raisebox{-0.5\height}{\includegraphics[width=0.3\textwidth,trim={0mm 0mm 0mm 0mm},clip]{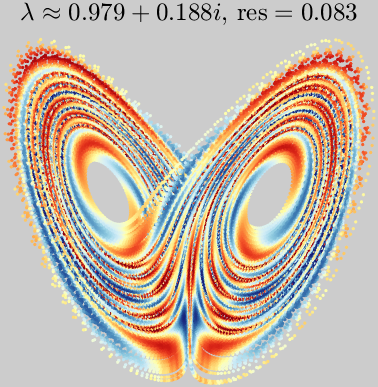}}
\caption{Pseudoeigenfunctions of the Lorenz system computed using \cref{alg:HDMD}. The residuals are computed using \cref{alg:ResDMD1}.}
\label{fig:lorenz}
\end{figure}

We apply \cref{alg:HDMD} with $M=10^5$, after a burn-in period to ensure the trajectory lies (approximately) on the Lorenz attractor. The observables are the coordinate functions $g_1=x,g_2=y,g_3=z$ and we set $N=250$, together with the constant function $1$. The tolerance is chosen as $\epsilon_{\mathrm{tol}}=10^{-10}$. For each computed eigenpair, we then use \cref{alg:ResDMD1} to evaluate the residual. \cref{fig:lorenz} shows the three nontrivial Hankel-DMD eigenfunctions with the smallest residuals. These are pseudoeigenfunctions of $\Koop$, reflecting coherent temporal behavior. The first has also been reported in earlier studies \cite{korda2020data,colbrook2024rigged}.

\section{Koopman modes and Generalized Laplace Analysis}

We now consider Koopman operators that are spectral. In fact, we will address bounded operators on a complex Banach space $X$ in greater generality. We first present a general result, which was proven in \cite{mohr2014construction} using the machinery of Yosida's mean ergodic theorem \cite{yosida2012functional}. Here, instead, we provide a direct elementary proof.

\subsection{Laplace averages for spectral operators}

The following standard facts can be found, for example, in the survey by Dunford \cite{dunford1958survey} and in Part III of Dunford and Schwartz \cite{MR412888}. Let $B(\mathbb{C})$ denote the Borel subsets of $\mathbb{C}$. A bounded and countably additive spectral measure in $X$ is map $\mathcal{E}$ from $B(\mathbb{C})$ to the set of bounded (not necessarily orthogonal) projections in $X$ such that:
\begin{itemize}
    \item $\mathcal{E}(\emptyset)=0$ and $\mathcal{E}(\mathbb{C})=I$;
    \item $\mathcal{E}(E_1\cap E_2)=\mathcal{E}(E_1)\mathcal{E}(E_2)$ and $\mathcal{E}(\mathbb{C}\backslash E)=I-\mathcal{E}(E)$ for any $E,E_1,E_2\in B(\mathbb{C})$;
    \item $\sup_{E\in B(\mathbb{C})}\|\mathcal{E}(E)\|<\infty$;
    \item If $\{E_j\}\subset B(\mathbb{C})$ are disjoint and $x\in X$,
    $
    \mathcal{E}(\cup_{j=1}^\infty E_j)x=\sum_{j=1}^\infty \mathcal{E}(E_j)x$, with convergence in $X$.
\end{itemize}
The measure $\mathcal{E}$ is called a \textit{resolution of the identity} (or \textit{spectral resolution}) for a bounded operator $T$ on $X$ if
$$
\mathcal{E}(E)T=T\mathcal{E}(E),\quad \mathrm{Sp}(\mathcal{E}(E)T\mathcal{E}(E)^*)\subset \mathrm{Cl}(E)\quad \forall E\in B(\mathbb{C}).
$$
If such a $\mathcal{E}$ exists, we say that $T$ is a \textit{spectral operator}. Any bounded spectral operator $T$ can be written as $T=S+N$, where: $S$ is a \textit{scalar type operator}, meaning that
$$
S=\int \lambda \dd\mathcal{E}(\lambda)\quad\text{and $S$ has spectral resolution $\mathcal{E}$};
$$
$N$ is quasi-nilpotent, meaning that $\lim_{k\rightarrow\infty}\|N^k\|^{1/k}=0$ (equivalently, $\mathrm{Sp}(N)=\{0\}$); and $S$ and $N$ commute. If $X$ is a Hilbert space, then a bounded operator $S$ is a scalar type operator, if and only if there exists a bounded operator $B$ with bounded inverse $B^{-1}$ such that $BSB^{-1}$ is normal.
If $X$ is finite-dimensional, then any linear operator $T$ on $X$ is spectral, and the above decomposition is the usual Jordan reduction.

The following proposition is the basis for generalized Laplace analysis.

\begin{proposition}\label{new_prop}
Let $X$ be a complex Banach space and $S$ a bounded scalar type operator on $X$ with spectral resolution $\mathcal{E}$ and spectral radius greater than $0$. Let $z\in\mathbb{C}$ have $|z|=\sup_{\lambda\in\mathrm{Sp}(S)}|\lambda|$,
then
\begin{equation}
\label{GLA_mean}
\lim_{n\rightarrow\infty} \frac{1}{n}\sum_{k=1}^n z^{-k} S^k x=\mathcal{E}(\{z\})x\quad \forall x\in X.
\end{equation}
\end{proposition}

Since $S$ is of scalar type, the range of $\mathcal{E}(\{z\})$ is either $0$, or, if $z$ is an eigenvalue, the corresponding eigenspace. If $z$ is an isolated point of the spectrum, then $\mathcal{E}(\{z\})$ is precisely the associated Riesz projection. It might surprise the reader that \cref{new_prop} makes no assumption about $\lambda$ being a strictly dominant point of the spectrum: $S$ does not need to satisfy the property $\sup_{\lambda\in\mathrm{Sp}(S)\backslash \{z\}}|\lambda|<|z|$.

\begin{proof}[Proof of \cref{new_prop}]
Using the functional calculus and summing a geometric series, we have
$$
\frac{1}{n}\sum_{k=1}^n z^{-k} S^k x=
\left[\int \frac{1}{n}\sum_{k=1}^n z^{-k}\lambda^k \dd \mathcal{E}(\lambda)\right]x=\mathcal{E}(\{z\})x+\left[\int_{\mathrm{Sp}(S)\backslash\{z\}}  \frac{\lambda (1-(\lambda/z)^n)}{n(z-\lambda)} \dd \mathcal{E}(\lambda)\right]x.
$$
Consider the functions in the integral on the right-hand side:
$$
f_n(\lambda) = \frac{1}{n} \frac{\lambda (1-(\lambda/z)^n)}{z-\lambda}.
$$
On any bounded set $K$ that is separated from $z$, $f_n$ converges uniformly to $0$ as $n\rightarrow\infty$. Since, integration against $\mathcal{E}$ is continuous as a map from the space of bounded measurable functions equipped with the supremum norm to the algebra of bounded operators on $X$, it follows that
$$
\lim_{n\rightarrow\infty} \int_{K}f_n(\lambda)\dd\mathcal{E}(\lambda) x=0.
$$
Moreover, the $f_n$ are uniformly bounded and, hence,
$$
\left\|\int_{E} f_n(\lambda) \dd \mathcal{E}(\lambda) x\right\|\leq C\sup_{E'\subset E} \|\mathcal{E}(E')x\|\quad \forall n\in\mathbb{N}
$$
for some constant $C$.
Hence, it is enough to show that if $C_n$ is the punctured disk $\{w\in\mathbb{C}:0<|w-z|\leq 1/n\}$, then
$
\lim_{n\rightarrow\infty}
\sup_{E\subset C_n} \|\mathcal{E}(E)x\|=0.
$
Suppose this were false, then without loss of generality (using the countably additivity of $\mathcal{E}$ and writing
$E=\cup_{n} E\cap (C_n\backslash C_{n+1})$ for $E\subset C_1$) after taking subsequences if necessary, there exists $E_n\subset C_n\backslash C_{n+1}$ and $\delta>0$ such that
$
\|\mathcal{E}(E_n)x\|\geq \delta.
$
The $E_n$'s are disjoint and, hence,
$$
\mathcal{E}(\cup_{n=1}^\infty E_n)x = 
\sum_{n=1}^\infty \mathcal{E}(E_n)x.
$$
However, the sum on the right-hand side cannot converge, a contradiction.
\end{proof}

\cref{new_prop} need not hold for spectral operators in general. For instance, let
$$
T=
\begin{pmatrix}
			1 & 1 \\
0 &1
		\end{pmatrix}
        \quad\text{so that}\quad
\frac{1}{n}\sum_{k=1}^n 1^{-k} T^k=\begin{pmatrix}
			1 & \frac{n+1}{2}\\
0 &1
		\end{pmatrix},
$$
which clearly does not converge. However, if we can write $T$ as $T=S+A$, where $S$ is of scalar type, $\|A\|<\sup_{\lambda\in\mathrm{Sp}(S)}|\lambda|$, and such that $T^k=S^k+A^k$ for any $k\in\mathbb{N}$ (which need not hold in general of course) then
$$
\frac{1}{n}\sum_{k=1}^n z^{-k} T^k=\frac{1}{n}\sum_{k=1}^n z^{-k} S^k+\frac{1}{n}\sum_{k=1}^n z^{-k} A^k.
$$
The operator norm of the second term converges to zero and, hence, the proposition holds in this case too.

\subsection{Consequence for Koopman operators}

The setting described in the previous subsection can be directly applied to Koopman operators. Specifically, let $X$ be a Banach space of observables on $\mathcal{X}$ for which the associated Koopman operator is bounded and spectral, with resolution $\mathcal{E}$. Often, $\Koop$ has a portion of its spectrum located on the unit circle, corresponding to on-attractor dynamics, and eigenvalues inside the open unit disk, corresponding to dissipative off-attractor dynamics \cite{Drmac-Mezic-Koopman-Schur-2024, mohr2014construction, mezic2020spectrum, DMD-Vand-Cauchy-DFT}.

Suppose that $z$ is an eigenvalue of algebraic multiplicity 1, with eigenfunction $\phi_z$. Then, for an observable $g\in X$, we can write
$
\mathcal{E}(\{z\})g = s_g\phi_z.
$
The scalar $s_g$ is known as the \emph{Koopman mode} associated with $z$. If we consider multiple observables simultaneously, say $g=(g_1,\ldots,g_k)$, we obtain a Koopman mode vector $s_g\in\mathbb{C}^k$. Note that, even after normalizing the eigenfunction, Koopman modes remain defined only up to a complex phase. Therefore, one typically compares Koopman modes for several observables by evaluating the eigenfunction $\phi_z$ at a fixed point in the state space—for example, at an initial condition from trajectory data.

If $\Koop$ is scalar-type and the eigenvalue $z$ has absolute value equal to the spectral radius of $\Koop$, we can apply \cref{new_prop} directly to compute the associated Koopman mode. Suppose instead that $z$ is an eigenvalue for which there exist finitely many points in the set $\{w\in\mathrm{Sp}:|w|>|z|\}$, all of which are eigenvalues. We may then apply \cref{new_prop} successively to these eigenvalues, subtracting the differences to ultimately obtain the Koopman mode corresponding to $z$.

However, this procedure can become numerically unstable and requires prior knowledge of these eigenvalues. If $z$ is an isolated spectral point, a better approach—when feasible—is to perform a shift-and-invert transformation and consider the scalar operator $(\Koop-z'I)^{-1}$, where $z'$ is chosen sufficiently close to (but not equal to) $z$.

For a discussion of generalized Laplace analysis applied to Koopman operators of fluid flows, see \cite{mezic2013analysis}. For invertible measure-preserving systems, where the Koopman operator is unitary, harmonic averages such as \cref{GLA_mean} were already known by Schuster \cite{schuster1897lunar} for uncovering hidden periodicities in signals; see the discussion in Wiener's classic generalized harmonic analysis \cite{wiener1930generalized}.

\subsection{Example}

We consider a large-scale, wall-resolved turbulent flow past a periodic cascade of airfoils with a stagger angle of $56.9^{\circ}$ and a one-sided tip gap. This configuration is motivated by the need to reduce noise generated by flying objects \cite{peake2012modern}. We employ a high-fidelity simulation that solves the fully nonlinear Navier--Stokes equations \cite{koch2021large} at a Reynolds number $3.88\times10^5$ and Mach number $0.07$. The dataset comprises a two-dimensional slice of the mean-subtracted pressure field, measured at $295,122$ spatial points, obtained from a single trajectory containing $797$ snapshots sampled every $2\times 10^{-5}$s.

We first apply a kernelised version of ResDMD with a Gaussian kernel \cite{colbrook2024another}, corresponding to a dictionary of $N=797$ Gaussian radial basis functions. The minimum residuals over the unit circle (corresponding to pseudospectral approximations) are shown in the top left panel of \cref{fig:aerofoil}. Several minima have been highlighted, and for four of these, we present the Koopman modes computed using generalized Laplace analysis. Here, the vector of observables corresponds to the pressure measured at each grid point, resulting in a Koopman mode represented by a $295,122$-dimensional vector plotted over the state space. Applying \cref{GLA_mean} is particularly simple for a single trajectory of data since the Koopman operator acts by shifting along the time sequence.

\begin{figure}[t]
\centering
\raisebox{-0.5\height}{\includegraphics[width=0.8\textwidth,trim={2mm 0mm 0mm 0mm},clip]{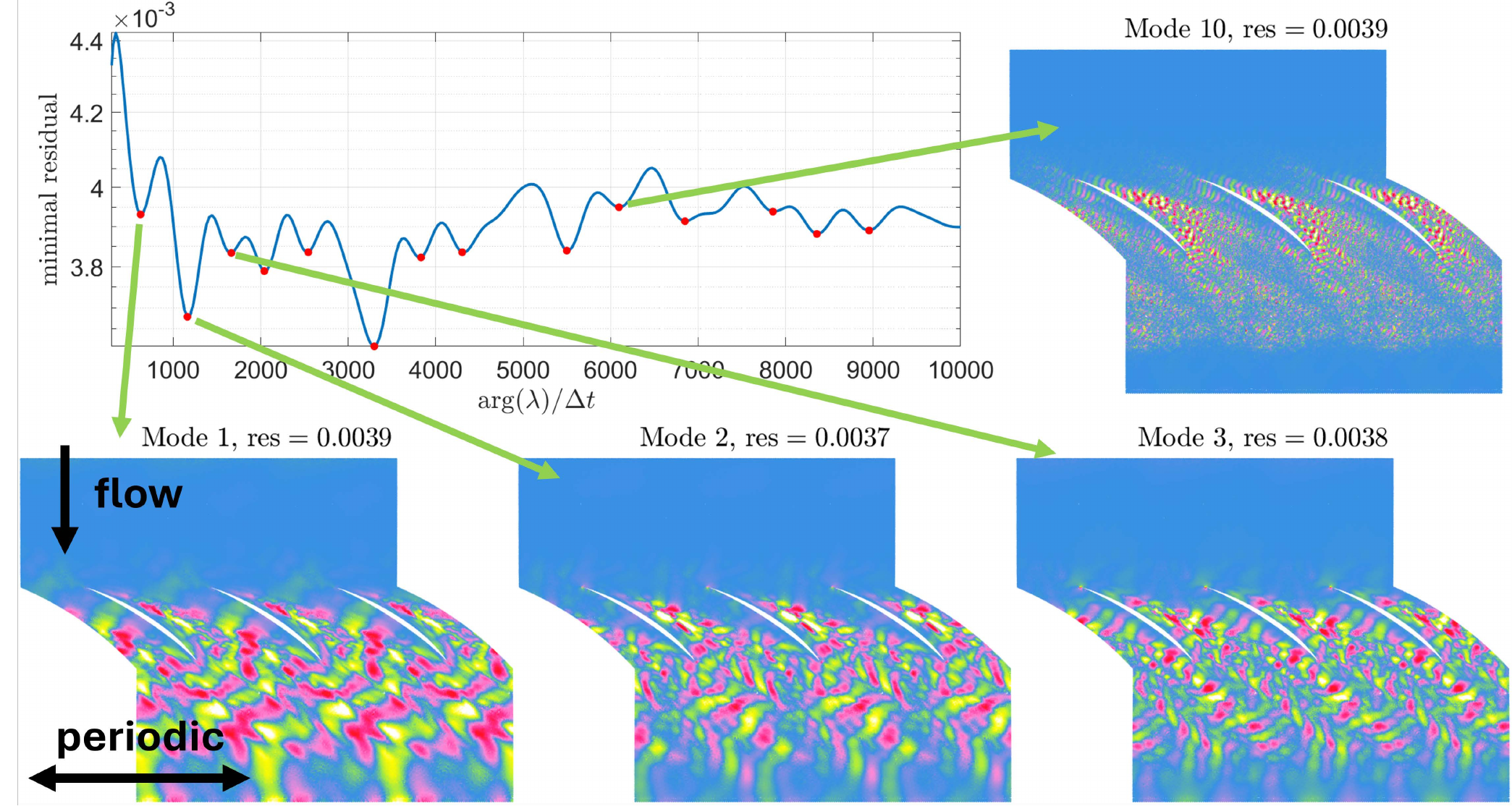}}
\caption{Koopman modes computed using generalized Laplace analysis for a turbulent flow past a periodic cascade of airfoils.}
\label{fig:aerofoil}
\end{figure}

\section{Computing spectral measures of Koopman operators}

Koopman operators associated with measure-preserving systems often exhibit continuous spectrum. This is not a pathological oddity! It arises naturally in connection with chaos and mixing phenomena, Hamiltonian structure, singular invariant manifolds, and other key dynamical features~\cite{mezic2005spectral,arbabi2017ergodic,mezic2020spectrum,korda2020data,colbrook2021rigorousKoop,colbrook2024rigged}. Unlike eigenvalues, points in the continuous spectrum are not generally linked to finite-dimensional invariant subspaces of the Hilbert space $L^2(\mathcal{X},\omega)$. As a result, Koopman operators with continuous spectra require a distinct set of tools for analysis and computation.

In this section, we examine the three leading computational paradigms for \textit{spectral measures} of the Koopman operator, which form the foundation of computational spectral analysis for unitary operators on a Hilbert space. We conclude with a brief discussion of other approaches for operators with continuous spectrum.

\subsection{Spectral measures of unitary operators}

Suppose the dynamical system is measure-preserving with respect to $\omega$ and that the dynamics are invertible $\omega$-almost everywhere. In this setting, $\mathcal{K}: L^2(\mathcal{X},\omega)\rightarrow L^2(\mathcal{X},\omega)$ is a unitary operator and its spectrum is contained in the unit circle $\mathbb{T}$. Since $\mathcal{K}$ may have a blend of point and continuous spectrum, its spectral decomposition is written as an integral against an orthogonal projection-valued measure supported on $\mathbb{T}$. For convenience, we parametrise $\mathbb{T}$ with the angle coordinate $\theta\in\pp$ through the change-of-variables $z=\exp(i\theta)$. Here, $\pp$ denotes the periodic interval $[-\pi,\pi]$ and we write integrals over $\pp$ to avoid ambiguity when measures have point masses at $\pm\pi$. Then, $\mathcal{K}$ can be diagonalized by an orthogonal projection-valued spectral measure $\pvm$ supported on $\pp$, in the sense that
\begin{equation}\label{eqn:koopman_diagonalization}
g = \int_{\pp} 1\dd\pvm(\theta)\,g, \qquad\text{and}\qquad \mathcal{K}g = \int_{\pp} e^{i\theta}\dd\pvm(\theta)\,g.
\end{equation}
Typically, one is interested in scalar spectral measures $\xi_g$ defined by the moments
\begin{equation}\label{eqn:koopman_moments}
c_n\coloneqq\langle \mathcal{K}^ng,g\rangle = \int_{\pp} e^{in\theta} \dd\xi_g(\theta), \qquad\text{for}\qquad n\in\mathbb{Z},
\end{equation}
which uniquely determines $\mathcal{K}$ on $\mathrm{span}\{g,\mathcal{K}^{\pm 1}g,\mathcal{K}^{\pm 2}g,\ldots\}\subset L^2(\mathcal{X},\omega)$. The moments of $\xi_g$ are, up to rescaling and conjugation, its Fourier coefficients. Consequently, the spectral measure $\xi_g$ associated with an observable $g$ provides a Koopman analogue of the classical power spectral density of a time series. For ergodic systems, discrete-time correlation functions coincide with inner products and $\xi_g$ is precisely the power spectral density of the time series $\{\langle\mathcal{K}^ng,g\rangle\}_{k=0}^\infty$~\cite {arbabi2017study}.

In the following subsections, we examine three families of algorithms for approximating scalar spectral measures from snapshot data:
\begin{itemize}
\item The first family approximates $\xi_g$ from the \textit{moments} in~\eqref{eqn:koopman_moments} using techniques for Fourier series, quadrature rules, and orthogonal polynomials on the unit circle.

\item The second family approximates $\xi_g$ via the \textit{eigenvalues} of finite-dimensional approximations of the Koopman operator. Structure-preserving discretizations of $\mathcal{K}$ are crucial for ensuring convergence of the resulting measures.

\item The third family approximates $\xi_g$ through finite-dimensional approximations of the Koopman operator's \textit{resolvent}, motivated by Sokhotski--Plemelj formulas that recover measures from their Cauchy transforms.
\end{itemize}
All three approaches produce a sequence of approximate measures $\tilde \xi_g^{(N,M)}$ that converge weakly to $\xi_g$ as the snapshot dimensions grow, i.e.,
\begin{equation}\label{eqn:weak_conv_meas}
\lim_{N\rightarrow\infty}\lim_{M\rightarrow\infty}\int_{\pp} \phi(\theta)\dd\tilde\xi_g^{(N,M)}(\theta) = \int_{\pp} \phi(\theta)\dd\xi_g(\theta) \qquad\forall \phi\in\Phi,
\end{equation}
for suitable classes of test functions $\Phi$.
The limit $M\rightarrow\infty$ increases the resolution of the dynamics on the state space $\mathcal{X}$ as more snapshots are collected, while $N\rightarrow\infty$ refines the resolution in observable space by enlarging the dictionary. It is \textit{essential that these limits are taken in this order} to obtain convergent approximations to $\xi_g$.

For each family, we provide convergence rates (and, in some cases, explicit error bounds) as $N\rightarrow\infty$ that improve with the regularity of the test functions in $\Phi$. This enables one to select $N$ to meet a desired accuracy and then refine $M$ accordingly. In practice, algorithms that exploit smoothness in $\Phi$ are typically more data-efficient, achieving high accuracy with smaller $N$ and often smaller $M$ as well.

Additionally, some methods from the moment-based and resolvent-based families can approximate elements of the Radon--Nikodym decomposition of $\xi_g$:
\begin{equation}\label{eqn:radon_nikodym_decomp}
\xi_g = \xi_g^{{\rm pp}} + \xi_g^{{\rm ac}} + \xi_g^{{\rm sc}},
\end{equation}
where $\xi_g^{{\rm pp}}$ is atomic, $\xi_g^{{\rm ac}}$ is absolutely continuous, and $\xi_g^{{\rm sc}}$ is singularly continuous (all with respect to the uniform measure) \cite{colbrook2019computing,korda2020data}, \cite[Section 11.6.3]{colbrookBook}. For example, if $\rho_g=\mathrm{d}\xi_g^{{\rm ac}}/\mathrm{d}\lambda$ is the Radon--Nikodym derivative of the absolutely continuous part of $\xi_g$, one may attempt to approximate $\rho_g$ directly from snapshot data. Similarly, one may aim to identify the support of the atomic component. To achieve stability and convergence with the limited state-space and snapshot resolutions associated with finite data, such approximations use regularization---both for estimating $\rho_g$ and, in some cases, for localizing the support of the atomic part. Convergence results then depend on $N$, $M$, and a regularization parameter $\sigma>0$.

Such techniques can reveal valuable spectral information about $\mathcal{K}$, but caution is needed: computing the full Lebesgue decomposition of $\xi_g$ typically requires multiple limits~\cite{colbrook2019computing}, a fact formalized in the SCI hierarchy~\cite{ben2015can,Hansen_JAMS} and recently analyzed for Koopman operators in~\cite{colbrook2024limits}.

\subsection{Family I: Moment-based methods}

In this section, we present two methods for approximating the spectral measure $\xi_g$ from a truncated sequence of moments $\{c_n\}_{n=-N}^N$. These moments can be computed from snapshot data using, for example, Monte Carlo integration, state-space quadrature, or ergodic averaging. In practice, the choice of quadrature for computing moments depends on the structure of the snapshot data collected along dynamical trajectories~\cite{colbrook2021rigorousKoop}. This is directly analogous to the choice of quadrature used to construct finite-dimensional approximations of the Koopman operator in~\cref{eq:LS:Upi}.

\subsubsection{Quadrature approximation}

The first type of moment-based approximation constructs a discrete approximation to $\xi_g$ in the form of a quadrature rule~\cite{korda2020data}. Given nodes $\theta_{-N},\ldots,\theta_N$ and weights $w_{-N},\ldots,w_N$, the measure $\xi_g$ is approximated by a weighted sum of Dirac deltas:
\begin{equation}\label{eqn:quad_approx}
    \xi_g^{(N)}(\theta)=\sum_{j=-N}^N w_j\delta(\theta-\theta_j), \qquad\text{where}\qquad \theta\in \pp.
\end{equation}
Discrete approximations to $\xi_g$ are equivalent to quadrature formulas, since
$$
\int_{\pp}\phi(\theta)\dd \xi_g^{(N)}(\theta) = \sum_{j=-N}^N w_j\phi(\theta_j).
$$
The quadrature rule is called \textit{interpolatory} if the nodes and weights are chosen to integrate all trigonometric polynomials of degree at most $N$ exactly, i.e.,
\begin{equation}\label{eqn:gqmoments}
c_k=\int_{\pp} e^{ik\theta}\dd\xi_g(\theta) = \sum_{j=-N}^N w_j e^{ik\theta_j}, \qquad\text{for}\qquad k = 0,\pm 1,\ldots,\pm N.
\end{equation}
As a result, one might expect that the quadrature approximation is accurate for integrands that are well-approximated by degree-$N$ trigonometric polynomials on the periodic interval $\pp$. This is true provided that the quadrature weights are \textit{uniformly summable}, i.e., that $\sup_{N\geq 1}\sum_{j=-N}^N|w_j|< \infty$. When this condition holds, the interpolatory quadrature rule converges for every continuous integrand and the measure $\smash{\xi_g^{(N)}}$ converges weakly to $\xi_g$ in the sense of measures. Moreover, the rate of weak convergence improves with the regularity of the test function.

\begin{theorem}\label{thm:weak_conv_gq}
    Let $\mathcal{K}:L^2(\mathcal{X},\omega)\rightarrow L^2(\mathcal{X},\omega)$ be unitary, $g\in L^2(\mathcal{X},\omega)$, and $c_n=\langle\mathcal{K}^ng,g\rangle$ for $n=-N,\ldots,N$. Given nodes $\theta_{-N},\ldots,\theta_N$ and weights $w_{-N},\ldots,w_N$ satisfying~\cref{eqn:gqmoments}, define $\xi_g^{(N)}(\theta)=\sum_{j=-N}^N w_j\delta(\theta-\theta_j)$. For a given $\phi\in C(\pp)$, consider
$$
E_N(\phi) = \left\lvert \int_{\pp} \phi(\theta)\dd\xi_g(\theta) - \int_{\pp} \phi(\theta)\dd \xi_g^{(N)}(\theta)\right\rvert,
$$
where $\xi_g$ is the scalar spectral measure of $\mathcal{K}$ in~\cref{eqn:koopman_moments}. If the weights are uniformly summable, that is, $\sup_{N\geq 1}\sum_{j=-N}^N|w_j|<\infty$, then the following hold:
\begin{itemize}
    \item[(a)] If $\phi\in C(\pp)$, then $E_N(\phi)\rightarrow 0$ as $N\rightarrow\infty$.
    \item[(b)] If $\phi\in C^{p-1}(\pp)$ and $\phi^{(p)}$ has bounded variation ($p\geq 1$), $E_N(\phi) = \mathcal{O}(N^{-p})$.
    \item[(c)] If $\phi$ is real analytic on $\pp$, then there is a $\gamma>1$ such that $E_N(\phi) = \mathcal{O}(\gamma^{-N})$.
\end{itemize}
\end{theorem}

\begin{proof}[Proof of \cref{thm:weak_conv_gq}] For any continuous function $\phi\in C(\pp)$, denote its norm by $\|\phi\|_{\pp} = \sup_{\theta\in\pp}|\phi(\theta)|$. The quadrature rule is bounded above by
$$
\left|\int_{\pp} \phi(\theta)\dd\xi_g^{(N)}(\theta)\right| = \left|\sum_{j=-N}^N w_j\phi(\theta_j)\right| \leq \|\phi\|_{\pp}\sum_{j=-N}^N |w_j|.
$$
Since the weights are uniformly summable, the uniform boundedness principle implies that the family of quadrature rules are uniformly bounded on $C(\pp)$:
$$
\sup_{N\geq 1}\sup_{\phi\in C(\pp)} \left|\int_{\pp} \phi(\theta)\dd\xi_g^{(N)}(\theta)\right|<\infty.
$$
Since the quadrature rule is also interpolatory, a standard density argument shows that the quadrature rule converges for any continuous function, which proves (a). 

To derive the convergence rates in (b) and (c), we expand $\phi$ in a Fourier series and apply the moment matching conditions in~\cref{eqn:gqmoments}. Note that if $\phi'$ has bounded variation, its Fourier coefficients satisfy $\hat\phi_k=\mathcal{O}(k^{-2})$ and its Fourier series converges uniformly and absolutely. Therefore, the approximation error satisfies
\begin{equation}\label{eqn:quadrature_error}
\begin{aligned}
E_N(\phi) &= \left\lvert \sum_{k=-\infty}^\infty \hat\phi_k \left(\int_{\pp} e^{ik\theta}\dd\xi_g(\theta) - \sum_{j=-N}^N w_je^{ik\theta_j}\right)\right\rvert \\
&\leq \sum_{|k|>N} |\hat\phi_k|\left\lvert\int_{\pp} e^{ik\theta}\dd\xi_g(\theta)\right\rvert + \sum_{|k|>N} |\hat\phi_k|\left\lvert\sum_{j=-N}^N w_je^{ik\theta_j}\right\rvert. 
\end{aligned}
\end{equation}
Now, $\xi_{\tilde g}$ with $\tilde g = g/\|g\|$ is a probability measure and  $\smash{|\xi_g|  = \|g\|^2 \int_{\pp} 1\dd\xi_{\tilde g}(\theta) = \|g\|^2}$. Since the complex exponentials have unit modulus, we have total variation bounds for the moments,
$$
\left|\int_{\pp}e^{ik\theta}\dd\xi_g(\theta)\right|\leq |\xi_g| = \|g\|^2 \qquad\forall k\in\mathbb{Z}.
$$
Applying the total variation bound for the individual moments in the first sum on the second line of~\cref{eqn:quadrature_error} and invoking the uniform summability of the weights for the second sum, we obtain the following upper bound for the quadrature error: 
\begin{equation}\label{eqn:gq_approx_bound}
E_N(\phi)\leq \left(\|g\|^2+\sup_{N'\geq 1}\sum_{j=-N'}^{N'}|w_j|\right)  \sum_{|k|>N} |\hat\phi_k| = C\sum_{|k|>N}|\hat\phi_k|,
\end{equation}
where $C$ does not depend on $N$. 
The convergence rates in (b) and (c) follow immediately from the bound in~\cref{eqn:gq_approx_bound} and standard results on the decay of Fourier coefficients for differentiable and real analytic functions~\cite[Ch.~7--8]{trefethen2019approximation}.
\end{proof}

Now, the key question: how should we choose the nodes and weights of the quadrature rule so that the moment-matching conditions in~\cref{eqn:gqmoments} are satisfied? Given any $2N+1$ distinct nodes in $\pp$ these conditions hold if the $2N+1$ unknown weights solve the complex Vandermonde system
$$
\begin{pmatrix}
    e^{-iN\theta_{-N}} & e^{-iN\theta_{-N+1}} & \cdots & e^{-iN\theta_N} \\
    e^{-i(N-1)\theta_{-N}} & e^{-i(N-1)\theta_{-N+1}} & \cdots & e^{-i(N-1)\theta_N} \\
    \vdots & \vdots & & \vdots \\
    e^{iN\theta_{-N}} & e^{iN\theta_{-N+1}} & \cdots & e^{iN\theta_N}
\end{pmatrix}
\begin{pmatrix}
    w_{-N} \\ w_{-N+1} \\ \vdots \\ w_N
\end{pmatrix}
=
\begin{pmatrix}
    c_{-N} \\ c_{-N+1} \\ \vdots \\ c_N
\end{pmatrix}.
$$
When the nodes are distinct, this Vandermonde system has a unique solution, and the resulting interpolatory quadrature rule achieves the convergence rates stated in~\cref{thm:weak_conv_gq}. In particular, if the quadrature nodes are chosen as the $(2N+1)$th roots of unity, the Vandermonde matrix becomes a scaled unitary matrix corresponding to the Discrete Fourier Transform (DFT).

A distinguished family of interpolatory quadrature formulas are the Gauss--Szegő rules, the analogue of Gauss--Jacobi quadrature for the unit circle. In Gauss--Szegő quadrature, the nodes are chosen as the roots of certain $\xi_g$-orthogonal polynomials—specifically, the paraorthogonal polynomials associated with $\xi_g$~\cite{SBarry-1,SBarry-2}. This construction allows exact integration of $2N+1$ Fourier modes, $e^{-iN\theta},\ldots,e^{iN\theta}$, using only $N$ quadrature nodes.
The nodes can be computed directly from the moments by first solving a Toeplitz linear system for the coefficients of the denominator polynomial and then performing a global root-finding procedure. To avoid the numerical difficulties inherent in root-finding, however, Korda, Putinar, and Mezić recommend instead solving an optimization problem to compute approximate Gauss--Szegő-type interpolatory rules~\cite{korda2020data}.
Alternatively, if a Koopman Krylov subspace is available—for instance, through time-delay embedding—Gauss--Szegő quadrature nodes and weights can be computed using the isometric Arnoldi algorithm~\cite{gragg1993positive,helsen2005convergence}.

\begin{figure}[t]
\centering
\includegraphics[width=0.45\linewidth]{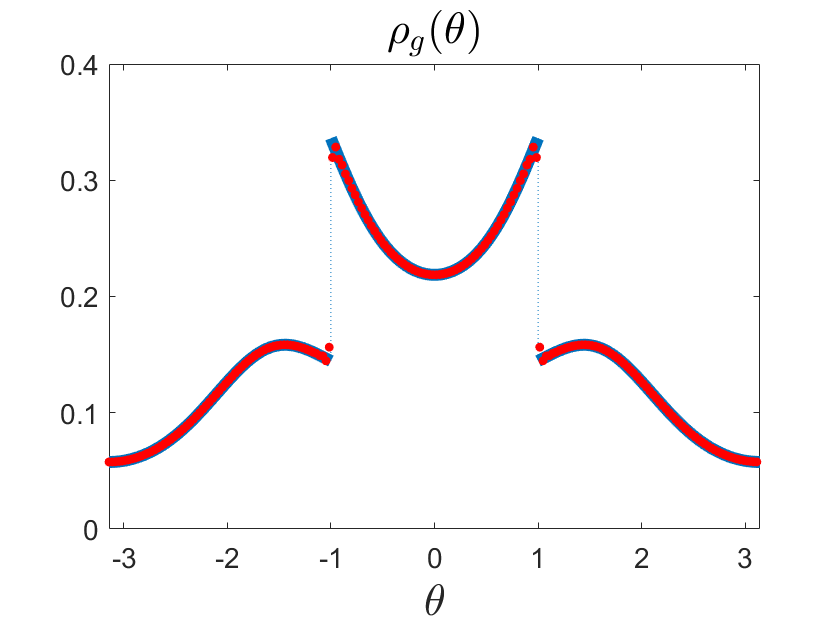}\hfill
\includegraphics[width=0.45\linewidth]{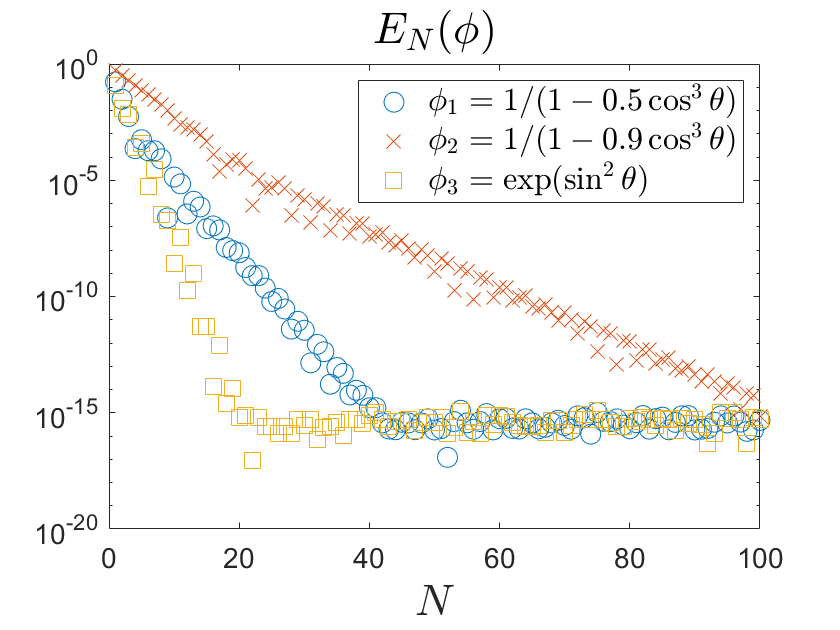}
\caption{\textit{Left:} The integrable density of the absolutely continuous measure used for moment-based numerical tests is displayed (solid blue line). The quadrature weights of the interpolatory rule for equispaced points cluster along the density after normalizing by the equilibrium (uniform) measure on the unit circle (red circles). \textit{Right:} Weak convergence for three functions of varying regularity, illustrating the rates of convergence in~\cref{thm:weak_conv_gq}. Note the classic ``kink" feature of interpolatory quadrature with fixed nodes, well-known in the setting of Clenshaw--Curtis and Fejér quadrature on the unit interval, is visible in the slowest convergence curve.}
\label{fig:interp}
\end{figure}

To illustrate the numerical properties of the interpolatory quadrature approximation, we consider an absolutely continuous measure with discontinuous density
\begin{equation}\label{eqn:test_meas}
    \rho(\theta) = 3e^{\cos\theta} + 5(1+{\rm sign}(\theta+1)-{\rm sign}(\theta-1))e^{1-\cos^2\theta}.
\end{equation}
\Cref{fig:interp} (left) shows the density (solid blue line). After rescaling the quadrature weights of the interpolatory quadrature rule by the equilibrium density $\mathrm{d}\theta/(2\pi)$ and the number of nodes, the renormalized weights (red circles) cluster along the density of $\xi_g$. \Cref{fig:interp} (right) illustrates the rates of weak convergence in~\cref{thm:weak_conv_gq} for three test functions of decreasing regularity. The least regular of these exhibits the classic ``kink" feature of interpolatory quadrature with fixed nodes, well-known in the setting of Clenshaw--Curtis and Fejér quadrature on the unit interval~\cite{weideman2007kink,trefethen2008gauss}.

\subsubsection{Fourier approximation}

The second type of moment-based approximation exploits the connection between the moments $\smash{\{c_k\}_{k=-N}^N}$ and the Fourier coefficients of $\smash{\xi_g(\theta)}$ in~\cref{eqn:koopman_moments}~\cite{korda2020data,colbrook2021rigorousKoop,arbabi2017ergodic}. Consider the truncated Fourier series
\begin{equation}\label{eqn:trig_poly_approx}
    \rho_g^{(N)}(\theta) = \frac{1}{2\pi}\sum_{k=-N}^N c_k e^{-i k \theta}, \qquad \text{where}\qquad\theta\in \pp.
\end{equation}
This truncated series provides an absolutely continuous approximation to $\xi_g$ and converges to $\xi_g$ in an appropriate weak sense as $N \to \infty$. As with the discrete quadrature approximations discussed in the previous section, the rate of weak convergence in~\cref{eqn:weak_conv_meas} depends on the regularity of the test function.

For a precise convergence statement, we work with the space of functions on the periodic interval $\pp$ whose Fourier series converge absolutely. This space, denoted $A(\pp)$ and known as the Wiener algebra, is equipped with the norm
$$
\|\phi\|_A = \sum_{k=-\infty}^\infty |\hat\phi_k|, \qquad\text{where}\qquad \hat\phi_k = \frac{1}{2\pi}\int_{\pp}\phi(\theta) e^{-ik\theta}\dd\theta.
$$
The Wiener algebra contains all Lipschitz continuous functions on $\pp$ and is a proper closed subalgebra of the continuous functions on $\pp$~\cite{katznelson2004introduction}. 

\begin{theorem}\label{thm:weak_conv_fs}
    Let $\mathcal{K}:L^2(\mathcal{X},\omega)\rightarrow L^2(\mathcal{X},\omega)$ be unitary, $g\in L^2(\mathcal{X},\omega)$, and $c_n=\langle\mathcal{K}^ng,g\rangle$ for $n=-N,\ldots,N$. Define $\rho_g^{(N)}(\theta)=\smash{\frac{1}{2\pi}\sum_{k=-N}^N c_k e^{-ik\theta}}$ and, given $\phi\in C(\pp)$, consider
$$
E_N(\phi) = \left\lvert \int_{\pp} \phi(\theta)\dd\xi_g(\theta) - \int_{\pp} \phi(\theta)\rho_g^{(N)}(\theta)\dd\theta\right\rvert,
$$
where $\xi_g$ is the scalar spectral measure of $\mathcal{K}$ in~\cref{eqn:koopman_moments}. The following hold:
\begin{itemize}
    \item[(a)] If $\phi\in A(\pp)$, then $E_N(\phi)\rightarrow 0$ as $N\rightarrow\infty$.
    \item[(b)] If $\phi\in C^{p-1}(\pp)$ and $\phi^{(p)}$ has bounded variation ($p\geq 1$), $E_N(\phi) = \mathcal{O}(N^{-p})$.
    \item[(c)] If $\phi$ is real analytic on $\pp$, then there is a $\sigma>1$ such that $E_N(\phi) = \mathcal{O}(\sigma^{-N})$.
\end{itemize}
\end{theorem}

Since the Wiener algebra $A(\pp)$ contains all Lipschitz continuous functions on $\pp$, (a) implies that the approximation $\smash{\rho_g^{(N)}}$ converges to $\xi_g$ in the Wasserstein-1 metric. However, counterexamples with divergent Fourier series show that weak convergence need not hold for all continuous bounded functions on $\pp$. Thus, the truncated Fourier series $\rho_g^{(N)}$ may fail to converge to $\xi_g$ weakly in the sense of measures, or equivalently, in the Levy--Prokhorov metric.

\begin{proof}[Proof of \cref{thm:weak_conv_fs}]
The function $\phi\in A(\pp)$ has an absolutely convergent Fourier series $\phi(\theta) = \sum_{k=-\infty}^\infty \hat\phi_k e^{ik\theta}$, with
$$
\int_{\pp} \phi(\theta)\dd\xi_g(\theta)
= \sum_{|k|\leq N} c_k\hat\phi_k + \sum_{|k|>N} \hat\phi_k\int_{\pp} e^{ik\theta}\dd\xi_g(\theta). 
$$
On the other hand, integrating $\phi$ against the approximate measure $\rho_g^{(N)}$ yields
$$
\int_{\pp} \phi(\theta)\rho_g^{(N)}(\theta)\dd\theta = \sum_{|k|\leq N} c_k \left[\frac{1}{2\pi}\int_{\pp}\phi(\theta)e^{-ik\theta}\dd\theta\right] = \sum_{|k|\leq N} c_k\hat\phi_k.
$$
Consequently, the approximation error satisfies the upper bound (c.f.~\cref{eqn:gq_approx_bound})
\begin{equation}\label{eqn:fs_approx_err}
E_N(\phi) = \left\lvert \sum_{|k|>N} \hat\phi_k\int_{\pp} e^{ik\theta}\dd\xi_g(\theta) \right\rvert \leq \|g\|^2 \sum_{|k|>N} |\hat\phi_k|.
\end{equation}
This bound is analogous to~\cref{eqn:gq_approx_bound} in the proof of~\cref{thm:weak_conv_gq} and the convergence in (a) follows since $\phi\in A(\pp)$. Similarly, the convergence rates in (b) and (c) again follow from~\cref{eqn:fs_approx_err} and the standard results on the decay of Fourier coefficients for differentiable and real analytic functions~\cite[Ch.~7--8]{trefethen2019approximation}.
\end{proof}

In practice, test functions are integrated against the approximate density $\rho_g^{(N)}$ using a numerical quadrature rule. Given nodes $\theta_1,\ldots,\theta_\ell$ and weights $w_1,\ldots,w_\ell$,
$$
\int_{\pp}\phi(\theta)\rho_g^{(N)}(\theta)\dd\theta \approx \sum_{j=1}^\ell w_j\, \phi(\theta_j)\,\rho_g^{(N)}(\theta_j).
$$
A natural choice is the periodic trapezoidal rule on the unit circle, with nodes $\theta_j = 2\pi j / \ell$ and weights $w_j = 2\pi / \ell$ for $j=1,\ldots,\ell$. If one sets $\ell = 2N+1$, then the quadrature error decreases with $N$ at a rate comparable to the rate of convergence in~\cref{thm:weak_conv_fs}. In general, the degree of the quadrature rule used should typically be increased as the number of moments used for $\rho_g^{(N)}$ in~\cref{eqn:trig_poly_approx} is increased, in order to accurately integrate a test function $\phi$ against $\smash{\rho_g^{(N)}}$.

\begin{figure}[t]
\centering
\includegraphics[width=0.45\linewidth]{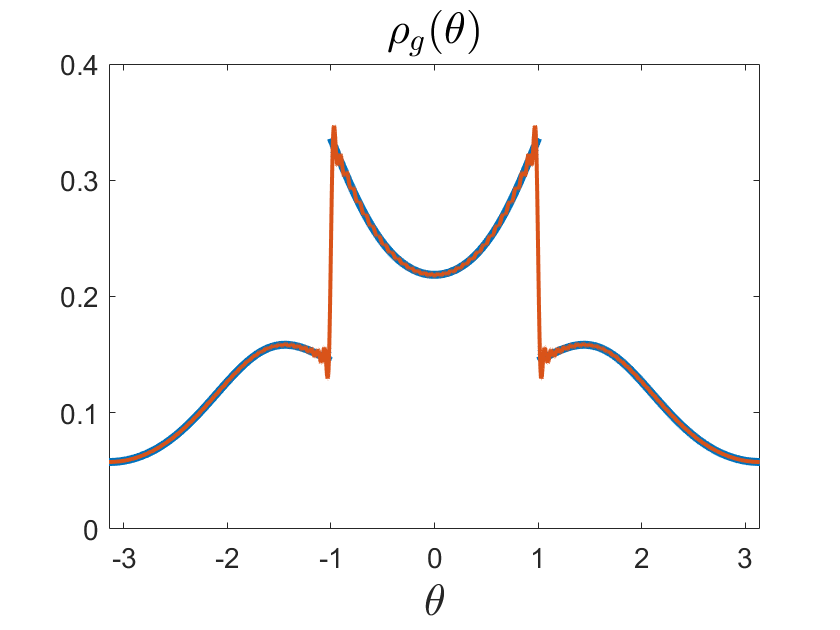}\hfill
\includegraphics[width=0.45\linewidth]{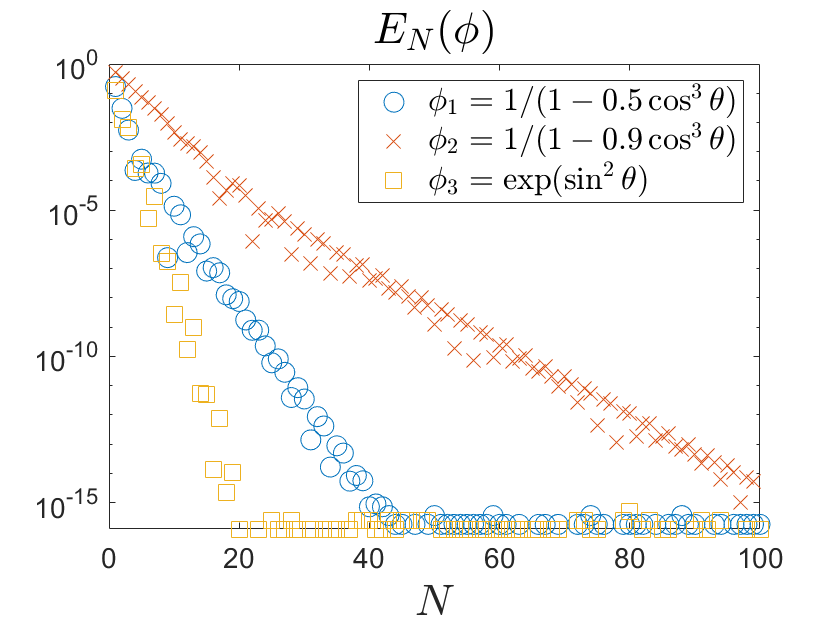}
\caption{\textit{Left:} The integrable density of the absolutely continuous measure used for moment-based numerical tests (solid blue line) is compared with the density of the truncated Fourier approximation (solid orange line). Note the classic Gibbs phenomena in the approximation, which occurs at the discontinuities of the measure. \textit{Right:} Weak convergence for three functions of varying regularity, illustrating the rates of convergence in~\cref{thm:weak_conv_fs}. The convergence curves are nearly identical to those of the interpolatory quadrature in \cref{fig:interp}, a consequence of the choice of trapezoidal rule to integrate the Fourier series against the test functions in this experiment.}
\label{fig:fourier}
\end{figure}

To compare the truncated Fourier series approximation with the interpolatory quadrature rules, we show the Fourier series analogue of the experiments in~\cref{fig:interp}. \Cref{fig:fourier} (left) compares the density (solid blue line) density in~\cref{eqn:test_meas} with the Fourier series approximation (orange line). The discontinuities in the density give rise to the classic Gibbs phenomenon, i.e., spurious oscillations of the Fourier series approximation near the discontinuities in the density, which disrupt convergence and lead to `overshoot' in the approximation. \Cref{fig:fourier} (right) demonstrates the weak convergence rates in~\cref{thm:weak_conv_gq} for the same three test functions of decreasing regularity as in~\cref{fig:interp}. Since the trapezoidal rule (a Gauss--Szegő rule for the uniform measure on the unit circle) is used to numerically integrate the truncated Fourier series against the test functions, the error curves in~\cref{fig:interp,fig:fourier} are nearly identical: they essentially reflect the accuracy of the best trigonometric polynomial approximation (up to degree $N$) of the test function.

\subsubsection{Polynomial filters and the Radon--Nikodym decomposition}

Often the spectrum of the Koopman operator is absolutely continuous on part of the unit circle. This raises the question: can one design absolutely continuous approximations to $\xi_g$ that converge locally to its Radon--Nikodym derivative $\rho_g$?

A natural starting point is the truncated Fourier series approximation $\smash{\rho_g^{(N)}}$ in~\cref{eqn:trig_poly_approx}. If $\xi_g$ is absolutely continuous on $\pp$ with density $\rho_g\in A(\pp)$, then $\smash{\rho_g^{(N)}}$ converges pointwise to $\rho_g$. As in~\cref{thm:weak_conv_fs}, the rate of pointwise convergence improves with the regularity of $\rho_g$. However, if $\xi_g$ has a singular component, or if the Radon--Nikodym derivative itself has singularities on $\pp$, the truncated Fourier series may fail to converge pointwise.

To obtain locally convergent pointwise approximations, one applies a carefully chosen \textit{filter} to the Fourier coefficients:
\begin{equation}
     \rho_{g,\nu}^{(N)}(\theta) = \frac{1}{2\pi}\sum_{k=-N}^N c_k\nu\left(\frac{k}{N}\right) e^{-i k \theta}.
\end{equation}
The filter $\nu:[-1,1]\rightarrow\mathbb{R}$ is designed to preserve the low-frequency content of $\rho_g$ while damping high-frequency components by tapering smoothly to zero as $|k|\rightarrow N$.

Intuitively, the role of the filter is to smooth $\rho_g^{(N)}$ so that the resulting function can be stably approximated by a truncated Fourier series. By the convolution theorem,
\begin{equation}\label{eqn:poly_kern}
    \rho_{g,\nu}^{(N)}(\theta) = \int_{\pp} K_\nu(\tilde\theta-\theta)\dd\xi_g(\tilde\theta), \qquad\text{where}\qquad K_\nu(\theta) = \sum_{k=-N}^N \nu\left(\frac{k}{N}\right)e^{ik\theta}.
\end{equation} 
Here, the kernel $K_\nu$ associated with the filter $\nu$ is designed to approximate the identity operator to recover $\xi_g$. At the same time, because $K_\nu$ is a degree-$N$ trigonometric polynomial in $\theta$, the convolution both smooths $\xi_g$ and truncates its Fourier series.

A classic example of such a kernel is the Fejér kernel, corresponding to Cesàro summation of Fourier series. The associated filter is the ``hat" function $\nu(x)= 1 - |x|$. In the context of Koopman spectral analysis, Korda, Putinar, and Mezić employed the Fejér kernel to approximate the cumulative distribution function of $\xi_g$, $\smash{\chi(\theta) = \int_0^{\theta}\dd\xi_g(\tilde\theta)}$, pointwise at its continuity points~\cite{korda2020data}. Colbrook and Townsend established rigorous criteria—drawing on ideas from signal processing, density estimation, and spectral computation—for constructing filters that achieve high-order pointwise approximation of $\rho_g$ as $N \to \infty$~\cite{colbrook2021rigorousKoop}. The essence of the idea is to construct a polynomial filter whose smoothing kernel approximates a Dirac delta function by matching moments up to order $m$.

In what follows, we write $\phi\in\mathcal{C}^{n,\alpha}(\mathcal{I})$ to denote a function with $n \in \mathbb{N}\cup\{0\}$ continuous derivatives on an interval $\mathcal{I}$ and an $\alpha$-H\"older continuous $n$th derivative.

\begin{theorem}[Convergence of filtered Fourier series]
\label{thm:filter_convergence}
Given integer $m>0$, let $\nu:[-1,1]\rightarrow\mathbb{R}$ be an even continuous function that satisfies (i) $\nu\in C^{(m)}([-1,1])$, (ii) $\nu(0)=1$ and $\nu^{(n)}(0) = 0$ for integers $1\leq n\leq m-1$, (iii) $\nu^{(n)}(\pm 1) = 0$ for integers $0\leq n\leq m-1$, and (iv) $\nu|_{[0,1]}\in C^{(m+1)}([0,1])$. Let $\xi$ be a Borel measure with total variation $\|\xi\|<\infty$ that is absolutely continuous on the closed interval $\mathcal{I}=[\theta_0-\eta,\theta_0+\eta]$, for some $\theta_0\in\pp$ and $\eta\in(0,\pi)$, with Radon--Nikodym derivative $\rho\in\mathcal{C}^{n,\alpha}(\mathcal{I})$. Then
$$
\left|\rho(\theta_0){-}[K_\nu*\xi](\theta_0)\right|{\lesssim}
\frac{N^{-m}\|\xi\|}{(N^{-1}+\eta)^{m+1}}{+}\begin{cases}
\|\rho\|_{\mathcal{C}^{n,\alpha}(\mathcal{I})}(1{+}\eta^{-(n+\alpha)})N^{-n-\alpha},\!\!&\text{if }n+\alpha<m,\\
\|\rho\|_{\mathcal{C}^{m}(\mathcal{I})}(1{+}\eta^{-m})N^{-m}\log(1{+}N), \!\!&\text{if }n+\alpha\geq m.
\end{cases}
$$
Here, $K_\nu$ is the smoothing kernel associated with the filter $\nu$ defined in~\cref{eqn:poly_kern}.
\end{theorem}

\begin{figure}[t]
\centering
\includegraphics[width=0.45\linewidth]{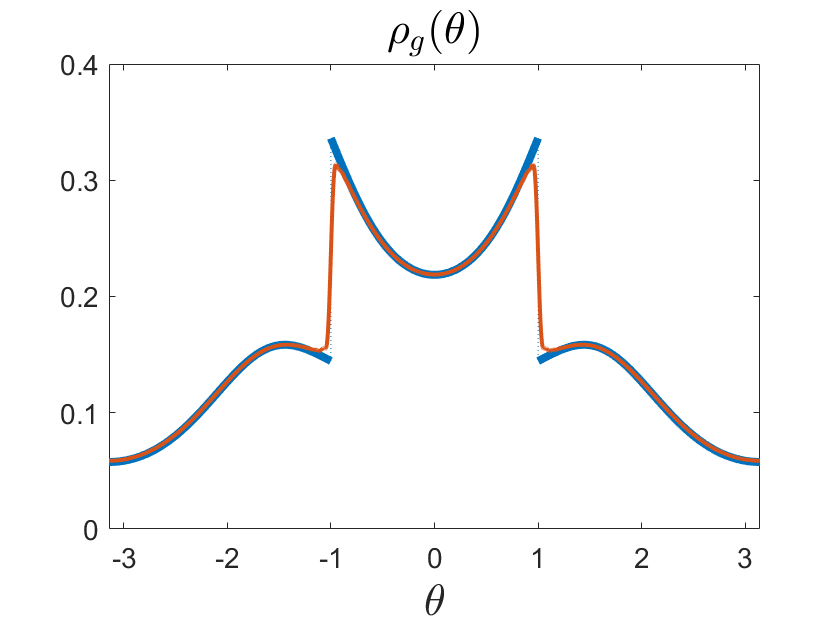}\hfill
\includegraphics[width=0.45\linewidth]{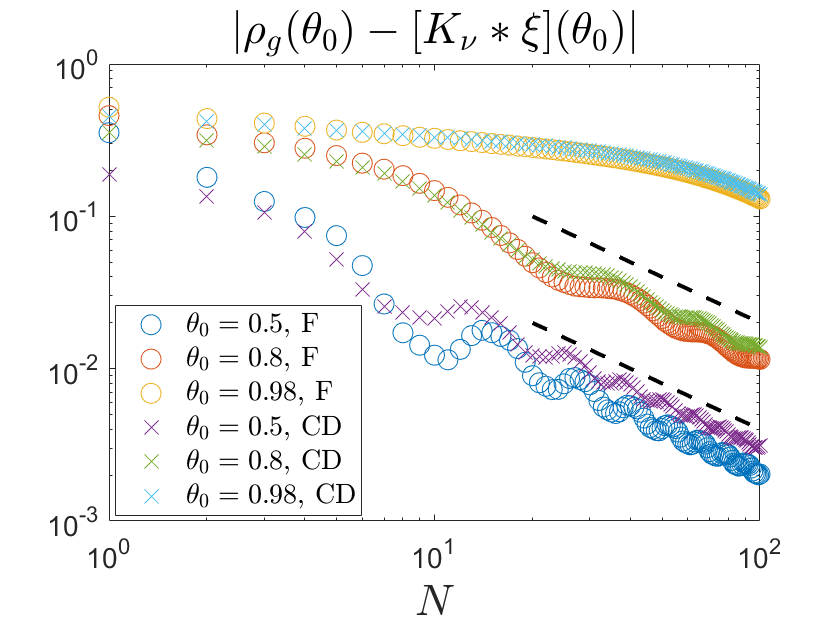}
\caption{\textit{Left:} The integrable density of the absolutely continuous measure used for moment-based numerical tests (solid blue line) is compared with the hat-filtered Fourier approximation (solid orange line). Note the classic Gibbs phenomena, which appeared in Figure~11, has been mitigated by the filter (at the expense of high-order convergence). \textit{Right:} Pointwise convergence at three points located an increasing distance from the discontinuity in the measure at $\theta= 1$, illustrating the rates (dashed black lines delineate $\mathcal{O}(1/N)$ curves as $N\rightarrow\infty$) in~\cref{thm:weak_conv_gq} for the first-order hat filter (circles). Note the growth of constants in the bounds near singular points of the density. The corresponding convergence curves obtained using the Christoffel--Darboux (CD) kernel approach in~\cite{korda2020data} (displayed with `x' markers) are nearly identical to those of the hat filter, which corresponds to the Fejér (F) smoothing kernel.}
\label{fig:filter}
\end{figure}

The Christoffel--Darboux kernel associated with the measure $\xi_g$ can also be used to construct pointwise approximations to $\rho_g$ from its moments~\cite{korda2020data}. This approach may be viewed as a smoothing kernel (with $m=1$ first-order convergence) induced by a family of $\xi_g$-orthogonal polynomials, via the variational characterization of the Christoffel--Darboux kernel~\cite{simon2008christoffel}. In fact, in the special case of the uniform measure on $\pp$, the Christoffel--Darboux method is mathematically equivalent to Fourier filtering with the Fejér kernel~\cite{simon2008christoffel}.

To illustrate the effect of filtering,~\cref{fig:filter} (left panel) compares the integrable density of the absolutely continuous measure from~\cref{fig:interp,fig:fourier} with the hat-filtered Fourier approximation. In contrast to~\cref{fig:fourier}, the filter has eliminated the Gibbs phenomena caused by the discontinuities at $\pm 1$. The trade-off is that the smoothed density converges point-wise rather slowly on the periodic interval, with the error decreasing proportionally to $1/N$. The first-order point-wise convergence of the hat-filter is demonstrated in the right panel of~\cref{fig:filter} at three points located at increasing distances from the discontinuities of $\rho_g$. The convergence curves for the Christoffel--Darboux (CD) kernel~\cite{korda2020data} are close to those of the hat filter.

\begin{figure}[t]
\centering
\includegraphics[width=0.45\linewidth]{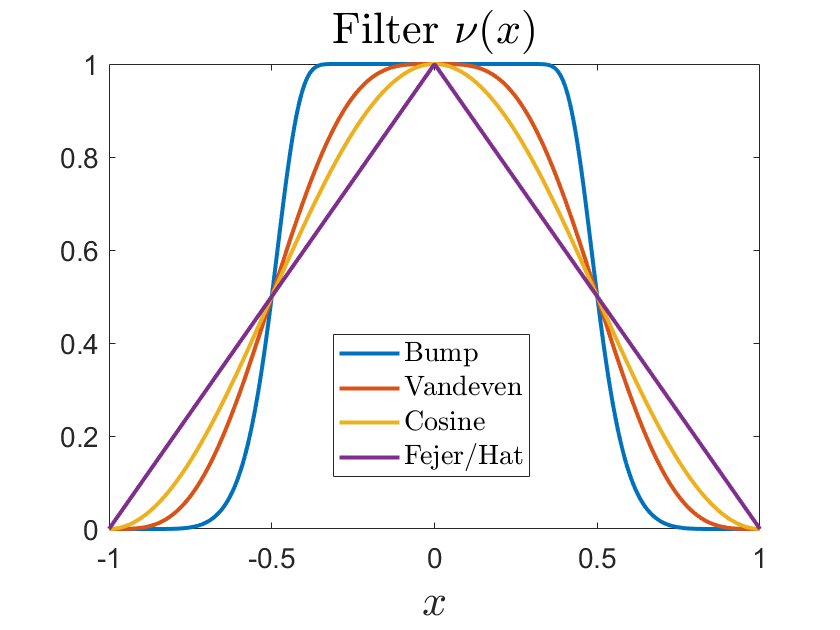}\hfill
\includegraphics[width=0.45\linewidth]{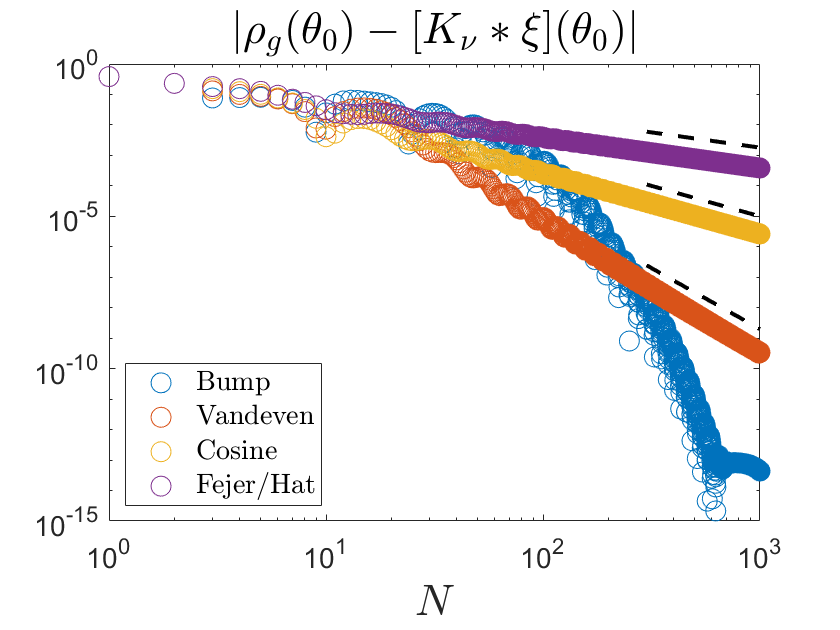}
\caption{\textit{Left:} Four filters for filtered Fourier series approximation: the first-order hat (Fejér) filter (purple), the second-order cosine filter (yellow), the fourth-order Vandeven filter (red), and the infinite-order bump filter (blue). \textit{Right:} The pointwise error in the hat filter (purple) approximaiton at $\theta = 0.6$ is compared with the pointwise error with the three higher-order filters (the color code is the same as in the left panel). Dashed lines denote, from top to bottom, rates of $\mathcal{O}(N^{-p})$ for $p=1$, $2$, and $4$ as $N\rightarrow\infty$. The bump filter converges faster than any polynomial in $1/N$ as $N\rightarrow\infty$.}
\label{fig:filter2}
\end{figure}

The left panel of~\cref{fig:filter2} shows the hat (Fejér) filter along with three filters that acheive higher-order pointwise convergence rates for smooth densities (see~\cref{thm:filter_convergence}). The second-order cosine filter (yellow) is defined as $\nu_{\rm cos}(x)=(1+\cos(\pi x))/2$~\cite{gottlieb1997gibbs} and the fourth-order Vandeven filter (red) is defined as $\nu_{\rm Vand}(x) = 1-35|x|^4 + 84|x|^5 - 70|x|^6 + 20|x|^7$~\cite{vandeven1991family}. The ``bump" filter (blue) recommended by Colbrook and Townsend~\cite{colbrook2021rigorousKoop} is a filter of infinite-order, whose convergence rate is limited only by the local smoothness of the Radon--Nikodym derivative. If the density is locally real analytic, it converges faster than any polynomial in $1/N$ as $N\rightarrow\infty$. It is defined by
$$
\nu_{\rm bump}(x) = \exp\left[-\frac{2}{1-|x|}\exp\left(-\frac{c}{|x|^4}\right)\right], \qquad\text{where}\qquad c = 0.109550455106347.
$$
The constant $c$ is chosen so that $\nu_{\rm bump}(1/2)=1/2$. The right panel of~\cref{fig:filter2} illustrates the potential advantages of higher-order filters by demonstrating their high-order pointwise approximation of the density from~\cref{fig:interp,fig:fourier} at $\theta=0.6$.

\subsection{Family II: Eigenvalue-based methods}

Another approach to approximating spectral measures is to build a finite-dimensional approximation of the Koopman operator $\mathcal{K}$ and then use its eigenvalues and eigenvectors to approximate the spectral measure of $\mathcal{K}$. A key advantage of this approach is that it does not require access to a single long trajectory of moments, which may be difficult to obtain in practice. Instead, we can use snapshot data, consisting of input-output pairs collected from trajectories of arbitrary length (short or long).

EDMD itself generally fails to produce convergent approximations of the spectral measures $\xi_g$. The root of the problem is that standard finite-dimensional Galerkin approximations of unitary operators are not necessarily unitary, and hence their eigenvalues do not, in general, lie on the unit circle. For example, consider the unitary shift operator on $\ell^2(\mathbb{Z})$, the space of square-summable sequences, defined by $\mathcal{S}(e_n) = e_{n-1}$ on the canonical basis vectors $\{e_n\}_{n=-\infty}^\infty$. The spectrum of $\mathcal{S}$ is absolutely continuous and coincides with the unit circle $\mathbb{T}$. However, the Galerkin approximation of $\mathcal{S}$ on ${\rm span}\{e_{-N},\ldots,e_N\}$ is a $(2N+1)\times(2N+1)$ Jordan block, whose spectrum is a single eigenvalue at the origin with algebraic multiplicity $2N+1$.

To fix this problem, one can use a variant of EDMD called measure-preserving EDMD (mpEDMD) \cite{colbrook2023mpedmd}, which computes the \textit{best unitary approximation} to the Koopman operator on a finite-dimensional subspace spanned by the snapshot data and dictionary of choice. \Cref{alg:mpEDMD} shows how to compute the matrix representation, $\Kv_\mathrm{mp}$, of this unitary approximation in the basis defined by the columns of the data matrix $\Psiv_X$. The matrix $\Kv_\mathrm{mp}$ is unitary on the Hilbert space $\mathbb{C}^N$ with the inner product $\langle\cdot,\cdot\rangle_{\Gv}$ induced by the Gram matrix $\Gv= \Psiv_X^*W\Psiv_X$. Assuming that the quadrature rule converges, one can show that every limit point of the matrices $\Kv_{\mathrm{mp}}$ as $M\rightarrow\infty$ corresponds to the unitary part of a polar decomposition of $\mathcal{P}_{V_{N}}\mathcal{K}\mathcal{P}_{V_{N}}^*$ (recall that $\mathcal{P}_{V_N}$ is the orthogonal projector onto the dictionary span $V_N$). Moreover, if $g=\Psiv \gv$ and $\mathcal{K}g\in V_N$, then $\lim_{M\rightarrow\infty} \Psiv\Kv_{\mathrm{mp}} \gv=\mathcal{K}g$. 

\begin{algorithm}[t]
\textbf{Input:} Snapshot data $\{(x^{(m)},y^{(m)})\}_{m=1}^M$, quadrature weights $\{w_m\}_{m=1}^{M}$, dictionary $\{\psi_j\}_{j=1}^{N}$.\\
\vspace{-4mm}
\begin{algorithmic}[1]
\STATE Compute $\Psiv_X,\Psiv_Y\in\mathbb{C}^{M\times N}$ from~\cref{eq:psidef} and $\Wv=\mathrm{diag}(w_1,\ldots,w_{M})$.
\STATE Compute an economy (pivoted) QR decomposition $\Wv^{1/2}\mathbf{\Psi}_X=\Qv\Rv\Pv^\top$.
\STATE Compute an SVD of $(\Pv\Rv^{-1})^{*}\Psiv_Y^*\Wv^{1/2}\Qv=\Uv_1\mathbf{\Sigma} \Uv_2^*$.
\STATE Compute the eigendecomposition $\Uv_2\Uv_1^*=\hat{\Vv}\mathbf{\Lambda} \hat{\Vv}^*$ (via a Schur decomposition).
\STATE Compute $\Kv_{\mathrm{mp}}=\Pv\Rv^{-1}\Uv_2\Uv_1^*\Rv\Pv^\top$ and $\Vv=\Pv\Rv^{-1}\hat{\Vv}$.
\end{algorithmic} \textbf{Output:} Koopman matrix $\Kv_{\mathrm{mp}}$, with eigenvectors $\Vv$ and eigenvalues $\mathbf{\Lambda}$.
\caption{The mpEDMD algorithm.}
\label{alg:mpEDMD}
\end{algorithm}

To approximate the projection-valued spectral measure of $\mathcal{K}$ using mpEDMD, we consider the spectral measure $\mathcal{E}_{N,M}$ of the matrix $\Kv_{\mathrm{mp}}$ on $\mathbb{C}^N$ with the inner product $\langle\cdot,\cdot\rangle_{\Gv}$:\index{projection-valued spectral measure!Koopman operators}
\begin{equation}
\label{eq:mpEDMD_proj_spec_meas}
\mathrm{d}\mathcal{E}_{N,M}(\theta)=\sum_{j=1}^{N}\mathbf{v}_j\mathbf{v}_j^*\Gv\delta(e^{i \theta}-\lambda_j)\,\mathrm{d}\theta.
\end{equation}
For scalar-valued spectral measures with respect to $g\in L^2(\mathcal{X},\omega)$, suppose that $\gv_{N,M}\in\mathbb{C}^N$ with $\lim_{N\rightarrow\infty}\lim_{M\rightarrow\infty}\Psiv \gv_{N,M}=g$. We approximate $\xi_g$ by $\smash{\xi_{\gv}^{(N,M)}}$, where\index{scalar-valued spectral measure!Koopman operators}
\begin{equation}
\label{eq:mpEDMD_scalar_spec_meas}
\mathrm{d}\xi_{\gv}^{(N,M)}(\theta)=\sum_{j=1}^N\delta(e^{i \theta}-\lambda_j)|\mathbf{v}_j^*\Gv\gv_{N,M}|^2\mathrm{d}\theta.
\end{equation}
If $\|g\|=1$, we normalise $\gv_{N,M}$ so that $\gv_{N,M}^*\Gv\gv_{N,M}=1$. Since $\{\Gv^{1/2}\mathbf{v}_j\}_{j=1}^N$ is a $\langle\cdot,\cdot\rangle_{\Gv}$ orthonormal basis for $\mathbb{C}^N$, $\smash{\xi_{\gv}^{(N,M)}}$ is a probability measure on $\pp$.

The following theorem summarizes the convergence properties of mpEDMD, based on~\cite{colbrook2023mpedmd}. The Wasserstein-1 metric is defined by
$$
W_1(\mu,\nu)=\sup_{\substack{\phi:\pp\rightarrow \mathbb{R}\\|\phi|_{\mathcal{C}^{0,1}_{\mathrm{per}}}\leq 1}}\int_{\pp} \phi(\theta) \dd (\mu-\nu)(\theta),
$$
where $\mu$ and $\nu$ are Borel probability measures on $\pp$, and that weak convergence of measures is equivalent to convergence with respect to this metric.

\begin{theorem}[Convergence properties of mpEDMD]
Suppose $\mathcal{K}$ is an isometry, $\lim_{N\rightarrow\infty}\mathrm{dist}(h,V_{N})=0$ for all $h\in L^2(\mathcal{X},\omega)$, and the quadrature rule underlying \cref{eq:LS:Upi} converges. Let $g\in L^2(\mathcal{X},\omega)$ satisfy $\|g\|=1$ and let $\gv_{N,M}\in\mathbb{C}^N$ satisfy $\gv_{N,M}^*\Gv\gv_{N,M}=1$ with $\lim_{N\rightarrow\infty}\lim_{M\rightarrow\infty}\Psiv \gv_{N,M}=g$. Then,
\begin{itemize}
	\item The scalar-valued spectral measures converge:
	$$
	\lim_{N\rightarrow\infty}\limsup_{M\rightarrow\infty}W_1(\xi_{g},\xi_{\gv}^{(N,M)})=0.
	$$
	\item If $\mathcal{K}$ is unitary, the functional calculi converge:
	$$
	\lim_{N\rightarrow\infty}\limsup_{M\rightarrow\infty}\left\|\int_{\pp} \phi(\theta)\dd\big[\mathcal{E}(\theta)g-\Psiv\mathcal{E}_{N,M}(\theta) \gv_{N,M}\big]\right\|=0,
	$$
	for every continuous function $\phi:\pp\rightarrow\mathbb{C}$.
	\item If $\{g,\mathcal{K}g,\ldots, \mathcal{K}^Lg\}\subset V_N$ for some $L\in\mathbb{N}$ and $\lim_{M\rightarrow\infty}\Psiv \gv_{N,M}=g$, then
$$
\limsup_{M\rightarrow\infty}W_1(\xi_{g},\xi_{\gv}^{(N,M)})\leq \frac{\pi}{L+1}.
$$
\end{itemize}
\end{theorem}
Remarkably, the final part of the theorem provides an explicit error bound for dictionaries with delay embeddings. The proof is based on moment-matching properties of mpEDMD and is similar to the proofs of~\cref{thm:weak_conv_gq,thm:weak_conv_fs}. Consequently, the spectral measures of mpEDMD matrices also achieve high-order convergence rates.

In the continuous-time setting, Giannakis and Valva have developed spectrally consistent finite-dimensional approximations for skew-symmetric continuous-time Koopman generators. The spectral measures of these approximations converge to those of the Koopman generator in the large-data and large-dictionary limits~\cite{valva2024physics}.

\subsection{Family III: Resolvent-based methods}

The third approach to approximating spectral measures uses the resolvent of the Koopman operator to construct carefully smoothed approximations of $\xi_g$~\cite{colbrook2021rigorousKoop,colbrook2024rigged}. These methods apply to general snapshot data and are particularly robust to noise due to principled regularization. They also achieve local high-order accuracy for the Radon--Nikodym derivative wherever it is smooth.

To illustrate, consider the \textit{Carathéodory function} associated with $\xi_g$, defined by
\begin{equation}
\label{eq:carath_def}
F_{\xi_g}(z)=\int_{\pp}\frac{e^{i\varphi}+z}{e^{i\varphi}-z}\dd\xi_{g}(\varphi)=\langle(\mathcal{K}+zI)(\mathcal{K}-zI)^{-1}g,g\rangle,\quad |z|\neq 1.
\end{equation}
The second equality follows from the Borel functional calculus for unitary operators applied to $\mathcal{K}$.
Letting $z=re^{i\theta}$ with $r=1/(1+\epsilon)<1$ one obtains
\begin{equation}
\label{eq:carath_poisson}
\frac{1}{4\pi}\left[F_{\xi_g}(re^{i\theta})-F_{\xi_g}(r^{-1}e^{i\theta})\right]=\frac{1}{2\pi}\int_{\pp}\frac{(1-r^2)\dd\xi_{g}(\varphi)}{1+r^2-2r\cos(\theta-\varphi)}=[K^{\rm (P)}_\epsilon*\xi_g](\theta),
\end{equation}
where the right-hand side is the convolution of $\xi_g$ with the Poisson kernel $K^{\rm (P)}$ for the unit disk~\cite{katznelson2004introduction}. In other words, one can compute a smoothed data-driven approximation of $\xi_g$ by approximating the Koopman operator and its resolvent in~\cref{eq:carath_def} from snapshot data. This can be achieved using the mpEDMD matrix introduced in the previous section, with convergence to $F_{\xi_g}(z)$ for any $z\not\in\mathbb{T}$~\cite{colbrook2024rigged}. To approximate $\xi_g$ accurately this way, one typically needs to take $r$ very close to $1$, which requires evaluating the resolvent of $\mathcal{K}$ near its spectrum. Accurate computation of the resolvent near the continuous spectrum generally demands large dictionaries, while larger dictionaries typically require larger amounts of data for reliable quadrature.

To design data-efficient resolvent-based schemes, one can replace the Poisson kernel with higher-order smoothing kernels. The idea is to construct periodic rational kernels that reproduce as many moments of a Dirac delta distribution as possible. High-order smoothing kernels in the Koopman setting were introduced in~\cite{colbrook2021rigorousKoop}, with connections to Fourier filters. For a streamlined presentation of the theory via Carathéodory functions, see ~\cite{colbrook2024rigged}. See also~\cite{colbrook2021computing,colbrook2025computing} for related applications of high-order rational smoothing kernels in spectral computations for self-adjoint operators.

To demonstrate the construction, consider a periodic kernel of the form
\begin{equation}\label{eq:rat_kern_R}
K_\epsilon^{\mathbb{T}}(\theta)=\frac{1}{\epsilon}\sum_{n\in\mathbb{Z}} K\left(\frac{\theta+2\pi n}{\epsilon}\right), \quad\text{where}\quad
K(x)=\frac{1}{2\pi i}\sum_{j=1}^{m}\left(\frac{\alpha_j}{x-a_j}-\frac{\beta_j}{x-b_j}\right),
\end{equation}
Now, select the poles of $K(x)$ to be equally spaced poles in $[-1,1]$, with
\begin{equation}\label{eq:equi_poles}
a_j=\frac{2j}{m+1}-1+i, \quad b_j=\overline{a_j}, \quad 1\leq j\leq m,
\end{equation}
and solve the following Vandermonde equations for the residues $\alpha_j = \overline{\beta_j}$,
\begin{equation}\label{eq:vandermonde_condition}
\begin{pmatrix}
1 & \dots & 1 \\
a_1 & \dots & a_{m} \\
\vdots & \ddots & \vdots \\
a_1^{m-1} &  \dots & a_{m}^{m-1}
\end{pmatrix}
\begin{pmatrix}
\alpha_1 \\ \alpha_2\\ \vdots \\ \alpha_{m}
\end{pmatrix}\!
=
\begin{pmatrix}
1 & \dots & 1 \\
b_1 & \dots & b_{m} \\
\vdots & \ddots & \vdots \\
b_1^{m-1} &  \dots & b_{m}^{m-1}
\end{pmatrix}
\begin{pmatrix}
\beta_1 \\ \beta_2\\ \vdots \\ \beta_{m}
\end{pmatrix}
=\begin{pmatrix}
1 \\ 0 \\ \vdots \\0
\end{pmatrix}.
\end{equation}
The Vandermonde condition on the residues ensures that the moments of the kernel match those of the Dirac delta up to order $m$. The rational form of the kernel then allows us to apply the Borel functional calculus for unitary operators, yielding an analogue of~\cref{eq:carath_poisson} that links the smoothed measure to Carathéodory functions:
\begin{equation}\label{eqn:mth_order_stone}
\left[K_\epsilon^{\mathbb{T}}*\xi_g\right](\theta)=\frac{-1}{2\pi}\sum_{j=1}^{m}{\rm Re}\left(\alpha_j \left\langle (\mathcal{K}-e^{i\theta-i\epsilon a_j}I)^{-1}g,(\mathcal{K}+e^{i\theta-i\epsilon a_j}I)^*g\right\rangle  \right).
\end{equation}
The resulting smoothed approximation converges rapidly as $\epsilon \downarrow 0$, both in the weak sense and locally pointwise wherever $\xi_g$ is absolutely continuous with a sufficiently regular Radon--Nikodym derivative~\cite{colbrook2021rigorousKoop}. 

\begin{theorem}[Convergence of smoothed measures]
\label{thm:unitary_weak_convergence}
Let $K_{\epsilon}^\mathbb{T}$ be the kernel defined in~\cref{eq:rat_kern_R,eq:equi_poles,eq:vandermonde_condition} and let $\xi$ be a Borel measure on $\pp$ with total variation $\|\xi\|$.
\begin{itemize}

\item If $\phi\in\mathcal{C}^{n,\alpha}(\pp)$, then
$$
\!\!\!\!\!\!\left|\int_{\pp}\!\!\!\!\!\!\!\!\!\!\phi(\theta)\left[\dd\xi(\theta)  - [K_\epsilon^\mathbb{T}*\xi](\theta)\dd\theta\right]\right|\lesssim \begin{cases}
\|\phi\|_{\mathcal{C}^{n,\alpha}(\pp)}\epsilon^{n+\alpha}, &\text{if }n+\alpha<m,\\
\|\phi\|_{\mathcal{C}^{m}(\pp)}\epsilon^{m}\log(1+\epsilon^{-1}), &\text{if }n+\alpha\geq m.
\end{cases}
$$

\item If for some $\theta_0\in\pp$ and $\eta\in(0,\pi)$, $\xi$ is absolutely continuous on the closed interval $\mathcal{I}=[\theta_0-\eta,\theta_0+\eta]$ with Radon--Nikodym derivative $\rho\in\mathcal{C}^{n,\alpha}(\mathcal{I})$, then
$$\!\!\!\!\!\!
\left|\rho(\theta_0){-}[K_\epsilon^\mathbb{T}*\xi](\theta_0)\right|\lesssim
\frac{\epsilon^m\|\xi\|}{(\epsilon+\eta)^{m+1}}{+}\begin{cases}
\|\rho\|_{\mathcal{C}^{n,\alpha}(\mathcal{I})}(1{+}\eta^{-(n+\alpha)})\epsilon^{n+\alpha},\quad\!\!\!\!\!\!\!\!&\text{if }n+\alpha<m,\\
\|\rho\|_{\mathcal{C}^{m}(\mathcal{I})}(1{+}\eta^{-m})\epsilon^{m}\log(1{+}\epsilon^{-1}),\quad\!\!\!\!\!\!\!\! &\text{if }n+\alpha\geq m.
\end{cases}
$$
\end{itemize}
\end{theorem}

The proof of weak convergence is based on a Taylor expansion of $\phi$ together with the delta-moment-matching conditions for the $m$th-order kernel $K_\epsilon^\mathbb{T}$, combined with a rigorous estimate for the Taylor remainder of functions with H\"older-continuous derivatives. Analogously, the proof of pointwise convergence is based on a Taylor expansion of the density $\rho$ about the point $\theta_0$~\cite{colbrook2021rigorousKoop,colbrook2024rigged}.

\begin{figure}[t]
\centering
\includegraphics[width=0.45\linewidth]{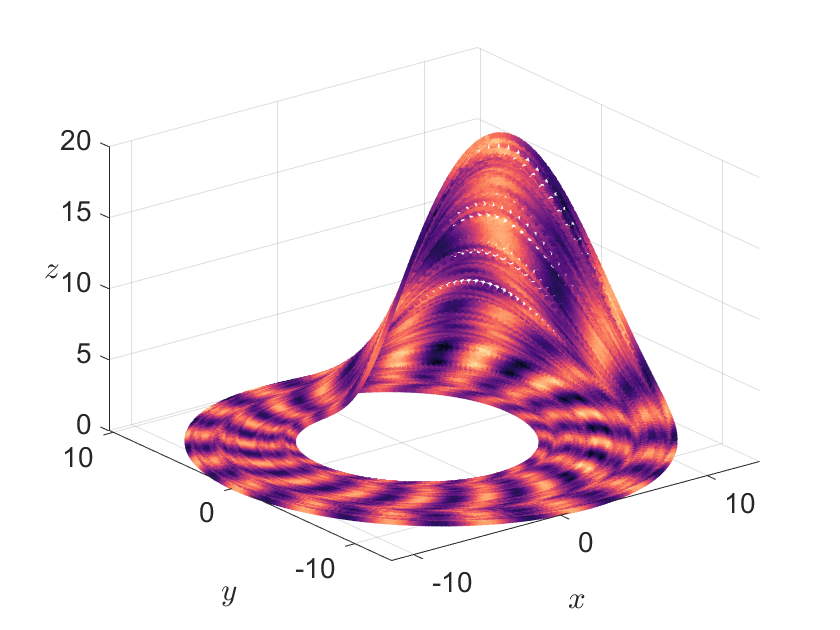}\hfill
\includegraphics[width=0.45\linewidth]{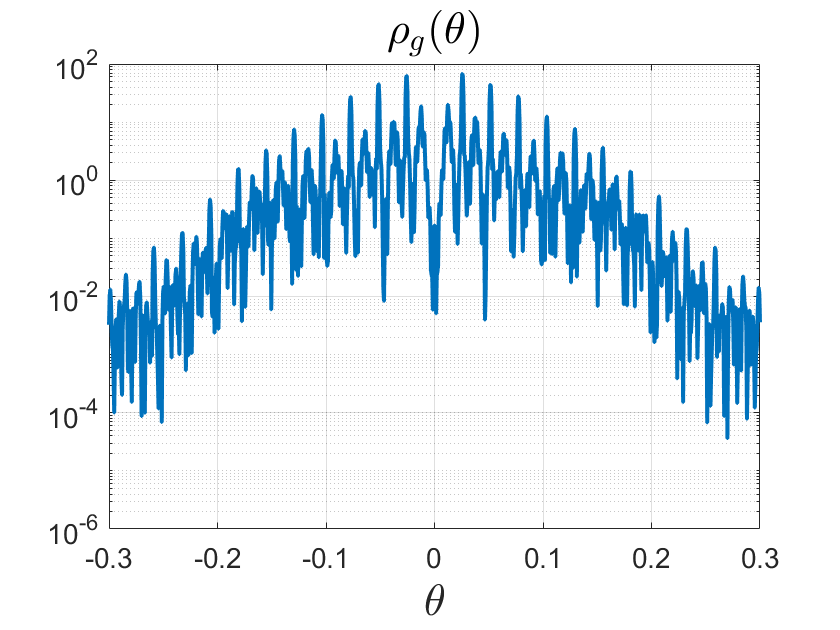}
\caption{\textit{Left:} A visualization of the so-called ``simple" Rössler attractor, which exhibits a type of non-mixing chaos known as sharp phase coherence: while the attractor is chaotic, some trajectories in the attractor are approximately periodic~\cite{farmer1980power}. This leads to sharp spikes in the power spectral density of the $z$-variable embedded in a continuous background of broader peaks. \textit{Right:} A smoothed power spectral density of the $z$-variable is plotted on a log scale, corresponding to a smoothed spectral measure of the Koopman operator computed via mpEDMD and a $6$th-order smoothing kernel. Notable features associated with sharp phase coherence are the sharp peaks at regular intervals and the approximate vanishing of the power spectral density at $\theta=0$.}
\label{fig:rossler_simple}
\end{figure}

\subsubsection{Example: Rössler system}

To illustrate the resolvent-based approach in practice, consider the Rössler system of differential equations:
$$
\dot x = -(y+z), \quad \dot y = x+ay, \quad \dot z = b + xz - cz,
$$
where $a,b,c\in\mathbb{R}$ are real parameters. When $a=0.15$, $b=0.4$, and $c=8.5$, the so-called ``simple" Rössler attractor (see~\cref{fig:rossler_simple}) exhibits a chaotic phenomenon known as sharp phase coherence. The attractor is chaotic, but some trajectories in the attractor remain approximately periodic~\cite{farmer1980power}, which produces a power spectral density consisting of sharp spikes embedded in a continuous background of broader, shorter peaks. When the $a$ parameter is increased to $a=0.3$, the attractor passes through a topological bifurcation to the ``funnel" Rössler attractor (see~\cref{fig:rossler_funnel}) and loses the sharp phase coherence. Consequently, the sharp spikes in the power spectral density broaden and blend into the continuous background. A smoothed power spectral density of the $z$-variable for the ``simple" Rössler attractor, corresponding to a Koopman spectral measure $\xi_g$ with observable $g(x,y,z)=z$, is plotted on a log scale in the right panel of~\cref{fig:rossler_simple}. The smoothed measure was computed via mpEDMD and a $6$th-order smoothing kernel. \Cref{fig:rossler_funnel} shows the power spectral density for the ``funnel" Rössler attractor.

\begin{figure}[t]
\centering
\includegraphics[width=0.45\linewidth]{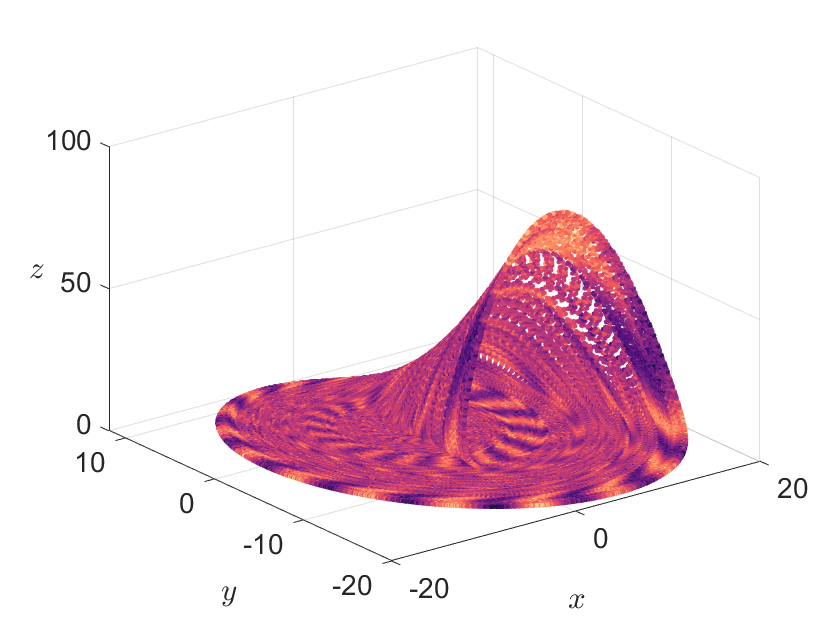}\hfill
\includegraphics[width=0.45\linewidth]{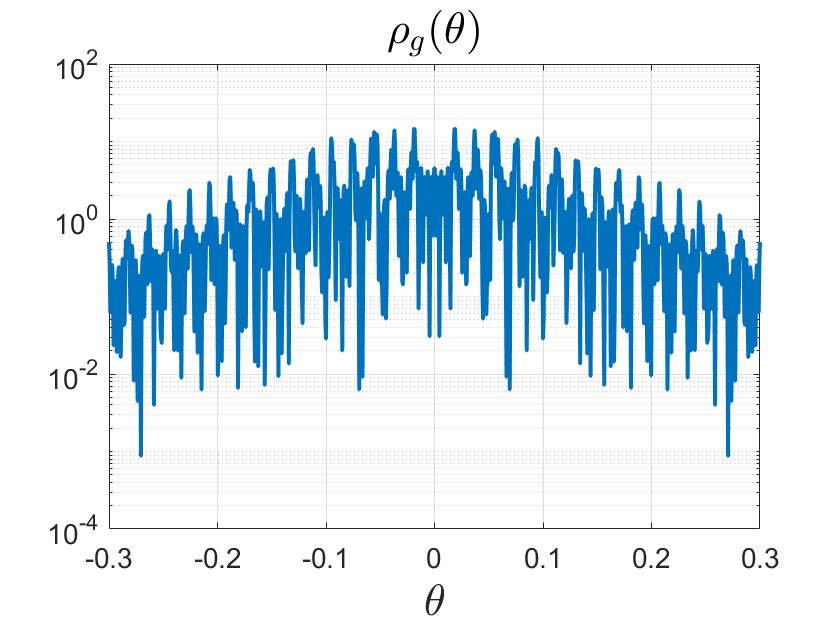}
\caption{\textit{Left:} A visualization of the so-called ``funnel" Rössler attractor, which loses phase coherence during a topological bifurcation in the Rössler attracter~\cite{farmer1980power}. \textit{Right:} A smoothed power spectral density of the $z$-variable is plotted on a log scale, corresponding to a smoothed spectral measure of the Koopman operator computed via mpEDMD and a $6$th-order smoothing kernel. In contrast to Figure~12, the spikes in the power spectral density of the $z$-variable have decreased an order of magnitude in amplitude and broadened, joining the continuous peaks previously in the background, indicating that the strong signals from the approximately periodic trajectories present in the ``simple" Rössler attractor have attenuated. The spectral density does not vanish at the origin.}
\label{fig:rossler_funnel}
\end{figure}

\subsection{Spectral analysis beyond measures}

Most of the algorithms and analysis in this section extend naturally to the action of the projection-valued measure $\mathcal{E}$ on observables in $L^2(\mathcal{X},\omega)$. This is useful for constructing approximate coherent modes and for building approximations of the Koopman operator and dynamics on restricted invariant subspaces associated with selected portions of the unit circle~\cite{korda2020data}.

Among data-driven methods for nonlinear dynamical systems, a distinctive feature of Koopman operator theory is the rich connection between Koopman eigenfunctions and the geometric structure of the state space. While unitary Koopman operators with continuous spectrum may lack traditional ``normalizable" modes in $L^2(\mathcal{X},\omega)$, one can nevertheless identify and numerically approximate a complete family of \textit{generalized eigenfunctions} by working within rigged Hilbert spaces~\cite{colbrook2024rigged}.

Another fascinating and active area of research concerns the \textit{resonances} of Koopman and transfer operators. These resonances typically appear as isolated points inside the unit disk, arising as poles of meromorphic extensions of the resolvent, and have been studied rigorously in the context of Fredholm--Riesz operators~\cite{bandtlow1997resonances}. They are closely connected to the decay rates of correlation functions in chaotic systems, known as Pollicott--Ruelle resonances~\cite{pollicott1985rate,ruelle1986resonances}. Recent work by Bandtlow, Just, and Slipantschuk~\cite{slipantschuk2020dynamic,bandtlow2023edmd}, as well as Wormell~\cite{wormell2023orthogonal}, has shown that certain EDMD eigenvalues may be associated to these resonances in the large-subspace limit.

\section{Towards a classification theory}

A careful reader will note that some of the algorithms we have discussed require taking several parameters successively to infinity, such as $M \to \infty$ (large data) followed by $N \to \infty$ (large subspace). These limits do not, in general, commute, and it may be impossible to reformulate the problem with fewer limits or design an alternative algorithm that avoids them. This is a generic feature of infinite-dimensional spectral problems~\cite{colbrook2020PhD,colbrookBook} and has given rise to the \textit{Solvability Complexity Index} (SCI)~\cite{ben2015can,colbrook4,colbrook3}.

From the SCI perspective, problems are classified according to the minimal number of successive limits needed for their resolution. For example, some spectral quantities of Koopman operators require two successive limits, while others provably belong to higher levels of the SCI hierarchy. Determining these classifications provides a rigorous framework for understanding which spectral computations are fundamentally more difficult than others.

We do not go into detail here, but there remain many open questions about the foundations of computing spectral properties of Koopman operators. In particular, establishing lower bounds on the number of successive limits required is an active area of research. Recent results have addressed this problem in $L^2$ spaces~\cite{colbrook2024limits}, $L^p$ spaces~\cite{sorg2025solvability}, and reproducing kernel Hilbert spaces of observables~\cite{boulle2025convergent}.

\bibliographystyle{plain}
\footnotesize
\renewcommand{\baselinestretch}{0.9}
\bibliography{biblios.bib}

\begin{thebibliography}{10}

\bibitem{arbabi2017study}
H.~Arbabi and I.~Mezi{\'c}.
\newblock Study of dynamics in post-transient flows using {K}oopman mode
  decomposition.
\newblock {\em Physical Review Fluids}, 2(12):124402, 2017.

\bibitem{arbabi2017ergodic}
Hassan Arbabi and Igor Mezi{\'c}.
\newblock Ergodic theory, dynamic mode decomposition, and computation of
  spectral properties of the {K}oopman operator.
\newblock {\em SIAM Journal on Applied Dynamical Systems}, 16(4):2096--2126,
  2017.

\bibitem{babuvska1991eigenvalue}
Ivo Babu{\v{s}}ka and John Osborn.
\newblock Eigenvalue problems.
\newblock In {\em Handbook of Numerical Analysis}, volume~2, pages 641--787.
  Elsevier, 1991.

\bibitem{bandtlow1997resonances}
Oscar~F. Bandtlow, Ioannis Antoniou, and Zdzislaw Suchanecki.
\newblock Resonances of dynamical systems and {F}redholm-{R}iesz operators on
  rigged {H}ilbert spaces.
\newblock {\em Computers \& Mathematics with Applications}, 34(2-4):95--102,
  1997.

\bibitem{bandtlow2023edmd}
Oscar~F. Bandtlow, Wolfram Just, and Julia Slipantschuk.
\newblock {EDMD} for expanding circle maps and their complex perturbations.
\newblock {\em arXiv preprint arXiv:2308.01467}, August 2023.

\bibitem{beer1993topologies}
Gerald Beer.
\newblock {\em {Topologies on Closed and Closed Convex Sets}}, volume 268 of
  {\em Mathematics and its Applications}.
\newblock Springer Netherlands, 1993.

\bibitem{ben2015can}
Jonathan Ben-Artzi, Matthew~J. Colbrook, Anders~C. Hansen, Olavi Nevanlinna,
  and Markus Seidel.
\newblock Computing spectra - {O}n the solvability complexity index hierarchy
  and towers of algorithms.
\newblock {\em arXiv preprint arXiv:1508.03280}, 2020.

\bibitem{boulle2025convergent}
Nicolas Boull{\'e}, Matthew~J. Colbrook, and Gustav Conradie.
\newblock Convergent methods for {K}oopman operators on reproducing kernel
  {H}ilbert spaces.
\newblock {\em arXiv preprint arXiv:2506.15782}, 2025.

\bibitem{brunton2016koopman}
Steven~L. Brunton, Bingni~W. Brunton, Joshua~L. Proctor, and J.~Nathan Kutz.
\newblock Koopman invariant subspaces and finite linear representations of
  nonlinear dynamical systems for control.
\newblock {\em PLOS ONE}, 11(2):e0150171, 2016.

\bibitem{brunton2021modern}
Steven~L. Brunton, Marko Budi{\v{s}}i{\'c}, Eurika Kaiser, and J.~Nathan Kutz.
\newblock Modern {K}oopman theory for dynamical systems.
\newblock {\em SIAM Review}, 64(2):229--340, 2022.

\bibitem{budivsic2012applied}
Marko Budi{\v{s}}i{\'c}, Ryan Mohr, and Igor Mezi{\'c}.
\newblock Applied {K}oopmanism.
\newblock {\em Chaos: An Interdisciplinary Journal of Nonlinear Science},
  22(4):047510, 2012.

\bibitem{colbrook2020PhD}
Matthew~J. Colbrook.
\newblock {\em The {F}oundations of {I}nfinite-dimensional {S}pectral
  {C}omputations}.
\newblock PhD thesis, University of Cambridge, 2020.

\bibitem{colbrook2019computing}
Matthew~J. Colbrook.
\newblock Computing spectral measures and spectral types.
\newblock {\em Communications in Mathematical Physics}, 384(1):433--501, 2021.

\bibitem{colbrook4}
Matthew~J. Colbrook.
\newblock On the computation of geometric features of spectra of linear
  operators on {H}ilbert spaces.
\newblock {\em Foundations of Computational Mathematics}, pages 1--82, 2022.

\bibitem{colbrook2023mpedmd}
Matthew~J. Colbrook.
\newblock The mp{EDMD} algorithm for data-driven computations of
  measure-preserving dynamical systems.
\newblock {\em SIAM Journal on Numerical Analysis}, 61(3):1585--1608, 2023.

\bibitem{colbrook2024another}
Matthew~J. Colbrook.
\newblock Another look at residual dynamic mode decomposition in the regime of
  fewer snapshots than dictionary size.
\newblock {\em Physica D: Nonlinear Phenomena}, 469:134341, 2024.

\bibitem{colbrook2023multiverse}
Matthew~J. Colbrook.
\newblock {\em The multiverse of dynamic mode decomposition algorithms},
  volume~25 of {\em Handbook of Numerical Analysis}, pages 127--230.
\newblock Elsevier, 2024.

\bibitem{colbrookBook}
Matthew~J. Colbrook.
\newblock {\em Infinite-Dimensional Spectral Computations}.
\newblock Cambridge University Press, to appear.

\bibitem{colbrook2023residualJFM}
Matthew~J. Colbrook, Lorna~J. Ayton, and M{\'a}t{\'e} Sz{\H{o}}ke.
\newblock Residual dynamic mode decomposition: {R}obust and verified
  {K}oopmanism.
\newblock {\em Journal of Fluid Mechanics}, 955:A21, 2023.

\bibitem{colbrook2024rigged}
Matthew~J. Colbrook, Catherine Drysdale, and Andrew Horning.
\newblock Rigged {D}ynamic {M}ode {D}ecomposition: {D}ata-driven generalized
  eigenfunction decompositions for {K}oopman operators.
\newblock {\em arXiv preprint arXiv:2405.00782}, 2024.

\bibitem{colbrook3}
Matthew~J. Colbrook and Anders~C. Hansen.
\newblock The foundations of spectral computations via the solvability
  complexity index hierarchy.
\newblock {\em Journal of the European Mathematical Society},
  25(12):4639--4728, 2022.

\bibitem{colbrook2021computing}
Matthew~J. Colbrook, Andrew Horning, and Alex Townsend.
\newblock Computing spectral measures of self-adjoint operators.
\newblock {\em SIAM review}, 63(3):489--524, 2021.

\bibitem{colbrook2025computing}
Matthew~J. Colbrook, Andrew Horning, and Tianyiwa Xie.
\newblock Computing generalized eigenfunctions in rigged {Hilbert} spaces.
\newblock {\em Pure and Applied Analysis}, 7(2):413--443, 2025.

\bibitem{colbrook2024limits}
Matthew~J. Colbrook, Igor Mezi{\'c}, and Alexei Stepanenko.
\newblock Limits and powers of {K}oopman learning.
\newblock {\em arXiv preprint arXiv:2407.06312}, 2024.

\bibitem{colbrook2019compute}
Matthew~J. Colbrook, Bogdan Roman, and Anders~C. Hansen.
\newblock How to compute spectra with error control.
\newblock {\em Physical Review Letters}, 122(25):250201, 2019.

\bibitem{colbrook2021rigorousKoop}
Matthew~J. Colbrook and Alex Townsend.
\newblock Rigorous data-driven computation of spectral properties of {K}oopman
  operators for dynamical systems.
\newblock {\em Communications on Pure and Applied Mathematics}, 77(1):221--283,
  2023.

\bibitem{davies2004spectral}
Brian~E. Davies and Michael Plum.
\newblock Spectral pollution.
\newblock {\em {IMA} Journal of Numerical Analysis}, 24(3):417--438, 2004.

\bibitem{Drmac-2020-koopman-book-chapter}
Zlatko Drma\v{c}.
\newblock Dynamic mode decomposition--a numerical linear algebra perspective.
\newblock In Alexandre Mauroy, Igor Mezi{\'{c}}, and Yoshihiko Susuki, editors,
  {\em The Koopman Operator in Systems and Control: Concepts, Methodologies,
  and Applications}, pages 161--194. Springer International Publishing, Cham,
  2020.

\bibitem{Drmac-Herm-DMD-TOMS-2024}
Zlatko Drma\v{c}.
\newblock Hermitian dynamic mode decomposition - numerical analysis and
  software solution.
\newblock {\em ACM Trans. Math. Softw.}, 50(1), March 2024.

\bibitem{Drmac-DMD-TOMS-2024}
Zlatko Drma\v{c}.
\newblock A {LAPACK} implementation of the {Dynamic Mode Decomposition}.
\newblock {\em ACM Trans. Math. Softw.}, 50(1), March 2024.

\bibitem{Drmac-Mezic-Koopman-Schur-2024}
Zlatko {Drma{\v{c}}} and Igor {Mezi{\'c}}.
\newblock A data driven {Koopman--Schur} decomposition for computational
  analysis of nonlinear dynamics.
\newblock {\em arXiv e-prints}, page arXiv:2312.15837, December 2023.

\bibitem{Drmac-Mezic-Mohr-EnhancedDMD-2018}
Zlatko Drma\v{c}, Igor Mezi\'{c}, and Ryan Mohr.
\newblock Data driven modal decompositions: Analysis and enhancements.
\newblock {\em SIAM Journal on Scientific Computing}, 40(4):A2253--A2285, 2018.

\bibitem{DMD-Vand-Cauchy-DFT}
Zlatko Drma\v{c}, Igor Mezi\'{c}, and Ryan Mohr.
\newblock Data driven {Koopman} spectral analysis in {Vandermonde--Cauchy} form
  via the {DFT}: Numerical method and theoretical insights.
\newblock {\em SIAM Journal on Scientific Computing}, 41(5):A3118--A3151, 2019.

\bibitem{Drmac-Mezic-Mohr-LS-Khatri-Rao-2020}
Zlatko Drma\v{c}, Igor Mezi\'{c}, and Ryan Mohr.
\newblock On least squares problems with certain {Vandermonde-Khatri-Rao}
  structure with applications to dmd.
\newblock {\em SIAM Journal on Scientific Computing}, 42(5):A3250--A3284, 2020.

\bibitem{Drmac-Mezic-Mohr-InfGen-2021}
Zlatko Drma\v{c}, Igor Mezi\'{c}, and Ryan Mohr.
\newblock Identification of nonlinear systems using the infinitesimal generator
  of the {Koopman} semigroup -- a numerical implementation of the
  {Mauroy--Goncalves} method.
\newblock {\em Mathematics}, 9(17), 2021.

\bibitem{dunford1958survey}
Nelson Dunford.
\newblock A survey of the theory of spectral operators.
\newblock {\em Bulletin of the American Mathematical Society}, 64(5):217--274,
  1958.

\bibitem{MR412888}
Nelson Dunford and Jacob~T. Schwartz.
\newblock {\em Linear operators. {P}art {III}: {S}pectral operators}, volume
  Vol. VII of {\em Pure and Applied Mathematics}.
\newblock Interscience Publishers [John Wiley \& Sons], New York-London-Sydney,
  1971.
\newblock With the assistance of William G. Bade and Robert G. Bartle.

\bibitem{eckmann1985ergodic}
Jean-Pierre Eckmann and David Ruelle.
\newblock Ergodic theory of chaos and strange attractors.
\newblock {\em Reviews of Modern Physics}, 57(3):617--656, 1985.

\bibitem{farmer1980power}
Doyne Farmer, James Crutchfield, Harold Froehling, Norman Packard, and Robert
  Shaw.
\newblock Power spectra and mixing properties of strange attractors.
\newblock {\em Annals of the New York Academy of Sciences}, 357(1):453--471,
  1980.

\bibitem{gottlieb1997gibbs}
David Gottlieb and Chi-Wang Shu.
\newblock On the {G}ibbs phenomenon and its resolution.
\newblock {\em SIAM review}, 39(4):644--668, 1997.

\bibitem{gragg1993positive}
William~B. Gragg.
\newblock Positive definite {T}oeplitz matrices, the {A}rnoldi process for
  isometric operators, and {G}aussian quadrature on the unit circle.
\newblock {\em Journal of Computational and Applied Mathematics},
  46(1-2):183--198, 1993.

\bibitem{Hansen_JAMS}
Anders~C. Hansen.
\newblock On the solvability complexity index, the {$n$}-pseudospectrum and
  approximations of spectra of operators.
\newblock {\em Journal of the American Mathematical Society}, 24(1):81--124,
  2011.

\bibitem{helsen2005convergence}
S.~Helsen, Arno B.~J. Kuijlaars, and Marc Van~Barel.
\newblock Convergence of the isometric {Arnoldi} process.
\newblock {\em SIAM Journal on Matrix Analysis and Applications},
  26(3):782--809, 2005.

\bibitem{SPDMD-Software}
Mihailo~R. Jovanovi{\'c}, Peter~J. Schmid, and Joseph~W. Nichols.
\newblock {DMDSP} -- {Sparsity-Promoting Dynamic Mode Decomposition} (software
  toolbox).
\newblock \url{https://www.ece.umn.edu/users/mihailo/software/dmdsp/}.
\newblock Accessed: 2024-03-20.

\bibitem{Jovanovic-Schmid-SPDMD:2014}
Mihailo~R. Jovanovi\'{c}, Peter~J. Schmid, and Joseph~W. Nichols.
\newblock {Sparsity--promoting dynamic mode decomposition}.
\newblock {\em Physics of Fluids}, 26(2):024103, 02 2014.

\bibitem{kaiser2017data}
Eurika Kaiser, Nathan~J. Kutz, and Steven~L. Brunton.
\newblock Data-driven discovery of {K}oopman eigenfunctions for control.
\newblock {\em Machine Learning: Science and Technology}, 2(3):035023, 2021.

\bibitem{katznelson2004introduction}
Yitzhak Katznelson.
\newblock {\em An introduction to harmonic analysis}.
\newblock Cambridge University Press, 2004.

\bibitem{koch2021large}
R{\'e}gis Koch, Marl{\`e}ne Sanjos{\'e}, and Stephane Moreau.
\newblock Large-eddy simulation of a linear compressor cascade with tip gap:
  {A}erodynamic and acoustic analysis.
\newblock In {\em AIAA AVIATION 2021 FORUM}, page 2312. American Institute of
  Aeronautics and Astronautics, 2021.

\bibitem{korda2020data}
Milan Korda, Mihai Putinar, and Igor Mezi{\'c}.
\newblock Data-driven spectral analysis of the {Koopman} operator.
\newblock {\em Applied and Computational Harmonic Analysis}, 48(2):599--629,
  2020.

\bibitem{lewin2009spectral}
Mathieu Lewin and {\'E}ric S{\'e}r{\'e}.
\newblock Spectral pollution and how to avoid it.
\newblock {\em Proceedings of the London Mathematical Society},
  100(3):864--900, 2010.

\bibitem{lorenz1963deterministic}
Edward~N. Lorenz.
\newblock Deterministic nonperiodic flow.
\newblock {\em Journal of the Atmospheric Sciences}, 20(2):130--141, 1963.

\bibitem{luzzatto2005lorenz}
Stefano Luzzatto, Ian Melbourne, and Frederic Paccaut.
\newblock The {L}orenz attractor is mixing.
\newblock {\em Communications in Mathematical Physics}, 260(2):393--401, 2005.

\bibitem{mezic2005spectral}
Igor Mezi{\'c}.
\newblock Spectral properties of dynamical systems, model reduction and
  decompositions.
\newblock {\em Nonlinear Dynamics}, 41(1):309--325, 2005.

\bibitem{mezic2013analysis}
Igor Mezi{\'c}.
\newblock Analysis of fluid flows via spectral properties of the {K}oopman
  operator.
\newblock {\em Annual Review of Fluid Mechanics}, 45(1):357--378, 2013.

\bibitem{mezic2020spectrum}
Igor Mezi{\'c}.
\newblock Spectrum of the {K}oopman operator, spectral expansions in functional
  spaces, and state-space geometry.
\newblock {\em Journal of Nonlinear Science}, 30(5):2091--2145, 2020.

\bibitem{mohr2014construction}
Ryan Mohr and Igor Mezi{\'c}.
\newblock Construction of eigenfunctions for scalar-type operators via laplace
  averages with connections to the koopman operator.
\newblock {\em arXiv preprint arXiv:1403.6559}, 2014.

\bibitem{osborn1975spectral}
John~E. Osborn.
\newblock Spectral approximation for compact operators.
\newblock {\em Mathematics of Computation}, 29(131):712--725, 1975.

\bibitem{otto2021koopman}
Samuel~E. Otto and Clarence~W. Rowley.
\newblock Koopman operators for estimation and control of dynamical systems.
\newblock {\em Annual Review of Control, Robotics, and Autonomous Systems},
  4(1):59--87, 2021.

\bibitem{peake2012modern}
Nigel Peake and Anthony~B. Parry.
\newblock Modern challenges facing turbomachinery aeroacoustics.
\newblock {\em Annual Review of Fluid Mechanics}, 44(1):227--248, 2012.

\bibitem{Pokrzywa_79}
Andrzej Pokrzywa.
\newblock Method of orthogonal projections and approximation of the spectrum of
  a bounded operator.
\newblock {\em Studia Mathematica}, 65(1):21--29, 1979.

\bibitem{pollicott1985rate}
Mark Pollicott.
\newblock On the rate of mixing of {A}xiom {A} flows.
\newblock {\em Inventiones mathematicae}, 81:413--426, 1985.

\bibitem{ruelle1986resonances}
David Ruelle.
\newblock Resonances of chaotic dynamical systems.
\newblock {\em Physical review letters}, 56(5):405, 1986.

\bibitem{Schmid-DMD-2010}
Peter~J. Schmid.
\newblock Dynamic mode decomposition of numerical and experimental data.
\newblock {\em Journal of Fluid Mechanics}, 656:5--28, 2010.

\bibitem{schmid2022dynamic}
Peter~J. Schmid.
\newblock Dynamic mode decomposition and its variants.
\newblock {\em Annual Review of Fluid Mechanics}, 54(1):225--254, 2022.

\bibitem{Schmid-Sesterhenn-DMD-2008}
Peter~J. Schmid and J\"{o}rn Sesterhenn.
\newblock {Dynamic Mode Decomposition} of numerical and experimental data.
\newblock {\em Bull. Amer. Phys. Soc., 61st APS meeting, San Antonio.}, page
  208, 2008.

\bibitem{schuster1897lunar}
Arthur Schuster.
\newblock On lunar and solar periodicities of earthquakes.
\newblock {\em Proceedings of the Royal Society of London},
  61(369-377):455--465, 1897.

\bibitem{Shapiro-Comp_ops-1993}
Joel~H. Shapiro.
\newblock {\em Composition Operator ad Classical Function Theory}.
\newblock Universitext: Tracts in Mathematics. Springer--Verlag, 1993.

\bibitem{SBarry-1}
Barry Simon.
\newblock {\em Orthogonal Polynomials on the Unit Circle. Part 1: Classical
  theory}, volume~54.
\newblock American Mathematical Society, 2005.

\bibitem{SBarry-2}
Barry Simon.
\newblock {\em Orthogonal Polynomials on the Unit Circle, Part 2: Spectral
  Theory}, volume~54.
\newblock American Mathematical Society, 2005.

\bibitem{simon2008christoffel}
Barry Simon.
\newblock The {C}hristoffel-{D}arboux kernel.
\newblock {\em arXiv preprint arXiv:0806.1528}, 2008.

\bibitem{simon2010szegHo}
Barry Simon.
\newblock Szeg{\H{o}}'s theorem and its descendants: spectral theory for l2
  perturbations of orthogonal polynomials.
\newblock In {\em Szeg{\H{o}}'s Theorem and Its Descendants}. Princeton
  university press, 2010.

\bibitem{Singh-Compact-QuasiN-1974}
Raj~Kishor Singh.
\newblock Compact and quasinormal composition operators.
\newblock {\em Proceedings of the American Mathematical Society}, 45(1):80--82,
  1974.

\bibitem{slipantschuk2020dynamic}
Julia Slipantschuk, Oscar~F. Bandtlow, and Wolfram Just.
\newblock Dynamic mode decomposition for analytic maps.
\newblock {\em Communications in Nonlinear Science and Numerical Simulation},
  84:105179, May 2020.

\bibitem{sorg2025solvability}
Christopher Sorg.
\newblock Solvability complexity index classification for {K}oopman operator
  spectra in {$L^{p}$} for {$1< p<\infty$}.
\newblock {\em arXiv preprint arXiv:2509.16016}, 2025.

\bibitem{trefethen2008gauss}
Lloyd~N Trefethen.
\newblock Is {Gauss} quadrature better than {C}lenshaw--{C}urtis?
\newblock {\em SIAM review}, 50(1):67--87, 2008.

\bibitem{trefethen2019approximation}
Lloyd~N. Trefethen.
\newblock {\em Approximation theory and approximation practice, extended
  edition}.
\newblock SIAM, 2019.

\bibitem{trefethen2005spectra}
Lloyd~N. Trefethen and Mark Embree.
\newblock {\em {Spectra and Pseudospectra: The Behavior of Nonnormal Matrices
  and Operators}}.
\newblock Princeton University Press, 2005.

\bibitem{trefethen1993hydrodynamic}
Lloyd~N. Trefethen, Anne~E. Trefethen, Satish~C. Reddy, and Tobin~A. Driscoll.
\newblock Hydrodynamic stability without eigenvalues.
\newblock {\em Science}, 261(5121):578--584, 1993.

\bibitem{tu-rowley-dmd-theory-appl-2014}
Jonathan~H. Tu, Clarence~W. Rowley, Dirk~M. Luchtenburg, Steven~L. Brunton, and
  J.~Nathan Kutz.
\newblock On dynamic mode decomposition: Theory and applications.
\newblock {\em Journal of Computational Dynamics}, 1(2):391--421, 2014.

\bibitem{valva2024physics}
Claire Valva and Dimitrios Giannakis.
\newblock Physics-informed spectral approximation of {K}oopman operators.
\newblock {\em arXiv preprint arXiv:2408.05663}, 2024.

\bibitem{vandeven1991family}
Herv{\'e} Vandeven.
\newblock Family of spectral filters for discontinuous problems.
\newblock {\em Journal of Scientific Computing}, 6(2):159--192, 1991.

\bibitem{weideman2007kink}
Andr\'{e}~J.A.C. Weideman and Lloyd~N. Trefethen.
\newblock The kink phenomenon in {F}ej{\'e}r and {C}lenshaw--{C}urtis
  quadrature.
\newblock {\em Numerische Mathematik}, 107(4):707--727, 2007.

\bibitem{wiener1930generalized}
Norbert Wiener.
\newblock Generalized harmonic analysis.
\newblock {\em Acta mathematica}, 55(1):117--258, 1930.

\bibitem{Williams2015}
Matthew~O. Williams, Ioannis~G. Kevrekidis, and Clarence~W. Rowley.
\newblock A data--driven approximation of the {Koopman} operator: Extending
  dynamic mode decomposition.
\newblock {\em Journal of Nonlinear Science}, 25(6):1307--1346, 2015.

\bibitem{wormell2023orthogonal}
Caroline~L. Wormell.
\newblock Orthogonal polynomial approximation and extended dynamic mode
  decomposition in chaos.
\newblock {\em arXiv preprint arXiv:2305.08074}, May 2023.

\bibitem{yosida2012functional}
K{\^o}saku Yosida.
\newblock {\em Functional analysis}, volume 123.
\newblock Springer Science \& Business Media, 2012.

\end{thebibliography}
\end{document}